\documentclass[11pt,french,english,a4paper]{smfart}
\usepackage[utf8]{inputenc}
\usepackage[T1]{fontenc}
\usepackage[french]{babel}
\usepackage{verbatim}

\usepackage{amsmath}
\usepackage{amssymb}
\usepackage{amsthm}
\usepackage{lmodern}
\usepackage{latexsym}
\usepackage{amsfonts}
\usepackage{mathrsfs}
\usepackage[all]{xy}
\usepackage{bm}
\usepackage{yfonts}

\usepackage[mathcal]{euscript}

\usepackage{hyperref}
\hypersetup{pdftoolbar=true, pdftitle={Proj_Ellip}, pdffitwindow=true, colorlinks=true, citecolor=green, filecolor=black, linkcolor=blue, urlcolor=blue, hypertexnames=false}

\numberwithin{equation}{section}

\usepackage[top=4cm, bottom=3.4cm, left=2.9cm , right=2.9cm]{geometry}

\newtheorem{theorem}{Theorem}[section]
\newtheorem{cor}[theorem]{Corollary}
\newtheorem{proposition}[theorem]{Proposition}
\newtheorem{lemme}[theorem]{Lemma}

\usepackage{enumerate}

\usepackage{fge}

\theoremstyle{definition}
\newtheorem{defi}[theorem]{Definition}
\newtheorem{remq}[theorem]{Remark}

\newtheorem{rmk}[theorem]{Remark}

\newcommand{\Mscr}{\mathscr{M}}
\newcommand{\modm}{\mathrm{ mod \ }}
\newcommand{\sumstar}{\sideset{}{^*}\sum}
\newcommand{\Tscr}{\mathscr{T}}
\newcommand{\Mcal}{\mathcal{M}}
\newcommand{\Cbf}{\mathbf{C}}
\newcommand{\ifm}{\mathrm{if}}
\newcommand{\Scal}{\mathcal{S}}
\newcommand{\Rbf}{\mathbf{R}}
\newcommand{\Fbf}{\mathbf{F}}
\newcommand{\Kl}{\mathrm{Kl}}
\newcommand{\Qlbar}{\overline{\mathbf{Q}}_\ell}
\newcommand{\Abfq}{\mathbf{A}^1_{\Fbf_q}}
\newcommand{\cond}{\mathrm{cond}}
\newcommand{\PGL}{\mathrm{PGL}}
\newcommand{\GL}{\mathrm{GL}}
\newcommand{\Fbfbar}{\overline{\mathbf{F}}}
\newcommand{\Pbfq}{\mathbf{P}^1_{\mathbf{F}_q}}
\newcommand{\Ccal}{\mathcal{C}}
\newcommand{\chibar}{\overline{\chi}}
\newcommand{\Brm}{\mathrm{B}}
\newcommand{\Trm}{\mathrm{T}}
\newcommand{\Nrm}{\mathrm{N}}
\newcommand{\Gbf}{\mathbf{G}}
\newcommand{\Fcal}{\mathcal{F}}
\newcommand{\Frm}{\mathrm{F}}
\newcommand{\etabar}{\overline{\eta}}
\newcommand{\Tr}{\mathrm{Tr}}
\newcommand{\rank}{\mathrm{rank}}
\newcommand{\Swan}{\mathrm{Swan}}
\newcommand{\Lcal}{\mathcal{L}}
\newcommand{\Klcal}{\mathcal{K}\ell}
\newcommand{\xbar}{\overline{x}}
\newcommand{\Kbf}{\mathbf{K}}
\newcommand{\FT}{\mathrm{FT}}
\newcommand{\Gcal}{\mathcal{G}}
\newcommand{\Fscr}{\mathscr{F}}

\newcommand{\Hrm}{\mathrm{H}}
\newcommand{\Urm}{\mathrm{U}}
\newcommand{\amf}{\textfrak{a}}
\newcommand{\Hbf}{\mathbf{H}}
\newcommand{\Bcal}{\mathcal{B}}
\newcommand{\Zbf}{\mathbf{Z}}
\newcommand{\Ecal}{\mathcal{E}}
\newcommand{\Qbf}{\mathbf{Q}}
\newcommand{\Ncal}{\mathcal{N}}
\newcommand{\Abf}{\mathbf{A}}
\newcommand{\omegabar}{\overline{\omega}}
\newcommand{\Acal}{\mathcal{A}}

\title{Simultaneous non-vanishing for Dirichlet $L$-functions}
\author{Raphaël Zacharias}
\date{\today}

\begin{document}
\maketitle

%%%%%%%%%%%%%%%%%%%%%%%%%%%%%%%%%%%%%%%%%%%%%%%%%%%%%%%%%%%%%%%%%%%%%%%%%%%%%%%%%%%%%%%%%%
%%%%%%%%%%%%%%%%%%%%%%%%%%%%%%%%%%%%%%%%%%%%%%%%%%%%%%%%%%%%%%%%%%%%%%%%%%%%%%%%%%%%%%%%%%%%%%%%%%%%%%%%%%%%%%%%%%%%%%%%%%%%%%%%%%%%%%%%%%%%%%%%%%%%%%%%%%%%%%%%%%%%%%%%%%%%%%%%%%%%
%%%%%%%%%%								ABSTRACT								%%%%%%%%%%
%%%%%%%%%%%%%%%%%%%%%%%%%%%%%%%%%%%%%%%%%%%%%%%%%%%%%%%%%%%%%%%%%%%%%%%%%%%%%%%%%%%%%%%%%%%%%%%%%%%%%%%%%%%%%%%%%%%%%%%%%%%%%%%%%%%%%%%%%%%%%%%%%%%%%%%%%%%%%%%%%%%%%%%%%%%%%%%%%%%%%%%%%%%%%%%%%%%%%%%%%%%%%%%%%%%%%%%%%%%%%%%%%%%%%%%%%%%%%%%%%%%%%%%%%%%%%%%%%%%%%%%%%%%%%%
\begin{abstract}
We extend the work of Fouvry, Kowalski and Michel on correlation between Hecke eigenvalues of modular forms and algebraic trace functions in order to establish an asymptotic formula for a generalized cubic moment of modular $L$-functions at the central point $s=\frac{1}{2}$ and for prime moduli $q$. As an application, we exploit our recent result on the mollification of the fourth moment of Dirichlet $L$-functions to derive that for any pair $(\omega_1,\omega_2)$ of multiplicative characters modulo $q$, there is a positive proportion of $\chi \ (\modm q)$ such that $L(\chi,\frac{1}{2}),L(\chi\omega_1,\frac{1}{2})$ and $L(\chi\omega_2,\frac{1}{2})$ are simultaneously not too small.
\end{abstract}
%%%%%%%%%%%%%%%%%%%%%%%%%%%%%%%%%%%%%%%%%%%%%%%%%%%%%%%%%%%%%%%%%%%%%%%%%%%%%%%%%%%%%%%%%%%%%%%%%%%%%%%%%%%%%%%%%%%%%%%%%%%%%%%%%%%%%%%%%%%%%%%%%%%%%%%%%%%%%%%%%%%%%%%%%%%%%%%%%%%%%%%%%%%%%%%%%%%%%%%%%%%%%%%%%%%%%%%%%%%%%%%%%%%%%%%%%%%%%%%%%%%%%%%%%%%%%%%%%%%%%%%%%%%%%%

\tableofcontents

%%%%%%%%%%%%%%%%%%%%%%%%%%%%%%%%%%%%%%%%%%%%%%%%%%%%%%%%%%%%%%%%%%%%%%%%%%%%%%%%%%%%%%%%%%%%%%%%%%%%%%%%%%%%%%%%%%%%%%%%%%%%%%%%%%%%%%%%%%%%%%%%%%%%%%%%%%%%%%%%%%%%%%%%%%%%%%%%%%%%%%%%%%%%%%%%%%%%%%%%%%%%%%%%%%%%%%%%%%%%%%%%%%%%%%%%%%%%%%%%%%%%%%%%%%%%%%%%%%%%%%%%%%%%%%%%%%%%%%%%%%%%%%%%%%%%%%%%%%%%%%%%%%%%%%%%%%%%%%%%%%%%%%%%%%%%%%%%%%%%%%%%%%%%%%%%%%%%%%%%
%%%%%%%%%%				INTRODUCTION AND STATEMENT OF RESULTS				%%%%%%%%%%
%%%%%%%%%%%%%%%%%%%%%%%%%%%%%%%%%%%%%%%%%%%%%%%%%%%%%%%%%%%%%%%%%%%%%%%%%%%%%%%%%%%%%%%%%%%%%%%%%%%%%%%%%%%%%%%%%%%%%%%%%%%%%%%%%%%%%%%%%%%%%%%%%%%%%%%%%%%%%%%%%%%%%%%%%%%%%%%%%%%%%%%%%%%%%%%%%%%%%%%%%%%%%%%%%%%%%%%%%%%%%%%%%%%%%%%%%%%%%%%%%%%%%%%%%%%%%%%%%%%%%%%%%%%%%%%%%%%%%%%%%%%%%%%%%%%%%%%%%%%%%%%%%%%%%%%%%%%%%%%%%%%%%%%%%%%%%%%%%%%%%%%%%%%%%%%%%%%%%%%%
\section{Introduction and statement of results}\label{Section1}
The zeros of automorphic $L$-functions on the critical line have received considerable attention these last years \cite{bump,critical,luo,analytic99,iwaniec2000}. In particular, at the central point $s=\frac{1}{2}$, an $L$-function is expected to vanish only for either a good reason or a trivial reason. For example, if $E$ is an elliptic curve defined over $\Qbf$ and $L(E,s)$ is its associated $L$-function, then according to the Birch and Swinnerton-Dyer conjecture, $L(E,\frac{1}{2})=0$ if and only if the group of $\Qbf$-points $E(\Qbf)$ has positive rank. A trivial reason is for instance when the sign of the functional equation is $-1$, which is the case if the $L$-function is attached to an odd Hecke-Maass form.

A typical approach in the study of non-vanishing problems is to consider a family of $L$-functions $\{L(\pi,\frac{1}{2})\}$ for $\pi$ varying in some finite set $\Acal$ and try to give a lower bound for the proportion of $\pi\in\Acal$ such that $L(\pi,\frac{1}{2})\neq 0$ as $|\Acal|\rightarrow\infty$. In \cite{Iw-S}, H. Iwaniec and P. Sarnak examined $L(\chi,s)$ at $s=\frac{1}{2}$ as $\chi$ ranges over all primitive Dirichlet characters modulo $q$. They proved that at least $\frac{1}{3}$ of the central values $L(\chi,\frac{1}{2})$ are not zero. This proportion has been slightly improved to $3.3411$ by H.M. Bui \cite{bui} and finally to $\frac{3}{8}$ by R. Khan and H.T. Ngo \cite{khan} with the restriction to prime moduli $q$.

In \cite{simultaneous}, P. Michel and J. Vanderkam considered simultaneous non-vanishing problems : given three distinct Dirichlet characters $\omega_1,\omega_2,\omega_3$ of fixed modulus $D_1,D_2,D_3$ (satisfying some technical conditions), they proved that a positive proportion of Holomorphic primitive Hecke cusp form $f$ of weight $2$, prime level $q$ and trivial nebentypus are such that the product $L(f\otimes\omega_1,\frac{1}{2})L(f\otimes\omega_2,\frac{1}{2})L(f\otimes\omega_3,\frac{1}{2})$ is not zero for sufficiently large $q$ (in terms of $D_1,D_2,D_3$). They derived various arithmetic applications, especially the existence of quotients of $J_0(q)$ of large dimension satisfying the Birch and Swinnerton-Dyer conjecture over cyclic number fields of degree less than 5.

\vspace{0.2cm}

In this paper, we let $q>2$ be a prime number, $\omega_1,\omega_2$ be arbitrary Dirichlet characters of modulus $q$, $f$ a Hecke eigenform for $\mathrm{SL}_2(\Zbf)$ (holomorphic or Maass) and we are interested in the distribution of the values of the two families indexed by Dirichlet characters modulo $q$
$$\left\{ L\left(\chi,\frac{1}{2}\right)L\left(\chi\omega_1,\frac{1}{2}\right)L\left(\chi\omega_2,\frac{1}{2}\right) \ : \ \chi \ (\modm q)\right\},$$
$$\left\{ L\left(f\otimes\chi,\frac{1}{2}\right)L\left(\chi,\frac{1}{2}\right) \ : \ \chi \ (\modm q) \right\},$$
as $q\rightarrow \infty$. Using mollification method, a technique that has made the success of many of the papers cited in the previous paragraphs, we show that for both families, a positive proportion of these central values of character twists is not zero. We give in fact a more precise statement, saying that they are simultaneously not too small. We point out that there is no additional difficulty by considering fixed $\omega_1,\omega_2$ of conductors $D_1,D_2<q$. These simultaneous non-vanish results require the evaluation of the two cubic moments
\begin{equation}\label{DefinitionMomentDirichlet}
\Tscr^3(\omega_1,\omega_2,\ell;q):= \frac{1}{q-1}\sum_{\substack{\chi\ (\modm q) \\ \chi\neq 1,\omegabar_1,\omegabar_2}}L\left(\chi,\frac{1}{2}\right)L\left(\chi\omega_1,\frac{1}{2}\right)L\left(\chi\omega_2,\frac{1}{2}\right)\chi(\ell),
\end{equation}
\begin{equation}\label{DefinitionMomentCusp}
\Tscr^3(f,\ell;q):= \frac{1}{q-1}\sum_{\substack{\chi \ (\modm q) \\ \chi\neq 1}}L\left(f\otimes \chi,\frac{1}{2}\right)L\left(\chi,\frac{1}{2}\right)\chi(\ell),
\end{equation}
for $\ell\geqslant 1$ an integer coprime with $q$.

\vspace{0.1cm}

We now state the main results of this paper :

\begin{theorem}\label{TheoremPositive} Let $q>2$ be a prime number, $\omega_1,\omega_2$ two Dirichlet characters of modulus $q$ and $\varepsilon>0$. Then there exists an explicit absolute constant $c_1>0$ and $Q=Q(\varepsilon)>2$ such that for any prime $q\geqslant Q$, 
$$\left|\left\{ \chi \ (\modm q) \ : \ \left|L\left(\chi,\frac{1}{2}\right)\right|,\left|L\left(\chi\omega_1,\frac{1}{2}\right)\right|,\left|L\left(\chi\omega_2,\frac{1}{2}\right)\right|\geqslant \frac{1}{\log q}\right\}\right|\geqslant (c_1-\varepsilon)(q-1).$$
\end{theorem}

\begin{theorem}\label{TheoremPositive2} Let $q>2$ be a prime number, $f$ a Hecke cusp of level $1$ and spectral parameter $t_f$ satisfying the Ramanujan-Petersson conjecture and $\varepsilon>0$. Then there exists an explicit absolute constant $c_2>0$ and $Q=Q(\varepsilon,t_f)>2$ such that for any prime $q\geqslant Q$,
$$\left|\left\{ \chi \ (\modm q) \ : \  \left|L\left(f\otimes\chi,\frac{1}{2}\right)\right|\geqslant \frac{1}{\log(q)^2},\left|L\left(\chi,\frac{1}{2}\right)\right|\geqslant \frac{1}{\log q}\right\}\right|\geqslant (c_2-\varepsilon)(q-1).$$
\end{theorem}

The new main ingredient in the proof of Theorems \ref{TheoremPositive},\ref{TheoremPositive2} is the following result which establishes an asymptotic formula for the moments \eqref{DefinitionMomentDirichlet},\eqref{DefinitionMomentCusp}.
%%%%%%%%%%%%%%%%%%%%%%%%%%%%%%%%%%%%%%%%%%%%%%%%%%%%%%%%%%%%%%%%%%%%%%%%%%%%%%%%%%%%%%%%%%
%%%%%%%%%%							THEOREM 1								%%%%%%%%%%
%%%%%%%%%%%%%%%%%%%%%%%%%%%%%%%%%%%%%%%%%%%%%%%%%%%%%%%%%%%%%%%%%%%%%%%%%%%%%%%%%%%%%%%%%%
\begin{theorem}\label{Theorem1}Let $q>2$ be a prime number, $\omega_1,\omega_2$ be Dirichlet characters of modulus $q$, $f$ a primitive Hecke cusp form of level $1$ or $q$ and trivial nebentypus. Assume that $f$ satisfies the Ramanujan-Petersson conjecture, then for any $\varepsilon>0$, we have
\begin{equation}\label{AsymptoticFormulaTheorem1}
\Tscr^3(\omega_1,\omega_2,\ell;q) = \delta_{\ell=1}+O\left(q^{-\frac{1}{64}+\varepsilon}\right),
\end{equation}
\begin{equation}\label{AsymptoticForumalaTheorem1_2}
\Tscr^3(f,\ell;q) = \delta_{\ell=1} + O\left(q^{-\frac{1}{52}+\varepsilon}\right),
\end{equation}
where the implied constant only depends on $\varepsilon>0$ and polynomially on the Archimedean parameters of $f$ (the weight or the Laplace eigenvalue) in \eqref{AsymptoticForumalaTheorem1_2}.
\end{theorem}
%%%%%%%%%%%%%%%%%%%%%%%%%%%%%%%%%%%%%%%%%%%%%%%%%%%%%%%%%%%%%%%%%%%%%%%%%%%%%%%%%%%%%%%%%%

\begin{remq} The case where $f$ is of level one and $\ell=1$ has been announced by S. Das and S. Ganguly and it seems that their method is similar to our.

When $f$ is of level $q$ with non trivial central character, the proof of Theorem \ref{Theorem1} is similar but requires a mild extension of \cite[Theorem 1.3]{sawin} for Kloosterman sums twisted by characters. We will return to this question in a coming paper.
\end{remq}

\begin{remq} The asymptotic formula \eqref{AsymptoticForumalaTheorem1_2} is similar to the mixted cubic moment evaluated by S. Das and R. Khan in \cite{das}
$$\frac{1}{q-1}\sum_{\substack{\chi\ (\modm q) \\ \chi\neq 1}}L\left(f\otimes\chi,\frac{1}{2}\right)\overline{L\left(\chi,\frac{1}{2}\right)}.$$
As the authors explained, the complex conjugation above $L(\chi,1/2)$ was introduced to avoid difficulties connected to the oscillations of Gauss sums. What we show here is that this difficulties are resolved using variants of the methods of \cite{twists} \cite{sawin}.
\end{remq}

\subsection{The mollification method} We show here how to derive Theorems \ref{TheoremPositive} and \ref{TheoremPositive2} from Theorem \ref{Theorem1}. Let $1<L<q$ be a real number such that $\log L\asymp \log q$. For any multiplicative character $\chi \ (\modm q)$, we define the short linear form
\begin{equation}\label{DefinitionMollifier1}
\Mcal(\chi;L):= \sum_{\ell\leqslant L}\frac{\chi(\ell)\mu(\ell)}{\ell^{1/2}}\left(\frac{\log L/\ell}{\log L}\right)^2,
\end{equation}
\begin{comment}
and 
\begin{equation}\label{DefinitionMollifier2}
\Mcal(f,\chi;L):=\sum_{\ell\leqslant L}\frac{\chi(\ell)\mu_f(\ell)}{\ell^{1/2}}P_2\left(\frac{\log L/\ell}{\log L}\right)^2,
\end{equation}
\end{comment}
where $\mu$ is the Möbius function. Let $\{\lambda_f(n)\}_{n\geqslant 1}$ denotes the sequence of Hecke eigenvalues of a Hecke cusp form of level one and $\mu_f(n)$ be the convolution inverse of $\lambda_f(n)$ given by 
$$L(f,s)^{-1}=\prod_p\left(1-\frac{\lambda_f(p)}{p^s}+\frac{1}{p^{2s}}\right)=\sum_{n=1}^\infty\frac{\mu_f(n)}{n^s} \ , \ \Re e (s)>1.$$
For $1<L'<q$ with $\log L'\asymp \log q$, we also define
\begin{equation}\label{DefinitionMollifier2}
\Mcal(f\otimes\chi;L'):=\sum_{\ell\leqslant L'}\frac{\chi(\ell)\mu_f(\ell)}{\ell^{1/2}}\left(\frac{\log L'/\ell}{\log L'}\right).
\end{equation}
We finally consider the two mollified cubic moments
\begin{equation}\label{DefinitionMollifiedMoment2}
\Mscr^3(\omega_1,\omega_2;q):= \frac{1}{q-1}\sum_{\substack{\chi \ (\modm q) \\ \chi\neq 1,\omegabar_1,\omegabar_2}}\prod_{i=0}^2L\left(\chi\omega_i,\frac{1}{2}\right)\Mcal(\chi\omega_i;L),
\end{equation}
where $\omega_0$ is the trivial character, and 
\begin{equation}\label{DefinitionMollifiedMomentf}
\Mscr^3(f;q) := \frac{1}{q-1}\sum_{\substack{\chi \ (\modm q) \\ \chi\neq 1}}L\left(f\otimes\chi,\frac{1}{2}\right)\Mcal(f\otimes\chi;L')L\left(\chi,\frac{1}{2}\right)\Mcal(\chi;L).
\end{equation}
Note that \eqref{DefinitionMollifiedMoment2} and \eqref{DefinitionMollifiedMomentf} can be written in the form
$$\Mscr^3(\omega_1,\omega_2;q)=\sum_{\ell_1,\ell_2,\ell_3\leqslant L}\frac{\bm{x}(\ell_1)\bm{x}(\ell_2)\bm{x}(\ell_3)}{(\ell_1\ell_2\ell_3)^{1/2}}\Tscr^3(\omega_1,\omega_2,\ell_1\ell_2\ell_3;q),$$
$$\Mscr^3(f;q)=\sum_{\ell\leqslant L,\ell'\leqslant L'}\frac{\bm{x}_f(\ell)\bm{x}(\ell')}{(\ell\ell')^{1/2}}\Tscr^3(f,\ell\ell';q),$$
with
$$\bm{x}(\ell):= \mu(\ell)\left(\frac{\log L/\ell}{\log L}\right)^2 \ \ \mathrm{and} \ \ \bm{x}_f(\ell'):= \mu_f(\ell')\left(\frac{\log L/\ell'}{\log L}\right).$$
Since $f$ satisfies the Ramanujan-Petersson conjecture, we have for $1\leqslant \ell\leqslant L$ and $1\leqslant \ell'\leqslant L'$,
$$|\bm{x}(\ell)|\leqslant 1 \ \ \ \mathrm{and} \ \ \ |\bm{x}_f(\ell')|\leqslant \tau(\ell'),$$
where $\tau(n)=\sum_{d|n}1$ is the divisor function. Hence an immediate consequence of Theorem \ref{Theorem1} is the following corollary :
%%%%%%%%%%%%%%%%%%%%%%%%%%%%%%%%%%%%%%%%%%%%%%%%%%%%%%%%%%%%%%%%%%%%%%%%%%%%%%%%%%%%%%%%%%
%%%%%%%%%%								COROLLARY 1							%%%%%%%%%%
%%%%%%%%%%%%%%%%%%%%%%%%%%%%%%%%%%%%%%%%%%%%%%%%%%%%%%%%%%%%%%%%%%%%%%%%%%%%%%%%%%%%%%%%%%
\begin{cor}\label{Corollary1} For any $\varepsilon >0$, the mollified cubic moments \eqref{DefinitionMollifiedMoment2} and \eqref{DefinitionMollifiedMomentf} satisfies
$$
\Mscr^3(\omega_1,\omega_2;q)=1+O\left(L^{3/2}q^{-\frac{1}{64}+\varepsilon}\right) \ , \ \Mscr^3(f;q)=1+O\left((L')^{1/2}L^{1/2}q^{-\frac{1}{52}+\varepsilon}\right),
$$
where the implied constant only depends on $\varepsilon>0$ and polynomially on $t_f$ in the second expression.
\end{cor}
In \cite[Theorem 1.2]{zacharias}, we established an asymptotic formula for a mollified fourth moment of Dirichlet $L$-functions : for $L=q^{\lambda}$ with $0<\lambda<\frac{11}{8064}$, we obtained
\begin{equation}\label{FourthMoment}
\Mscr^4(q):= \frac{1}{q-1}\sum_{\substack{\chi \ (\modm q) \\ \chi\neq 1}}\left|L\left(\chi,\frac{1}{2}\right)\Mcal(\chi;L)\right|^4= P(\lambda^{-1})+o_\lambda(1),
\end{equation}
where $P(X)\in\Rbf[X]$ is a degree four polynomial with calculable coefficients. 

Similarly, in a paper of preparation \og \textit{Non-Vanishing of twisted L-function} \fg \ by Blomer, Fouvry, Kowalski, Michel, Mili\'cevi\'c and Sawin, they obtained
\begin{equation}\label{FourthMomentf}
\Mscr^4(f;q):= \frac{1}{q-1}\sum_{\substack{\chi \ (\modm q) \\ \chi\neq 1}}\left|L\left(f\otimes\chi,\frac{1}{2}\right)\Mcal(f\otimes\chi;L')\right|^2 = \frac{1}{1+\frac{1}{\lambda'}}+o_{\lambda',t_f}(1),
\end{equation}
for $L'=q^{\lambda'}$ with $0<\lambda'<\frac{1}{360}.$ Hence, combining Corollary \ref{Corollary1} with \eqref{FourthMoment} and \eqref{FourthMomentf} yields
%%%%%%%%%%%%%%%%%%%%%%%%%%%%%%%%%%%%%%%%%%%%%%%%%%%%%%%%%%%%%%%%%%%%%%%%%%%%%%%%%%%%%%%%%%
%%%%%%%%%%								COROLLARY 2							%%%%%%%%%%
%%%%%%%%%%%%%%%%%%%%%%%%%%%%%%%%%%%%%%%%%%%%%%%%%%%%%%%%%%%%%%%%%%%%%%%%%%%%%%%%%%%%%%%%%%

%%%%%%%%%%%%%%%%%%%%%%%%%%%%%%%%%%%%%%%%%%%%%%%%%%%%%%%%%%%%%%%%%%%%%%%%%%%%%%%%%%%%%%%%%
%%%%%%%%%%%%%%%%%%%%%%%%%%%%%%%%%%%%%%%%%%%%%%%%%%%%%%%%%%%%%%%%%%%%%%%%%%%%%%%%%%%%%%%%%
%%%%%%%%%%						PROOF OF THEOREM 1 & 2						%%%%%%%%%%
%%%%%%%%%%%%%%%%%%%%%%%%%%%%%%%%%%%%%%%%%%%%%%%%%%%%%%%%%%%%%%%%%%%%%%%%%%%%%%%%%%%%%%%%%
%%%%%%%%%%%%%%%%%%%%%%%%%%%%%%%%%%%%%%%%%%%%%%%%%%%%%%%%%%%%%%%%%%%%%%%%%%%%%%%%%%%%%%%%%

\begin{proof}[Proof of Theorems \ref{TheoremPositive} and \ref{TheoremPositive2}]
$\mathbf{The \ Dirichlet \ characters \ case \ : \ }$We first present the proof of Theorem \ref{TheoremPositive}. For any $\chi \ (\modm q)$, we define the characteristic function
$$\mathbf{1}(\chi) :=\delta_{|L(\chi,1/2)|\geqslant \frac{1}{\log q}}\delta_{|L(\chi\omega_1,1/2)|\geqslant \frac{1}{\log q}}\delta_{|L(\chi\omega_2,1/2)|\geqslant \frac{1}{\log q}}.$$
By Cauchy-Schwarz inequality, we infer
\begin{alignat*}{1}
&\left|\frac{1}{q-1}\sum_{\substack{\chi \ (\modm q) \\ \chi\neq 1,\omegabar_1,\omegabar_2}}\mathbf{1}(\chi)\prod_{i=0}^2 L\left(\chi\omega_i,\frac{1}{2}\right)\Mcal(\chi\omega_i;L)\right| \leqslant
\left(\frac{1}{q-1}\sum_{\chi\neq 1}\mathbf{1}(\chi)\left|L\left(\chi,\frac{1}{2}\right)\Mcal(\chi;L)\right|^2\right)^{1/2} \\ & \ \times  \left(\frac{1}{q-1}\sum_{\chi\neq \omegabar_1,\omegabar_2}\left|L\left(\chi\omega_1,\frac{1}{2}\right)\Mcal(\chi\omega_1;L)\right|^2\left|L\left(\chi\omega_2,\frac{1}{2}\right)\Mcal(\chi\omega_2;L)\right|^2\right)^{1/2} \\ 
& \leqslant \left(\frac{1}{q-1}\sum_{\chi \ (\modm q)}\mathbf{1}(\chi)\right)^{1/4}\left(\frac{1}{q-1}\sumstar_{\chi \ (\modm q)}\left|L\left(\chi,\frac{1}{2}\right)\mathcal{M}(\chi;L)\right|^4\right)^{3/4},
\end{alignat*}
where the symbol $^*$ means that we sum over non trivial characters modulo $q$. On the other hand, we have the lower bound for the left handside in the first line
$$\left|\frac{1}{q-1}\sum_{\substack{\chi \ (\modm q) \\ \chi\neq 1,\omegabar_1,\omegabar_2}}\mathbf{1}(\chi)\prod_{i=0}^2 L\left(\chi\omega_i,\frac{1}{2}\right)\Mcal(\chi\omega_i;L)\right|\geqslant \left|\Mscr^3(\omega_1,\omega_2;q)\right|-\mathscr{D},$$
where $\Mscr^3(\omega_1,\omega_2;q)$ is defined in \eqref{DefinitionMollifiedMoment2} and
$$\mathscr{D} :=\frac{1}{q-1}\sum_{\substack{\chi\neq 1,\omegabar_1,\omegabar_2 \\ \mathbf{1}(\chi)=0}}\left|\prod_{i=0}^2 L\left(\chi\omega_i,\frac{1}{2}\right)\Mcal(\chi\omega_i;L)\right|.$$
To estimate $\mathscr{D}$, note that the condition $\mathbf{1}(\chi)=0$ means that one of the central values is less than $\log(q)^{-1}$. Therefore, if for $i=0,1,2$, $\mathscr{D}_i$ is the subsum of $\mathscr{D}$ restricted to $\chi$ such that $|L(\chi\omega_i,\frac{1}{2})|\leqslant \log(q)^{-1}$, we obtain, by positivity, $\mathscr{D}\leqslant \mathscr{D}_0+\mathscr{D}_1+\mathscr{D}_2$ with for each $i=0,1,2$,
\begin{alignat*}{1}
\mathscr{D}_i & \ \leqslant \frac{1}{\log(q)}\left(\frac{1}{q-1}\sum_{\chi \ (\modm q)}\left|\Mcal(\chi;L)\right|^2\right)^{1/2}\left(\frac{1}{q-1}\sumstar_{\chi \ (\modm q)}\left|L\left(\chi,\frac{1}{2}\right)\mathcal{M}(\chi;L)\right|^4\right)^{1/2} \\ & \ \ll_\lambda  \ \frac{1}{\log (q)}\left(\frac{1}{q-1}\sum_{\chi \ (\modm q)}\left|\Mcal(\chi;L)\right|^2\right)^{1/2},
\end{alignat*}
using again twice Cauchy-Schwarz inequality and \eqref{FourthMoment} (recall that $L=q^\lambda$). Moreover, opening the square in $|\Mcal(\chi;L)|^2$ and applying the orthogonality relation yields
$$\frac{1}{q-1}\sum_{\chi \ (\modm q)}|\Mcal(\chi;L)|^2 \leqslant \sum_{\substack{\ell\equiv \ell' \ (\modm q) \\ \ell,\ell'\leqslant L}}\frac{|\bm{x}(\ell)\bm{x}(\ell')|}{(\ell\ell')^{1/2}}\leqslant \sum_{\ell\leqslant L}\frac{1}{\ell}\ll\log L,$$
since $L<q$.
Hence, assuming $L=q^\lambda$ with $0<\lambda<\frac{11}{8064}$, we get
$$\frac{1}{q-1}\sum_{\chi \ (\modm q)}\mathbf{1}(\chi)\geqslant \frac{\left|\Mscr^3(\omega_1,\omega_2;q)\right|^4}{\Mscr^4(q)^{3}}+O_\lambda\left(\frac{1}{\log(q)^{1/2}}\right)= \frac{1}{P(\lambda^{-1})^3}+o_\lambda(1).$$
If $$c_1:=\max_{0<\lambda\leqslant \frac{11}{8064}}P(\lambda^{-1})^{-3},$$ then for any $\varepsilon>0$, there exists $0<\tilde{\lambda}<\frac{11}{8064}$ depending on $\varepsilon$ satisfying $|P(\tilde{\lambda}^{-1})^{-3}-c_1|\leqslant\varepsilon/2$. Finally, choosing $Q=Q(\varepsilon)$ large enough such that $|o_{\tilde{\lambda}}(1)|\leqslant\varepsilon/2$ for $q\geqslant Q$ and the result follows.

\vspace{0.2cm}

\noindent $\mathbf{The \ cuspidal \ case \ : \ }$ We proceed in a similar way. Setting
$$\mathbf{1}(\chi,f):= \delta_{|L(\chi,1/2)|\geqslant\frac{1}{\log q}}\delta_{|L(f\otimes\chi,1/2)|\geqslant \frac{1}{\log^2 q}},$$
we obtain
\begin{alignat*}{1}
\left|\frac{1}{q-1}\sum_{\chi\neq 1}\mathbf{1}(\chi,f)L\left(f\otimes\chi,\frac{1}{2}\right)\Mcal(f\otimes\chi;L')L\left(\chi,\frac{1}{2}\right)\Mcal(\chi;L)\right| & \ \leqslant \left(\frac{1}{q-1}\sum_{\chi \ (\modm q)}\mathbf{1}(\chi,f)\right)^{1/4} \\  & \times \left(\Mscr^4(q)\right)^{1/4}\left(\Mscr^4(f;q)\right)^{1/2}
\end{alignat*}
where $\Mscr^4(q)$ (resp. $\Mscr^2(f;q)$) are defined by \eqref{FourthMoment} (resp. by \eqref{FourthMomentf}). As in the previous part, the left handside admits the lower bound
$$\geqslant \left|\Mscr^3(f;q)\right|-\mathscr{C},$$
where $\mathscr{C}$ is the same as $\Mscr^3(f;q)$, but with the absolute values inside and with the restriction in the summation to $\chi$ such that $\mathbf{1}(\chi,f)=0$. Writing $\mathscr{C}_1$ (resp. $\mathscr{C}_2$) for the contribution of $|L(\chi,1/2)|\leqslant\frac{1}{\log q}$ (resp. $|L(f\otimes\chi,1/2)|\leqslant \frac{1}{\log^2 q}$), we get $\mathscr{C}\leqslant \mathscr{C}_1+\mathscr{C}_2$ with 
$$\mathscr{C}_1 = O_\lambda\left(\frac{1}{\log (q)^{1/2}}\right).$$
Finally, we have
\begin{alignat*}{1}
\mathscr{C}_2 & \leqslant \frac{1}{\log^2(q)}\left(\frac{1}{q-1}\sum_{\chi \ (\modm q)}|\Mcal(f\otimes\chi;L')|^2\right)^{1/2}\left(\frac{1}{q-1}\sumstar_{\chi \ (\modm q)}\left|L\left(\chi,\frac{1}{2}\right)\Mcal(\chi;L)\right|^2\right)^{1/2} \\ 
  & \ll_\lambda \frac{1}{\log^2(q)}\left(\frac{1}{q-1}\sum_{\chi \ (\modm q)}|\Mcal(f\otimes\chi;L')|^2\right)^{1/2},
\end{alignat*}
with
$$\frac{1}{q-1}\sum_{\chi \ (\modm q)}|\Mcal(f\otimes\chi;L')|^2\leqslant \sum_{\substack{\ell\equiv \ell' \ (\modm q) \\ \ell,\ell'\leqslant L'}}\frac{\tau(\ell)\tau(\ell')}{(\ell\ell')^{1/2}} = \sum_{\ell\leqslant L'}\frac{\tau(\ell)^2}{\ell}\ll \log^3 L'.$$
Hence, 
$$\mathscr{C}_2=O_{\lambda'}\left(\frac{1}{\log(q)^{1/2}}\right),$$
and the rest of the proof is exactly the same as in the previous case.
\end{proof}

%%%%%%%%%%%%%%%%%%%%%%%%%%%%%%%%%%%%%%%%%%%%%%%%%%%%%%%%%%%%%%%%%%%%%%%%%%%%%%%%%%%%%%%%%
%%%%%%%%%%%%%%%%%%%%%%%%%%%%%%%%%%%%%%%%%%%%%%%%%%%%%%%%%%%%%%%%%%%%%%%%%%%%%%%%%%%%%%%%%
%%%%%%%%%%							END OF THE PROOF							%%%%%%%%%%
%%%%%%%%%%%%%%%%%%%%%%%%%%%%%%%%%%%%%%%%%%%%%%%%%%%%%%%%%%%%%%%%%%%%%%%%%%%%%%%%%%%%%%%%%
%%%%%%%%%%%%%%%%%%%%%%%%%%%%%%%%%%%%%%%%%%%%%%%%%%%%%%%%%%%%%%%%%%%%%%%%%%%%%%%%%%%%%%%%%

\begin{remq} Let $f$ be a primitive Hecke cusp form of prime level $q$ satisfying the Ramanujan-Petersson conjecture. The formula \eqref{AsymptoticForumalaTheorem1_2} could be used to prove simultaneous non-vanishing for $L(f\otimes\chi,\frac{1}{2})L(\chi,\frac{1}{2})$ as $\chi$ runs over non trivial Dirichlet characters modulo $q$ provided that it is possible to evaluate a second twisted moment of the form
$$\frac{1}{q-1}\sum_{\substack{\chi \ (\modm q) \\ \chi\neq 1}}\left|L\left(f\otimes\chi,\frac{1}{2}\right)\right|^2\chi(\ell_1)\chibar(\ell_2),$$
where $(\ell_1,\ell_2)=1$ and are coprime with $q$. An asymptotic formula for this moment is given in \cite{moments} in the special case where the level is $1$ and $\ell_1=\ell_2=1$ and for general $(\ell_1,\ell_2)=1$ in the paper of preparation mentioned above (also for level $1$). The principal difficulty here is that since the level is $q$, we have the solve a shifted convolution problem of the shape
$$\mathop{\sum\sum}_{\ell_1n-\ell_2m=hq}\lambda_f(n)\lambda_f(m)W_1\left(\frac{n}{N}\right)W_2\left(\frac{m}{M}\right),$$
for Hecke eigenvalues $\lambda_f(n)$ of level $q$.

\end{remq}

%%%%%%%%%%%%%%%%%%%%%%%%%%%%%%%%%%%%%%%%%%%%%%%%%%%%%%%%%%%%%%%%%%%%%%%%%%%%%%%%%%%%%%%%%%%%%%%%%%%%%%%%%%%%%%%%%%%%%%%%%%%%%%%%%%%%%%%%%%%%%%%%%%%%%%%%%%%%%%%%%%%%%%%%%%%%%%%%%%%%%%%%%%%%%%%%%%%%%%%%%%%%%%%%%%%%%%%%%%%%%%%%%%%%%%%%%%%%%%%%%%%%%%%%%%%%%%%%%%%%%%%%%%%%%%
%%%%%%%%%%					SKETCH OF THE PROOF OF THEOREM 1					  %%%%%%%%%%
%%%%%%%%%%%%%%%%%%%%%%%%%%%%%%%%%%%%%%%%%%%%%%%%%%%%%%%%%%%%%%%%%%%%%%%%%%%%%%%%%%%%%%%%%%%%%%%%%%%%%%%%%%%%%%%%%%%%%%%%%%%%%%%%%%%%%%%%%%%%%%%%%%%%%%%%%%%%%%%%%%%%%%%%%%%%%%%%%%%%%%%%%%%%%%%%%%%%%%%%%%%%%%%%%%%%%%%%%%%%%%%%%%%%%%%%%%%%%%%%%%%%%%%%%%%%%%%%%%%%%%%%%%%%%%
\subsection{Sketch of the proof of Theorem \ref{Theorem1}}\label{SectionSketch} After an application of the approximate functional equation to \eqref{DefinitionMomentCusp} and \eqref{DefinitionMomentDirichlet}, which expresses the central value of automorphic $L$-functions as a convergent series, and an average over the characters, we isolate a main term which appears only if $\ell=1$ (c.f. § \ref{SectionApplicationApproximateFunctionalEquation}-\ref{SectionMainTerm}).

\vspace{0.1cm}

The treatment of the error term passes by the analysis of sums of the shape
\begin{equation}\label{Shape1}
\Scal(\omega_1,\omega_2;q) = \frac{1}{(qN_0N_1N_2)^{1/2}}\mathop{\sum\sum\sum}_{\substack{n_0\sim N_0 , n_1\sim N_1 \\ n_2\sim N_2}}\omegabar_1(n_1)\omegabar_2(n_2)\Kl_3(n_0n_1n_2,\omega_1,\omega_2,1;q),
\end{equation}
\begin{equation}\label{Shape2}
\Ccal(f;q)=\frac{1}{(qMN)^{1/2}}\mathop{\sum\sum}_{n\sim N, m\sim M}\lambda_f(n)\Kl_3(nm;q),
\end{equation}
where $\Kl_3$ is the classical rank $3$ Kloosterman sum, $\Kl_3(\omega_1,\omega_2,1;q)$ is the twisted version as defined in \eqref{DefinitionKloosterman}, $\{\lambda_f(n)\}_{n\geqslant 1}$ are the Hecke eigenvalues of $f$ which satisfy $|\lambda_f(n)|\leqslant \tau(n)$ and $N_0,N_1,N_2,N,M$ are parameters such that 
$$1\leqslant N_i,N,M, \ \ N_0N_1N_2\leqslant q^{3/2+\varepsilon} \ \mathrm{and} \ MN\leqslant q^{3/2+\varepsilon}.$$
The ultimate goal is to obtain a bound of the form
$$\Scal(\omega_1,\omega_2;q),\Ccal(f;q) = O\left(q^{-\delta}\right),$$
for some absolute constant $\delta>0$. Using Poisson summation in the three variables in \eqref{Shape1}, or Voronoi formula in the $n$-variable in \eqref{Shape2} (followed by Poisson on $m$) allows us to get rid of the cases where the product of the variables is larger than $q$; namely in § \ref{Section3Poisson} and \ref{SectionVoronoi}, we prove
$$\Scal(\omega_1,\omega_2;q)\ll q^\varepsilon\left(\frac{q}{N_0N_1N_2}\right)^{1/2} \ \ \mathrm{and} \ \ \Ccal(f;q) \ll q^\varepsilon\left(\frac{q}{NM}\right)^{1/2}.$$
Combining these two estimates with the trivial bounds
$$\Scal(\omega_1,\omega_2;q)\ll \left(\frac{N_0N_1N_2}{q}\right)^{1/2} \ \ \mathrm{and} \ \ \Ccal(f;q) \ll q^\varepsilon\left(\frac{NM}{q}\right)^{1/2},$$
we can assume for the rest of this outline that 
$$N_0N_1N_2=NM=q.$$
We treat these sums differently according to the relative size of the various parameters. If $N_1\sim 1$ (say) and $M\sim 1$, we exploit the $n_0,n_2$-sum (resp. the $n$-sum) in \eqref{Shape1} (resp. in \eqref{Shape2}) and average trivially over the others. Grouping $n_0n_2$ into a long variable $n$ and we need to analyze roughly 
$$\sum_{n\sim q}\lambda_{\omegabar_2}(n,0)\Kl_3(nn_1,\omega_1,\omega_2;q) \ \ \mathrm{and} \ \ \sum_{n\sim q}\lambda_f(n)\Kl_3(nm;q),$$
where for any $t\in\Rbf$,
$$\lambda_{\omegabar_2}(n,it)=\sum_{n_0n_2=n}\omegabar_2(n_2)\left(\frac{n_2}{n_0}\right)^{it}.$$
In \cite{prime} and \cite{twists}, Fouvry, Kowalski and Michel studied these sums when $f$ is a fixed cusp form of level $1$ and $\lambda_{\omegabar_2}(n,it)$ is replaced by the generalized divisor function $d_{it}(n)=\sum_{ab=n}a^{it}b^{-it}.$ They treated the problem in a more general case than $\Kl_3$, namely when the $q$-periodic function arises as Frobenius trace function of a certain constructible $\Qlbar$-sheaves on the affine line $\Abfq$ for $\ell$ a prime number different from $q$.

\vspace{0.1cm}

Our second main result is a generalization of their results (see \cite[Theorem 1.2]{twists} and \cite[Theorem 1.15]{prime}) for modular forms with nebentypus $\omega$ (possibly trivial) whose level is the same as the \og level \fg{} of the trace function and also for the twisted divisor function $\lambda_{\omega}(n,it)$. The only restriction we made in comparison with their Theorems is that we avoid the case of exceptional trace functions (see Definition \ref{DefinitionExceptional}). More precisely, for $V$ a smooth and compactly supported function on $\Rbf_+^*$, we define
\begin{equation}\label{Coorelation1}
\Scal_V(f,K;q) :=\sum_{n\geqslant 1}\lambda_f(n)K(n)V\left(\frac{n}{q}\right),
\end{equation}
\begin{equation}\label{Correlation2}
\Scal_V(\omega,it,K;q) :=\sum_{n\geqslant 1}\lambda_\omega(n,it)K(n)V\left(\frac{n}{q}\right).
\end{equation}
We prove in Section \ref{ProofTwist} :
%%%%%%%%%%%%%%%%%%%%%%%%%%%%%%%%%%%%%%%%%%%%%%%%%%%%%%%%%%%%%%%%%%%%%%%%%%%%%%%%%%%%%%%%%%
%%%%%%%%%%%%%%%%%%%%%%%%%%%%%%%%%%%%%%%%%%%%%%%%%%%%%%%%%%%%%%%%%%%%%%%%%%%%%%%%%%%%%%%%%%
%%%%%%%%%%						SECOND MAIN THEOREM							%%%%%%%%%%
%%%%%%%%%%%%%%%%%%%%%%%%%%%%%%%%%%%%%%%%%%%%%%%%%%%%%%%%%%%%%%%%%%%%%%%%%%%%%%%%%%%%%%%%%%
%%%%%%%%%%%%%%%%%%%%%%%%%%%%%%%%%%%%%%%%%%%%%%%%%%%%%%%%%%%%%%%%%%%%%%%%%%%%%%%%%%%%%%%%%%
\begin{theorem}\label{Theorem2} Let $q>2$ be a prime number, $\omega$ a Dirichlet character of modulus $q$, $f$ a primitive Hecke cusp form of level $q$, nebentypus $\omega$ and spectral parameter $t_f$. Let $K$ be a non Fourier-exceptional isotypic trace function of conductor $\cond(K)$, as defined in \eqref{DefinitionTracesheaf},\eqref{DefinitionConductor},\eqref{DefinitionExceptional}, and $V$ a function satisfying $V(C,P,Q)$ $($see Definition \eqref{DefinitionVCQ}$)$. Then there exists absolute constants $s\geqslant 1$ and $A\geqslant 1$ such that 
\begin{alignat*}{1}
\Scal_V(f,K;q) & \ll_{C,\delta}(1+|t_f|)^A\cond(K)^s q^{1-\delta}(PQ)^{1/2}\left(P+Q\right)^{1/2},\\
\Scal_V(\omega,it,K;q) & \ll_{C,\delta} (1+|t|)^A\cond(K)^sq^{1-\delta}(PQ)^{1/2}\left(P+Q\right)^{1/2},
\end{alignat*}
for any $\delta<1/16$.
\end{theorem}
Therefore, Theorem \ref{Theorem2} provides the desired power saving (set $P=Q=1$) in the special case where one of the variable is very small in \eqref{Shape1} and $M\sim 1$ in \eqref{Shape2}.

\vspace{0.1cm}

Assume now that $N_0,N_1,N_2\geqslant q^\eta$ and $M\geqslant q^{\eta}$ for some small real number $\eta>0$. From now, we need to take care of the different nature of expressions \eqref{Shape1} and \eqref{Shape2}. Indeed, for \eqref{Shape1}, the fact of having three free variables allows us to factorize two of them (say $n_0n_2$) in such a way that $N_0N_2\geqslant q^{1/2+\eta}$. In this case, we can form a bilinear sum and use a general version of Poly\'a-Vinogradov (see Theorem \ref{TheoremPrime}) to obtain a power saving in the error term. The same method also works for \eqref{Shape2}, as long as $M\leqslant q^{1/2-\eta}$, because in this case $N\geqslant q^{1/2+\eta}$, or $M\geqslant q^{1/2+\eta}$. Hence, the critical range for the second sum, i.e. when Poly\'a-Vinogradov is useless, appears when $M\sim q^{1/2}$ and $N\sim q^{1/2}$ and here we apply the general result of Kowalski, Michel and Sawin concerning bilinear forms involving classical Kloosterman sums \cite[Theorem 1.3]{sawin}.

\subsection{Acknowledgment} This paper is part of my Phd thesis and I would like to thank my supervisor Philippe Michel for his valuable advice. I am also grateful to Satadal Ganguly for its careful reading and its reference to the paper \cite{das} of S. Das and R. Khan.

%%%%%%%%%%%%%%%%%%%%%%%%%%%%%%%%%%%%%%%%%%%%%%%%%%%%%%%%%%%%%%%%%%%%%%%%%%%%%%%%%%%%%%%%%%%%%%%%%%%%%%%%%%%%%%%%%%%%%%%%%%%%%%%%%%%%%%%%%%%%%%%%%%%%%%%%%%%%%%%%%%%%%%%%%%%%%%%%%%%%%%%%%%%%%%%%%%%%%%%%%%%%%%%%%%%%%%%%%%%%%%%%%%%%%%%%%%%%%%%%%%%%%%%%%%%%%%%%%%%%%%%%%%%%%%%%%%%%%%%%%%%%%%%%%%%%%%%%%%%%%%%%%%%%%%%%%%%%%%%%%%%%%%%%%%%%%%%%%%%%%%%%%%%%%%%%%%%%%%%%
%%%%%%%%%%				BACKGROUND ON AUTOMORPHIC FORMS						%%%%%%%%%%
%%%%%%%%%%%%%%%%%%%%%%%%%%%%%%%%%%%%%%%%%%%%%%%%%%%%%%%%%%%%%%%%%%%%%%%%%%%%%%%%%%%%%%%%%%%%%%%%%%%%%%%%%%%%%%%%%%%%%%%%%%%%%%%%%%%%%%%%%%%%%%%%%%%%%%%%%%%%%%%%%%%%%%%%%%%%%%%%%%%%%%%%%%%%%%%%%%%%%%%%%%%%%%%%%%%%%%%%%%%%%%%%%%%%%%%%%%%%%%%%%%%%%%%%%%%%%%%%%%%%%%%%%%%%%%%%%%%%%%%%%%%%%%%%%%%%%%%%%%%%%%%%%%%%%%%%%%%%%%%%%%%%%%%%%%%%%%%%%%%%%%%%%%%%%%%%%%%%%%%%
\section{Background on automorphic forms}
In this section, we briefly compile the main results from the theory of automorphic forms which we shall need in Section \ref{ProofTwist}. Among these are Hecke eigenbases, multiplicative properties of Hecke eigenvalues, the Kuznetsov trace formula and the spectral large sieve inequality.

%%%%%%%%%%%%%%%%%%%%%%%%%%%%%%%%%%%%%%%%%%%%%%%%%%%%%%%%%%%%%%%%%%%%%%%%%%%%%%%%%%%%%%%%%%%%%%%%%%%%%%%%%%%%%%%%%%%%%%%%%%%%%%%%%%%%%%%%%%%%%%%%%%%%%%%%%%%%%%%%%%%%%%%%%%%%%%%%%%%%%%%%%%%%%%%%%%%%%%%%%%%%%%%%%%%%%%%%%%%%%%%%%%%%%%%%%%%%%%%%%%%%%%%%%%%%%%%%%%%%%%%%%%%%%%
%%%%%%%%%%							HECKE EIGENBASES							%%%%%%%%%%
%%%%%%%%%%%%%%%%%%%%%%%%%%%%%%%%%%%%%%%%%%%%%%%%%%%%%%%%%%%%%%%%%%%%%%%%%%%%%%%%%%%%%%%%%%%%%%%%%%%%%%%%%%%%%%%%%%%%%%%%%%%%%%%%%%%%%%%%%%%%%%%%%%%%%%%%%%%%%%%%%%%%%%%%%%%%%%%%%%%%%%%%%%%%%%%%%%%%%%%%%%%%%%%%%%%%%%%%%%%%%%%%%%%%%%%%%%%%%%%%%%%%%%%%%%%%%%%%%%%%%%%%%%%%%%
\subsection{Hecke eigenbases}
Let $N \geqslant 1$ be an integer, $\omega$ a Dirichlet character of modulus $N$, $\kappa=\frac{1-\omega(-1)}{2}\in\{0,1\}$ and $k\geqslant 2$ satisfying $k\equiv \kappa$ (mod $2$). We denote by $\mathcal{S}_k(N,\omega)$, $\mathcal{L}^2(N,\omega)$ and $\mathcal{L}_0^2(N,\omega)\subset\mathcal{L}^2(N,\omega)$, respectively, the Hilbert spaces (with respect to the Petersson inner product) of holomorphic cusp forms of weight $k$, of Maass forms of weight $\kappa$, of Maass cusp forms of weight $\kappa$, with respect to the Hecke congruence group $\Gamma_0(N)$ and with nebentypus $\omega$. These spaces are endowed with an action of the commutative algebra $\mathbf{T}$ generated by the Hecke operators $\{ T_n \ | \ n\geqslant 1\}$. Among these, the $T_n$ with $(n,N)=1$ generate a subalgebra $\mathbf{T}^{(N)}$ of $\mathbf{T}$ made of normal operators. As an immediate consequence, the spaces $\mathcal{S}_k(N,\omega)$ and $\mathcal{L}_0^2(N,\omega)$ have an orthonormal basis made of eigenforms of $\mathbf{T}^{(N)}$. We denote this basis respectively by $\mathcal{B}_k(N,\omega)$ and $\mathcal{B}(N,\omega)$. 

\vspace{0.1cm}

The orthogonal complement of $\Lcal_0^2(N,\omega)$ in $\mathcal{L}^2(N,\omega)$ is the Eisenstein spectrum (plus a constant if the character is trivial) and it is denoted by $\mathcal{E}(N,\omega)$. The space $\mathcal{E}(N,\omega)$ is continuously spanned by the Eisenstein series $E_{\amf}(\cdot, 1/2+it)$ where $\amf$ runs over singular cusps (with respect to $\omega$) of $\Gamma_0(N)$.

%%%%%%%%%%%%%%%%%%%%%%%%%%%%%%%%%%%%%%%%%%%%%%%%%%%%%%%%%%%%%%%%%%%%%%%%%%%%%%%%%%%%%%%%%%%%%%%%%%%%%%%%%%%%%%%%%%%%%%%%%%%%%%%%%%%%%%%%%%%%%%%%%%%%%%%%%%%%%%%%%%%%%%%%%%%%%%%%%%%%
%%%%%%%%%%			THE EISENSTEIN SERIES IN THE SPECIAL CASE				%%%%%%%%%%
%%%%%%%%%%%%%%%%%%%%%%%%%%%%%%%%%%%%%%%%%%%%%%%%%%%%%%%%%%%%%%%%%%%%%%%%%%%%%%%%%%%%%%%%%%%%%%%%%%%%%%%%%%%%%%%%%%%%%%%%%%%%%%%%%%%%%%%%%%%%%%%%%%%%%%%%%%%%%%%%%%%%%%%%%%%%%%%%%%%%
\subsubsection{The Eisenstein series in the special case $N=2q$}\label{SpecialCase} Let $q>2$ be a prime number. For some technical reasons, it is convenient for the proof of Theorem \ref{Theorem2} to see the form $f$ of level $q$ as a form of level $2q$ (see the beginning of Section 4.1 and Section 5.5 in \cite{twists}). For arbitrary level $N$, the Eisenstein series $E_{\amf}(\cdot, 1/2+it)$ are usually not eigenfunctions of the Hecke operators. In the special case where $N=2q$, there are exactly four inequivalent cups for $\Gamma_0(2q)$ which are
$$\amf= 1 \ , \ \frac{1}{2} \ , \ \frac{1}{q} \ , \ \frac{1}{2q},$$
see for example \cite[Proposition 2.6]{classical} and all are singular. The main advantage in this situation is that these Eisentein series are eigenforms of the Hecke operators $T_n$ for $(n,2q)=1$. More precisely, if $\amf=1/v$ with $v\in\{1,2,q,2q\}$, then we have for $(n,2q)=1$, 
$$T_n E_{\amf}(\cdot,1/2+it)=\lambda_{\amf}(n,it)E_{\amf}(\cdot,1/2+it),$$
with explicitly
\begin{equation}\label{EigenvaluesEisensteinSeries}
\lambda_{\amf}(n,it)=\left\{ \begin{array}{ccc} \sum_{ab=n}\omega(a)\left(\frac{a}{b}\right)^{it} & \ifm & v=q,2q \\  & &  \\ \sum_{ab=n}\omega(b)\left(\frac{a}{b}\right)^{it} & \ifm & v=1,2,
\end{array} \right.
\end{equation}
see \cite[(6.16)-(6.17)]{artin}.
\begin{remq}\label{RemarkEisenstein} In the case $N=q$, there are exactly two inequivalent cusps $\amf=1,1/q$ and the two Eisenstein series are eigenfunctions of the Hecke operators $T_n$ for $(n,q)=1$ with eigenvalues given by \eqref{EigenvaluesEisensteinSeries}. Moreover, they are also Eisenstein series of level $2q$ after the normalization by $1/\sqrt{3}$.
\end{remq}

%%%%%%%%%%%%%%%%%%%%%%%%%%%%%%%%%%%%%%%%%%%%%%%%%%%%%%%%%%%%%%%%%%%%%%%%%%%%%%%%%%%%%%%%%%%%%%%%%%%%%%%%%%%%%%%%%%%%%%%%%%%%%%%%%%%%%%%%%%%%%%%%%%%%%%%%%%%%%%%%%%%%%%%%%%%%%%%%%%%%%%%%%%%%%%%%%%%%%%%%%%%%%%%%%%%%%%%%%%%%%%%%%%%%%%%%%%%%%%%%%%%%%%%%%%%%%%%%%%%%%%%%%%%%%%
%%%%%%%%%%			MULTIPLICATIVE AND BOUNDEDNESS PROPERTIES				%%%%%%%%%%
%%%%%%%%%%%%%%%%%%%%%%%%%%%%%%%%%%%%%%%%%%%%%%%%%%%%%%%%%%%%%%%%%%%%%%%%%%%%%%%%%%%%%%%%%%%%%%%%%%%%%%%%%%%%%%%%%%%%%%%%%%%%%%%%%%%%%%%%%%%%%%%%%%%%%%%%%%%%%%%%%%%%%%%%%%%%%%%%%%%%%%%%%%%%%%%%%%%%%%%%%%%%%%%%%%%%%%%%%%%%%%%%%%%%%%%%%%%%%%%%%%%%%%%%%%%%%%%%%%%%%%%%%%%%%%
\subsection{Multiplicative and boundedness properties of Hecke eigenvalues} Let $f$ be any Hecke eigenform of level $N$ and nebentypus $\omega$ and let $\lambda_f(n)$ be the corresponding Hecke eigenvalues for $T_n$. Then for $(nm,N)=1$, we have
\begin{equation}\label{HeckeMult1}
\lambda_f(m)\lambda_f(n)=\sum_{d|(n,m)}\omega(d)\lambda_f\left(\frac{nm}{d^2}\right),
\end{equation}
\begin{equation}\label{HeckeMult2}
\overline{\lambda_f(n)}=\omegabar(n)\lambda_f(n).
\end{equation}
Note that if $f$ is holomorphic, then by the work of Deligne and Serre, we have the Ramanujan-Petersson conjecture, namely 
\begin{equation}\label{Ramanujan-Petersson}
|\lambda_f(n)|\leqslant \tau(n).
\end{equation}
The same bound is of course trivial in the special case $N=2q$ or $q$ and the Eisenstein series are eigenfunctions with eigenvalues given by \eqref{EigenvaluesEisensteinSeries}. In the case of a Maass cusp form $f$, the best result is due to Kim and Sarnak \cite{kim} and it is given by 
\begin{equation}\label{BoundKimSarnack}
|\lambda_f(n)|\leqslant \tau(n)n^{\theta} \ , \ \ \theta=\frac{7}{64},
\end{equation}
with an analogous bound for the spectral parameter
\begin{equation}\label{BoundSpecteralParameter}
|\Im m (t_f)|\leqslant \theta,
\end{equation}
where $1/4+t_f^2$ is the Laplace eigenvalue of $f$. However, the conjecture is true on average, in the sense that for every $X\geqslant 1$,
\begin{equation}\label{Ramanujan-Petersson-Average}
\sum_{n\leqslant X}|\lambda_f(n)|^2\ll (N(1+|t_f|))^\varepsilon X,
\end{equation}
with an implied constant depending only on $\varepsilon$ \cite[Proposition 19.6]{artin}. We will also need later similar bound for the fourth-power on average and it is enough for our purpose to restrict to prime numbers $p$ not dividing the level $N$
\begin{equation}\label{FourthPower}
\sum_{\substack{ p\leqslant X \\ (p,N)=1}}|\lambda_f(p)|^4 \ll (XN(1+|t_f|))^\varepsilon X,
\end{equation}
for any $\varepsilon>0$ and the constant only depends on $\varepsilon$. This bound is a consequence of the automorphy of the symmetric square $\mathrm{Sym}^2f$ and Rankin-Selberg theory (see for example \cite{park}).

%%%%%%%%%%%%%%%%%%%%%%%%%%%%%%%%%%%%%%%%%%%%%%%%%%%%%%%%%%%%%%%%%%%%%%%%%%%%%%%%%%%%%%%%%%%%%%%%%%%%%%%%%%%%%%%%%%%%%%%%%%%%%%%%%%%%%%%%%%%%%%%%%%%%%%%%%%%%%%%%%%%%%%%%%%%%%%%%%%%%%%%%%%%%%%%%%%%%%%%%%%%%%%%%%%%%%%%%%%%%%%%%%%%%%%%%%%%%%%%%%%%%%%%%%%%%%%%%%%%%%%%%%%%%%%
%%%%%%%%%%			HECKE EIGENVALUES AND FOURIER COEFFICIENTS				%%%%%%%%%%
%%%%%%%%%%%%%%%%%%%%%%%%%%%%%%%%%%%%%%%%%%%%%%%%%%%%%%%%%%%%%%%%%%%%%%%%%%%%%%%%%%%%%%%%%%%%%%%%%%%%%%%%%%%%%%%%%%%%%%%%%%%%%%%%%%%%%%%%%%%%%%%%%%%%%%%%%%%%%%%%%%%%%%%%%%%%%%%%%%%%%%%%%%%%%%%%%%%%%%%%%%%%%%%%%%%%%%%%%%%%%%%%%%%%%%%%%%%%%%%%%%%%%%%%%%%%%%%%%%%%%%%%%%%%%%

\subsection{Hecke eigenvalues and Fourier coefficients} Let $f$ be a modular form. For $z=x+iy \in \Hbf$, we write the Fourier expansion as
$$f(z)=\sum_{n=1}^\infty \rho_f(n)n^{\frac{k-1}{2}}e(nz) \ \ \mathrm{for} \ f\in \Bcal_k(N,\omega),$$
$$f(z)=\sum_{n\neq 0}\rho_f(n)|n|^{-1/2}W_{\frac{|n|}{n}\frac{\kappa}{2},it_f}(4\pi |n|y)e(nx), \ \ \mathrm{for} \ f\in\Bcal(N,\omega),$$
where $W_{\alpha,\beta}$ is a Whittaker function, as defined in \cite[§ 4]{artin}. For an Eisenstein series $E_{\amf}(z,1/2+it)$, we write
$$E_{\amf}(z,1/2+it)= c_{1,\amf}(t)y^{1/2+it}+c_{2,\amf}(t)y^{1/2-it}+\sum_{n\neq 0}\rho_{\amf}(n,it)|n|^{-1/2}W_{\frac{|n|}{n}\frac{\kappa}{2},it}(4\pi|n|y)e(nx).$$
When $f$ is a Hecke eigenform, there is a closed relation between the Fourier coefficients and the Hecke eigenvalues $\lambda_f(n)$ ; for $(m,N)=1$ and $n\geqslant 1$, one has
\begin{equation}\label{RelationHecke-Fourier}
\lambda_f(n)\rho_f(n)=\sum_{d|(m,n)}\omega(d)\rho_f\left(\frac{mn}{d^2}\right).
\end{equation}
In particular, for all $(m,q)=1$,
\begin{equation}\label{RelationHecke-Fourier2}
\lambda_f(m)\rho_f(1)=\rho_f(m).
\end{equation}
If $f$ is primitive, the relations \eqref{RelationHecke-Fourier} and \eqref{RelationHecke-Fourier2} are valid for every $m\geqslant 1$. We will also need lower bounds for the first coefficient $\rho_f(1)$ ; we have for any $\varepsilon>0$
\begin{equation}\label{LowerBound}
|\rho_f(1)|^2 \gg_\varepsilon \left\{ \begin{array}{ccc} \frac{\cosh(\pi t_f)}{N(1+|t_f|)^{\kappa}(N+|t_f|)^\varepsilon} & \ifm & f\in\Bcal(N,\omega) \\ & & \\ \frac{(4\pi)^{k-1}}{(k-1)!N^{1+\varepsilon}k^\varepsilon} & \ifm & f\in \Bcal_k(N,\omega),
\end{array} \right.
\end{equation}
see \cite[(6.22),(7.16)]{artin} and \cite[Lemma 2.2 and (2.23)]{michel2004}. For an Eisenstein series $E_{\amf}(\cdot,1/2+it)$, we have
\begin{equation}\label{LowerBoundEisenstein}
|\rho_{\amf}(1,it)|^2\gg \frac{\cosh(\pi t)}{N(1+|t|)^\kappa (\log(N+|t|))^2},
\end{equation}
see \cite[(6.23),(7.15)]{artin}.

%%%%%%%%%%%%%%%%%%%%%%%%%%%%%%%%%%%%%%%%%%%%%%%%%%%%%%%%%%%%%%%%%%%%%%%%%%%%%%%%%%%%%%%%%%%%%%%%%%%%%%%%%%%%%%%%%%%%%%%%%%%%%%%%%%%%%%%%%%%%%%%%%%%%%%%%%%%%%%%%%%%%%%%%%%%%%%%%%%%%%%%%%%%%%%%%%%%%%%%%%%%%%%%%%%%%%%%%%%%%%%%%%%%%%%%%%%%%%%%%%%%%%%%%%%%%%%%%%%%%%%%%%%%%%%
%%%%%%%%%%		TRACE FORMULA AND THE SPECRAL LARGE SIEVE INEQUALITY			%%%%%%%%%%
%%%%%%%%%%%%%%%%%%%%%%%%%%%%%%%%%%%%%%%%%%%%%%%%%%%%%%%%%%%%%%%%%%%%%%%%%%%%%%%%%%%%%%%%%%%%%%%%%%%%%%%%%%%%%%%%%%%%%%%%%%%%%%%%%%%%%%%%%%%%%%%%%%%%%%%%%%%%%%%%%%%%%%%%%%%%%%%%%%%%%%%%%%%%%%%%%%%%%%%%%%%%%%%%%%%%%%%%%%%%%%%%%%%%%%%%%%%%%%%%%%%%%%%%%%%%%%%%%%%%%%%%%%%%%%

\subsection{Summation formula, trace formula and the spectral large sieve inequality}

%%%%%%%%%%%%%%%%%%%%%%%%%%%%%%%%%%%%%%%%%%%%%%%%%%%%%%%%%%%%%%%%%%%%%%%%%%%%%%%%%%%%%%%%%%%%%%%%%%%%%%%%%%%%%%%%%%%%%%%%%%%%%%%%%%%%%%%%%%%%%%%%%%%%%%%%%%%%%%%%%%%%%%%%%%%%%%%%%%%%
%%%%%%%%%%					VORONOI SUMMATION FORMULA						%%%%%%%%%%	%%%%%%%%%%%%%%%%%%%%%%%%%%%%%%%%%%%%%%%%%%%%%%%%%%%%%%%%%%%%%%%%%%%%%%%%%%%%%%%%%%%%%%%%%%%%%%%%%%%%%%%%%%%%%%%%%%%%%%%%%%%%%%%%%%%%%%%%%%%%%%%%%%%%%%%%%%%%%%%%%%%%%%%%%%%%%%%%%%%%
\subsubsection{Voronoi summation formula}We state a version of Voronoi summation formula for the cuspidal case. For our purpose, it is enough to restrict to modular form of prime level $q$. 

\begin{proposition}\label{PropositionVoronoi} Let $q>2$ be a prime number, $\omega$ a Dirichlet character of modulus $q$ and $f$ a primitive Hecke cusp form of level $q$ and nebentypus $\omega$ with associated Hecke eigenvalues $(\lambda_f(n))_{n\geqslant 1}$. Given an integer $a$ coprime with $q$ and $g : \Rbf_{+}^*\rightarrow \Cbf$ a smooth and compactly supported function, we set
$$\mathcal{V}(f,a;q):= \sum_{n\geqslant 1}\lambda_f(n)e\left(\frac{a n}{q}\right)g(n).$$
Then 
\begin{equation}\label{VoronoiFormula}
\mathcal{V}(f,a;q)= \frac{\omegabar(a)}{q}\sum_{\pm}\sum_{n\geqslant 1}\lambda_f(n)e\left(\mp \frac{\overline{a}n}{q}\right)g_{\pm}\left(\frac{n}{q^2}\right),
\end{equation}
where
$$g_{\pm}(y)= \int_0^\infty g(x)\mathcal{J}_{\pm}(4\pi\sqrt{xy})dx,$$
with
$$\mathcal{J}_{+}(x) = 2\pi i^k J_{k-1}(x)  \ \ , \ \ \mathcal{J}_-(x)=0$$
if $f$ is holomorphic of weight $k$ and 
$$\mathcal{J}_{+}(x)= -\frac{\pi}{\sin (\pi i t_f)}\left( J_{2it_f}(x)-J_{-2it_f}(x)\right) \ , \ \mathcal{J}_{-}(x)=\varepsilon_f 4 \cos(\pi i t_f)K_{2it_f}(x)$$
if $f$ is a Maass form of parity $\varepsilon_f\in\{-1,+1\}$ and spectral parameter $t_f$.
\end{proposition}
\begin{proof}The proof is contained in \cite[Appendix A.3-A.4]{rankin} and also \cite[Lemma 2.3.1]{park} for the holomorphic case.
\end{proof}
Finally, we consider the decay properties of the Bessel transforms $g_{\pm}$ (see \cite[Lemma 2.4]{moments}).
%%%%%%%%%%%%%%%%%%%%%%%%%%%%%%%%%%%%%%%%%%%%%%%%%%%%%%%%%%%%%%%%%%%%%%%%%%%%%%%%%%%%%%%%%%
%%%%%%%%%%						LEMMA DECAY PROPERTIES						  %%%%%%%%%%
%%%%%%%%%%%%%%%%%%%%%%%%%%%%%%%%%%%%%%%%%%%%%%%%%%%%%%%%%%%%%%%%%%%%%%%%%%%%%%%%%%%%%%%%%%
\begin{lemme}\label{LemmaVoronoi} Let $g : \Rbf_{+}^*\rightarrow\Cbf$ be a smooth and compactly supported function satisfying 
\begin{equation}\label{GenericCondition}
x^ig^{(i)}(x)\ll_{i,\varepsilon}q^{\varepsilon i}
\end{equation}
for any $\varepsilon>0$ and $j\geqslant 0$. In the non-holomorphic case set $\vartheta = \Re e (it_f)$, otherwise set $\vartheta=0$. Then for any $\varepsilon>0$, for any $i,j\geqslant 0$ and all $y>0$, we have
$$y^j g_{\pm}^{(j)}(y)\ll_{i,j,\varepsilon} \frac{(1+y)^{j/2}}{\left(1+(yq^{-\varepsilon})^{1/2}\right)^i}\left(1+y^{-2\vartheta-\varepsilon}\right).$$
\end{lemme}

%%%%%%%%%%%%%%%%%%%%%%%%%%%%%%%%%%%%%%%%%%%%%%%%%%%%%%%%%%%%%%%%%%%%%%%%%%%%%%%%%%%%%%%%%%%%%%%%%%%%%%%%%%%%%%%%%%%%%%%%%%%%%%%%%%%%%%%%%%%%%%%%%%%%%%%%%%%%%%%%%%%%%%%%%%%%%%%%%%%%
%%%%%%%%%%						THE PETERSSON FORMULA						%%%%%%%%%%
%%%%%%%%%%%%%%%%%%%%%%%%%%%%%%%%%%%%%%%%%%%%%%%%%%%%%%%%%%%%%%%%%%%%%%%%%%%%%%%%%%%%%%%%%%%%%%%%%%%%%%%%%%%%%%%%%%%%%%%%%%%%%%%%%%%%%%%%%%%%%%%%%%%%%%%%%%%%%%%%%%%%%%%%%%%%%%%%%%%%
\subsubsection{The Petersson formula} For $k\geqslant 2$ an integer such that $k\equiv \kappa \ (\modm 2)$, the Petersson trace formula expresses an average of product of Fourier coefficients over $\Bcal_k(N,\omega)$ in terms of sums of Kloosterman sums (see \cite[Theorem 9.6]{spectral} and \cite[Proposition 14.5]{analytic}) : for any integers $n,m>0$, we have
\begin{equation}\label{Petersson}
\frac{(k-2)!}{(4\pi)^{k-1}}\sum_{g\in \Bcal_k(N,\omega)}\rho_g(n)\overline{\rho_g(m)}=\delta(m,n)+\Delta_{N,k}(m,n),
\end{equation}
where
$$\Delta_{N,k}(m,n):= 2\pi i^{-k}\sum_{N|c}\frac{1}{c}S_{\omega}(m,n;c)J_{k-1}\left(\frac{4\pi \sqrt{mn}}{c}\right),$$
and the Kloosterman sum $S_\omega(m,n;c)$ is defined by
$$S_\omega(m,n;c)=\sum_{\substack{d \ (\modm c) \\ (d,c)=1}}\omegabar(d)e\left(\frac{m\overline{d}+nd}{c}\right).$$

%%%%%%%%%%%%%%%%%%%%%%%%%%%%%%%%%%%%%%%%%%%%%%%%%%%%%%%%%%%%%%%%%%%%%%%%%%%%%%%%%%%%%%%%%%%%%%%%%%%%%%%%%%%%%%%%%%%%%%%%%%%%%%%%%%%%%%%%%%%%%%%%%%%%%%%%%%%%%%%%%%%%%%%%%%%%%%%%%%%%
%%%%%%%%%%						THE KUZNETSOV FORMULA						%%%%%%%%%%
%%%%%%%%%%%%%%%%%%%%%%%%%%%%%%%%%%%%%%%%%%%%%%%%%%%%%%%%%%%%%%%%%%%%%%%%%%%%%%%%%%%%%%%%%%%%%%%%%%%%%%%%%%%%%%%%%%%%%%%%%%%%%%%%%%%%%%%%%%%%%%%%%%%%%%%%%%%%%%%%%%%%%%%%%%%%%%%%%%%%
\subsubsection{The Kuznetsov formula}Let $\phi : \Rbf_+\rightarrow \Cbf$ be a smooth function satisfying $\phi(0)=\phi'(0)=0$ and $\phi^{(j)}(x)\ll_\varepsilon (1+x)^{-2-\varepsilon}$ for $0\leqslant j\leqslant 3$ and every $\varepsilon>0$. For $\kappa\in\{0,1\}$, we define the two Bessel transforms
\begin{alignat*}{1}
\dot{\phi}(k)= & \ i^k\int_0^\infty J_{k-1}(x)\phi(x)\frac{dx}{x}, \\
\tilde{\phi}(t)= & \ \frac{it^\kappa}{2\sinh(\pi t)}\int_0^\infty \left( J_{2it}(x)-(-1)^\kappa J_{-2it}(x)\right)\phi(x)\frac{dx}{x}.
\end{alignat*}
Then for every integers $m,n>0$, we have the following spectral decomposition of the Kloosterman sums,
\begin{equation}\label{Kuznetsov}
\begin{split}
&\mathop{\sum\sum}_{\substack{k\equiv \kappa \modm 2 \\ k>\kappa \\ g\in \Bcal_k(N,\omega)}}\dot{\phi}(k)\frac{(k-1)!}{\pi(4\pi)^{k-1}}\overline{\rho_g(m)}\rho_g(n)+\sum_{g\in\Bcal(N,\omega)}\tilde{\phi}(t_g)\frac{4\pi}{\cosh(\pi t_g)}\overline{\rho_g(m)}\rho_g(n) \\ 
& \ + \sum_{\amf \ \mathrm{singular}}\int_0^\infty \tilde{\phi}(t)\frac{1}{\cosh(\pi t)}\overline{\rho_{\amf}(m,it)}\rho_{\amf}(n,it)dt=\Delta_{N,\phi}(m,n),
\end{split}
\end{equation}
where
$$\Delta_{N,\phi}(m,n)=\sum_{N|c}\frac{1}{c}S_\omega(m,n;c)\phi\left(\frac{4\pi\sqrt{mn}}{c}\right),$$
see \cite[Theorem 9.4 and 9.8]{spectral}.

%%%%%%%%%%%%%%%%%%%%%%%%%%%%%%%%%%%%%%%%%%%%%%%%%%%%%%%%%%%%%%%%%%%%%%%%%%%%%%%%%%%%%%%%%%%%%%%%%%%%%%%%%%%%%%%%%%%%%%%%%%%%%%%%%%%%%%%%%%%%%%%%%%%%%%%%%%%%%%%%%%%%%%%%%%%%%%%%%%%%
%%%%%%%%%%					THE SPECTRAL LARGE SIEVE							%%%%%%%%%%
%%%%%%%%%%%%%%%%%%%%%%%%%%%%%%%%%%%%%%%%%%%%%%%%%%%%%%%%%%%%%%%%%%%%%%%%%%%%%%%%%%%%%%%%%%%%%%%%%%%%%%%%%%%%%%%%%%%%%%%%%%%%%%%%%%%%%%%%%%%%%%%%%%%%%%%%%%%%%%%%%%%%%%%%%%%%%%%%%%%%
\subsubsection{The spectral large sieve inequality} 
\begin{lemme}\label{LemmaWeilBound} Let $N\geqslant 1$ be an integers and $\omega$ a Dirichlet character of modulus $N$ whose conductor is odd and squarefree. Then for any $m,n\in\Zbf$ and $N|c$, we have the Weil bound
$$|S_\omega(m,n;c)|\leqslant \tau(c)(m,n,c)^{1/2}c^{1/2}.$$
\end{lemme}
\begin{proof}
The proof is a consequence of the twisted multiplicativity of Kloosterman sums and \cite[Propositions 9.7 and 9.8]{weilbound}.
\end{proof}
\begin{proposition}\label{PropositionSieve} Let $N\geqslant 1$ be an integer, $\omega$ a Dirichlet character of modulus $N$ with squarefree and odd conductor and $\kappa\in\{0,1\}$ such that $\omega(-1)=(-1)^\kappa$. Let $T\geqslant 1$, $M\geqslant 1/2$ and $(a_m)_{M<m\leqslant 2M}$ a sequence of complex numbers. Then for any $\varepsilon>0$, each of the following three quantities
$$\sum_{\substack{\kappa<k\leqslant T \\ k\equiv \kappa \ (\modm 2)}}\Gamma(k)\sum_{g\in\Bcal_k(N,\omega)}\left|\sum_{M<m\leqslant 2M}a_m\rho_g(m)\right|^2,$$
$$\sum_{\substack{g\in \Bcal(N,\omega) \\ |t_g|\leqslant T}}\frac{(1+|t_g|)^\kappa}{\cosh(\pi t_g)}\left|\sum_{M<m\leqslant 2M}a_m\rho_g(m)\right|^2,$$
$$\sum_{\amf \ \mathrm{singular}}\int_{-T}^T \frac{(1+|t|)^\kappa}{\cosh(\pi t)}\left|\sum_{M<m\leqslant 2M}a_m\rho_{\amf}(m,t)\right|^2,$$
is bounded, up to a constant depending only on $\varepsilon$, by
\begin{equation}\label{BoundSieve}
\left(T^2+\frac{M^{1+\varepsilon}}{N}\right)\sum_{M<m\leqslant 2M}|a_m|^2.
\end{equation}
\end{proposition}
\begin{proof}
This result has been proved in \cite[Proposition 4.7]{drappeau}, except that the bound \eqref{BoundSieve} is replaced by $\left(T^2+c(\omega)^{1/2}M^{1+\varepsilon}N^{-1}\right)||a_m||_2^2$, where $c(\omega)$ is the conductor of $\omega$. The extra factor $c(\omega)$ in the proof of this Proposition appears in \cite[Lemma 4.6, (4.20)]{drappeau} and is a consequence of the general estimation \cite[Theorem 9.3]{weilbound}
$$S_\omega(m,n;c)\ll \tau(c)^{O(1)}(m,n,c)^{1/2}(c(\omega)c)^{1/2}.$$
By the hypothesis on the conductor of $\omega$ (that it is squarefree), we can apply Lemma \ref{LemmaWeilBound} whose consequence is the cancellation of the factor $c(\omega)^{1/2}$ in \cite[Lemma 4.6, (4.20)]{drappeau} and the rest of the proof is completely similar.
\end{proof}

%%%%%%%%%%%%%%%%%%%%%%%%%%%%%%%%%%%%%%%%%%%%%%%%%%%%%%%%%%%%%%%%%%%%%%%%%%%%%%%%%%%%%%%%
%%%%%%%%%%%%%%%%%%%%%%%%%%%%%%%%%%%%%%%%%%%%%%%%%%%%%%%%%%%%%%%%%%%%%%%%%%%%%%%%%%%%%%%%
%%%%%%%%%%%%%%%%%%%%%%%%%%%%%%%%%%%%%%%%%%%%%%%%%%%%%%%%%%%%%%%%%%%%%%%%%%%%%%%%%%%%%%%%
%%%%%%%%%%				L-FUNCTIONS AND FUNCTIONAL EQUATIONS					%%%%%%%%%%
%%%%%%%%%%%%%%%%%%%%%%%%%%%%%%%%%%%%%%%%%%%%%%%%%%%%%%%%%%%%%%%%%%%%%%%%%%%%%%%%%%%%%%%%
%%%%%%%%%%%%%%%%%%%%%%%%%%%%%%%%%%%%%%%%%%%%%%%%%%%%%%%%%%%%%%%%%%%%%%%%%%%%%%%%%%%%%%%%
%%%%%%%%%%%%%%%%%%%%%%%%%%%%%%%%%%%%%%%%%%%%%%%%%%%%%%%%%%%%%%%%%%%%%%%%%%%%%%%%%%%%%%%%
\subsection{$L$-functions and functional equations}

%%%%%%%%%%%%%%%%%%%%%%%%%%%%%%%%%%%%%%%%%%%%%%%%%%%%%%%%%%%%%%%%%%%%%%%%%%%%%%%%%%%%%%%%
%%%%%%%%%%%%%%%%%%%%%%%%%%%%%%%%%%%%%%%%%%%%%%%%%%%%%%%%%%%%%%%%%%%%%%%%%%%%%%%%%%%%%%%%
%%%%%%%%%%						DIRICHLET L-FUNCTIONS						%%%%%%%%%%
%%%%%%%%%%%%%%%%%%%%%%%%%%%%%%%%%%%%%%%%%%%%%%%%%%%%%%%%%%%%%%%%%%%%%%%%%%%%%%%%%%%%%%%%
%%%%%%%%%%%%%%%%%%%%%%%%%%%%%%%%%%%%%%%%%%%%%%%%%%%%%%%%%%%%%%%%%%%%%%%%%%%%%%%%%%%%%%%%
\subsubsection{Dirichlet $L$-functions}
Let $\chi$ be a non-principal Dirichlet character of modulus $q>2$ with $q$ prime, $\kappa\in\{0,1\}$ satisfying $\chi(-1)=(-1)^\kappa$ and define the complete $L$-function
$$\Lambda(\chi,s):= q^{s/2}L_\infty(\chi,s)L(\chi,s),$$
where
\begin{equation}\label{DefinitionLinfty}
L_\infty(\chi,s):=\pi^{-s/2}\Gamma\left( \frac{s+\kappa}{2}\right).
\end{equation}
It is well known that $\Lambda(\chi,s)$ admits an analytic continuation to the whole complex plane and satisfies the functional equation (c.f. Theorem 4.15, \cite{analytic})
\begin{equation}\label{FunctionalEquation}
\Lambda(\chi,s)= i^\kappa \varepsilon(\chi)\Lambda (\overline{\chi},1-s),
\end{equation}
where $\varepsilon(\chi)$ is the normalized Gauss sum defined by
\begin{equation}\label{DefinitionGaussSum}
\varepsilon(\chi) := \frac{1}{q^{1/2}}\sum_{x \in\Fbf_q^\times}\chi(x)e\left(\frac{x}{q}\right).
\end{equation}
Using \eqref{FunctionalEquation}, we can express the central value of a Dirichlet $L$-function as a convergent series (see Theorem 5.3, \cite{analytic}) and thus, extend in an easy way the proof to a product of three $L$-functions.
%%%%%%%%%%%%%%%%%%%%%%%%%%%%%%%%%%%%%%%%%%%%%%%%%%%%%%%%%%%%%%%%%%%%%%%%%%%%%%%%%%%%%%%%
%%%%%%%%%%				LEMMA APROXIMATE FUNCTIONAL EQUATION					%%%%%%%%%%
%%%%%%%%%%%%%%%%%%%%%%%%%%%%%%%%%%%%%%%%%%%%%%%%%%%%%%%%%%%%%%%%%%%%%%%%%%%%%%%%%%%%%%%%
\begin{lemme}\label{LemmeApproximate} Let $\chi,\omega_1,\omega_2$ be Dirichlet characters such that $\chi\neq 1,\overline{\omega}_1,\overline{\omega}_2$. If $\kappa$ $($resp $\kappa_1,$ $\kappa_2)$ $\in \{0,1\}$ are such that $\chi(-1)=(-1)^\kappa$ $($resp $\omega_1(-1)=(-1)^{\kappa_1}$, $\omega_2(-1)=(-1)^{\kappa_2})$, then we have
\begin{alignat*}{1}
L\left(\chi,\frac{1}{2}\right)&L\left(\chi\omega_1,\frac{1}{2}\right)L\left(\chi\omega_2,\frac{1}{2}\right)= \mathop{\sum\sum\sum}_{n_0,n_1,n_2\geqslant 1}\frac{\chi(n_0n_1n_2)\chi_1(n_1)\chi_2(n_2)}{(n_0n_1n_2)^{1/2}}\mathbf{V}_\chi\left(\frac{n_0n_1n_2}{q^{3/2}}\right) \\ + & \ \chi(-1)i^{\kappa+\kappa_1+\kappa_2}\varepsilon(\chi)\varepsilon(\chi\omega_1)\varepsilon(\chi\omega_2)\mathop{\sum\sum\sum}_{n_0,n_1,n_2\geqslant 1}\frac{\overline{\chi}(n_0n_1n_2)\overline{\chi}_1(n_1)\overline{\chi}_2(n_2)}{(n_0n_1n_2)^{1/2}}\mathbf{V}_\chi\left(\frac{n_0n_1n_2}{q^{3/2}}\right),
\end{alignat*}
where 
\begin{equation}\label{DefinitionV2}
\mathbf{V}_\chi(x):= \frac{1}{2\pi i}\int_{(2)}\frac{L_\infty\left(\chi,\frac{1}{2}+s\right)L_\infty\left(\chi\chi_1,\frac{1}{2}+s\right)L_\infty\left(\chi\chi_2,\frac{1}{2}+s\right)}{L_\infty\left(\chi,\frac{1}{2}\right)L_\infty\left(\chi\chi_1,\frac{1}{2}\right)L_\infty\left(\chi\chi_2,\frac{1}{2}\right)}x^{-s}Q(s)\frac{ds}{s},
\end{equation}
for some entire and even function $Q(s)$ with exponential decay in vertical strips and satisfying $Q(0)=1.$
\end{lemme}
\begin{rmk}\label{Remark} Since $\omega_1$ and $\omega_2$ have been fixed, the function $\mathbf{V}_\chi$ depends on $\chi$ only through its parity. Now shifting the $s$-contour to the right in \eqref{DefinitionV2} and we see that for $x\geqslant 1$ and any $A\geqslant 0$, we have the estimation
$$\mathbf{V}_\chi(x) \ll_A x^{-A}.$$
Now moving the $s$-line to $\Re e (s)=-\frac{1}{2}+\varepsilon$, we pass a simple pole at $s=0$ of residu $1$ and thus, we obtain for $x\leqslant 1$
$$\mathbf{V}_\chi(x)= 1+O_\varepsilon\left(x^{1/2-\varepsilon}\right).$$
\end{rmk}

%%%%%%%%%%%%%%%%%%%%%%%%%%%%%%%%%%%%%%%%%%%%%%%%%%%%%%%%%%%%%%%%%%%%%%%%%%%%%%%%%%%%%%%%
%%%%%%%%%%%%%%%%%%%%%%%%%%%%%%%%%%%%%%%%%%%%%%%%%%%%%%%%%%%%%%%%%%%%%%%%%%%%%%%%%%%%%%%%
%%%%%%%%%%						TWISTED L-FUNCTIONS							%%%%%%%%%%
%%%%%%%%%%%%%%%%%%%%%%%%%%%%%%%%%%%%%%%%%%%%%%%%%%%%%%%%%%%%%%%%%%%%%%%%%%%%%%%%%%%%%%%%
%%%%%%%%%%%%%%%%%%%%%%%%%%%%%%%%%%%%%%%%%%%%%%%%%%%%%%%%%%%%%%%%%%%%%%%%%%%%%%%%%%%%%%%%
\subsubsection{Twisted $L$-functions}\label{SectionTwisted}

Let $q>2$ be a prime number, $\omega$ a Dirichlet character of modulus $q$ and $f$ a primitive Hecke cusp of level $q$ and nebentypus $\omega$. For $\chi$ a non trivial character modulo $q$, we can construct the twisted modular form $f\otimes \chi$ whose $n$-th Fourier coefficient is given by $\rho_f(n)\chi(n)$. This form is again a Hecke eigenform of level $q^2$ with nebentypus $\omega\chi^2$ and eigenvalues $\lambda_f(n)\chi(n)$ for $(n,q)=1$ (see \cite[Chapter 7]{classical}), but it is not necessarily primitive because the level of $f$ and the conductor of $\chi$ are not coprime. In order to have a good understanding of this twist and a perfect description of the sign of the functional equation of the associated $L$-function $L(f\otimes\chi,s)$, it is convenient to switch into the representation theory language.

\vspace{0.1cm}

A well known principle associates to $f$ an automorphic cuspidal representation $\pi_f$ of $\GL_2(\Abf_{\Qbf})$, where $\Abf_{\Qbf}$ is the adele ring of $\Qbf$, with central character $\tilde{\omega}$, the adelization of $\omega$ (see \cite{goldfeld}) and of conductor $q$ (as defined in \cite[Chapter 5]{adele}). Such a representation admits a factorization $\pi_f=\otimes_{p\leqslant\infty}\pi_{f,p}$ into irreducible local representations of $\GL_2(\Qbf_p)$ with central character $\tilde{\omega}_p$ and where $\tilde{\omega}=\otimes_{p\leqslant\infty}\tilde{\omega}_p$.

\vspace{0.1cm}

Let $\tilde{\chi}=\otimes_{p\leqslant\infty}\tilde{\chi}_p$ be the adelization of the character $\chi$. Then the twisted representation $\pi_f\otimes\tilde{\chi}$ defined by $g\mapsto\tilde{\chi}(\det(g))\pi_f(g)$ admits the factorization $\otimes_{p\leqslant\infty}\pi_{f,p}\otimes\tilde{\chi}_p$ \cite[Corollary 11.2]{jacquet}. We can attach to $\pi_f\otimes\tilde{\chi}$ an $L$-function which can be represented as an infinite product over the primes of local $L$-functions
$$L(\pi_f\otimes\tilde{\chi},s)=\prod_{p<\infty}L_{p}(\pi_{f,p}\otimes\tilde{\chi}_p,s)=\sum_{n=1}^\infty\frac{\lambda_{\pi_f\otimes\tilde{\chi}}(n)}{n^s}, \ \ \Re e (s)>1,$$
with the property that 
$$\lambda_{\pi_f\otimes\tilde{\chi}}(n)=\lambda_f(n)\chi(n) \ \mathrm{for} \ (n,q)=1.$$ The $L$-function $L(\pi_f\otimes\tilde{\chi},s)$ is completed by the factor at the archimedean place $L(\pi_{f,\infty}\otimes\tilde{\chi}_\infty,s)$ and the product $L(\pi_f\otimes\tilde{\chi},s)L(\pi_{f,\infty}\otimes\tilde{\chi}_\infty,s)$ can be extended holomorphically to $\Cbf$ and satisfies a functional equation of the form
$$L(\pi_f\otimes\tilde{\chi},s)L(\pi_{f,\infty}\otimes\tilde{\chi}_\infty,s)=\varepsilon(\pi_f\otimes\tilde{\chi},s)L(\tilde{\pi}_f\otimes\overline{\tilde{\chi}},1-s)L(\tilde{\pi}_{f,\infty}\otimes\overline{\tilde{\chi}}_\infty,s),$$
where $\tilde{\pi}_f=\overline{\omega}\otimes\pi_f$ is the contragredient reprsentation and $\varepsilon(\pi_f\otimes\tilde{\chi},s)$ is called the root number. The representation $\pi_f\otimes\tilde{\chi}$ being ramified only at $p=q$ and possibly at $p=\infty$, we have 
$$\varepsilon(\pi_f\otimes\tilde{\chi},s)=\varepsilon(\pi_{f,\infty}\otimes\tilde{\chi}_\infty,s)\varepsilon(\pi_{f,q}\otimes\tilde{\chi}_q,s).$$
The infinite factor $\varepsilon(\pi_{f,\infty}\otimes\tilde{\chi}_\infty,s)$ is of modulus $1$ and depends on $\chi$ only through its parity (see for example \cite[Theorem 6.16]{adele}). To determine the finite part $\varepsilon(\pi_{f,q}\otimes\tilde{\chi}_q,s)$, we proceed as follows : Since the conductor of $\pi_{f,q}$, which is equal to the conductor of $\pi_f$, is a prime number, the local representation $\pi_{f,q}$ is not supercuspidal and thus corresponds either to an irreducible principal series, or to a special representation (see \cite[Remark 4.25]{adele}). More precisely, $\pi_{f,q}$ is a special representation $\Bcal(\mu_1,\mu_2)$ when the original character $\omega$ is trivial. In this case, the two quasi character $\mu_1,\mu_2$ of $\Qbf_q^\times$ are unramified and satisfy $\mu_1\mu_2=1$. If $\omega$ is not the trivial character, then $\pi_{f,q}$ is an irreducible principal series representation $\Bcal(\mu_1,\mu_2)$ with $\mu_1\mu_2=\tilde{\omega}_q$, $\mu_1$ unramified and $\mu_2$ is ramified with conductor $q$ (the argument is again a consequence of \cite[Remark 4.25]{adele}). The twist by $\tilde{\chi}_q$ preserves principal series or special representations is the sense that if $\pi_{f,q}=\Bcal(\mu_1,\mu_2)$, then $\pi_{f,q}\otimes \tilde{\chi}_q\simeq \Bcal(\mu_1\tilde{\chi}_q,\mu_2\tilde{\chi}_q)$. Since $\chi$ is not trivial, $\tilde{\chi}_q$ is ramified and we obtain in both cases (principal or special) 
\begin{equation}\label{Factorization}
\varepsilon(\pi_{f,q}\otimes\tilde{\chi}_q,s)=\varepsilon (\mu_1\tilde{\chi}_q,s)\varepsilon (\mu_2\tilde{\chi}_q,s),
\end{equation}
(see \cite[Theorem 6.15]{adele}), where for any quasi character $\mu$ of $\Qbf_q^\times$, $\varepsilon(\mu,s)$ is the root number appearing in the $\GL_1$ theory. If $\mu$ is unramified, we have $\varepsilon(\mu,s)=1$, otherwise, writing $\mu=\psi |\cdot|^\lambda_q$ for some $\lambda\in\Cbf$ and $\psi$ unitary with conductor $q^f$, we have $\varepsilon(\mu,s)=\varepsilon(\psi,s+\lambda)$ with (see \cite[Chapter 2, Ex 2.6]{goldfeld})
\begin{equation}\label{RootNumberGL1}
\varepsilon(\psi,s)= q^{-f(s-1/2)}\psi(q)^f\varepsilon(\overline{\psi}),
\end{equation}
where $\varepsilon(\overline{\psi})$ is the normalized Gauss sum
$$\varepsilon(\overline{\psi})=\frac{1}{q^{f/2}}\sumstar_{a (\modm q^f)}\overline{\psi}(a)e\left(\frac{a}{q^f}\right).$$
It follows from \cite[Definition 2.1.7]{goldfeld} that 
\begin{equation}\label{FinalSign}
\varepsilon(\pi_{f,q}\otimes\tilde{\chi}_q,s) = \left\{ \begin{array}{ccc} q^{-2(s-1/2)}\varepsilon(\chi)\varepsilon(\omega\chi) & \ifm & \chi\omega\neq 1 \\ & & \\ \mathrm{unimportant \ for \ us}. \end{array} \right. 
\end{equation}
We summarize all the previous computations in the following more convenient form:
%%%%%%%%%%%%%%%%%%%%%%%%%%%%%%%%%%%%%%%%%%%%%%%%%%%%%%%%%%%%%%%%%%%%%%%%%%%%%%%%%%%%%%%%
%%%%%%%%%%					PROPOSITION TWSITED L-FUNCTIONS					%%%%%%%%%%
%%%%%%%%%%%%%%%%%%%%%%%%%%%%%%%%%%%%%%%%%%%%%%%%%%%%%%%%%%%%%%%%%%%%%%%%%%%%%%%%%%%%%%%%
\begin{proposition}\label{PropositionTwistedL} Let $q>2$ be a prime number, $\omega$ a Dirichlet character of modulus $q$ and $f$ a primitive Hecke cusp form of level $q$ and nebentypus $\omega$ with associated Hecke eigenvalues $\lambda_f(n)$ for all $n\geqslant 1$. Then for every character $\chi$ modulo $q$ such that $\chi\neq 1,\overline{\omega}$, the twisted modular form $f\otimes \chi$ is a primitive Hecke cusp of level $q^2$ and nebentypus $\chi^2\omega$ with associated Hecke eigenvalues $\chi(n)\lambda_f(n)$ for every $n\geqslant 1$. If 
\begin{equation}\label{TwistedL-function}
L(f\otimes\chi,s):= \sum_{n=1}^\infty \frac{\lambda_f(n)\chi(n)}{n^s}, \ \Re e (s)>1
\end{equation}
is the associated $L$-function, then there exists a factor $L_\infty(f\otimes\chi,s)$ such that the product
$$\Lambda(f\otimes\chi,s):= q^sL_\infty(f\otimes\chi,s)L(f\otimes\chi,s)$$
extends holomorphically to $\Cbf$ and satisfies the functional equation 
\begin{equation}\label{FunctionalEquationTwisted}
\Lambda(f\otimes\chi,s)=\varepsilon_\infty(f,\chi)\varepsilon(f\otimes\chi)\Lambda(\overline{f}\otimes \overline{\chi},1-s),
\end{equation}
where $\varepsilon(f\otimes \chi)=\varepsilon(\chi)\varepsilon(\omega\chi)$ and the infinite factors $\varepsilon_\infty(f,\chi)$ and $L_\infty(f\otimes\chi,s)$ satisfy $|\varepsilon_\infty(f,\chi)|=1$ and both depends on $\chi$ only trough its parity.
\end{proposition}
\begin{remq} The infinite factor presents as a product of Gamma functions
$$L_\infty(f\otimes \chi,s)=\pi^{-s}\Gamma\left(\frac{s-\mu_{1,f\otimes\chi}}{2}\right)\Gamma\left(\frac{s-\mu_{2,f\otimes\chi}}{2}\right),$$
where $\mu_{i,f\otimes\chi}$ are the local parameters at the infinite place (encodes the weight in the holomorphic setting or the Laplace eigenvalue if $f$ is a Maass form) and we recall that they depend on $\chi$ at most trough its parity. In any case, a consequence of the work of Kim and Sarnak \cite{kim} toward the Ramanujan-Petersson conjecture is that
\begin{equation}\label{KimInfiny}
\Re e (\mu_{i,f\otimes\chi})\leqslant \frac{7}{64}.
\end{equation}
\end{remq}

We finally state the analogous of Lemma \ref{LemmeApproximate} for the product $L(f\otimes\chi,s)L(\chi,s)$ on the critical point $s=1/2$ (see \cite[§ 1.3.2]{park} for the general result).
%%%%%%%%%%%%%%%%%%%%%%%%%%%%%%%%%%%%%%%%%%%%%%%%%%%%%%%%%%%%%%%%%%%%%%%%%%%%%%%%%%%%%%%%
%%%%%%%%%%		PROPOSITION APPROXIMATE FUNCTIONAL EQUATION L-FUNCTIONS		%%%%%%%%%%
%%%%%%%%%%%%%%%%%%%%%%%%%%%%%%%%%%%%%%%%%%%%%%%%%%%%%%%%%%%%%%%%%%%%%%%%%%%%%%%%%%%%%%%%
\begin{proposition}\label{PropositionApproximateFunctionalEquationTwisted} Let $q>2$ be a prime number, $\omega$ a Dirichlet character of modulus $q$ and $f$ a primitive Hecke cusp form of level $q$ and nebentypus $\omega$ with associated Hecke eigenvalues $\lambda_f(n)$ for all $n\geqslant 1$. Then for every character $\chi$ modulo $q$ such that $\chi\neq 1,\overline{\omega}$, $\chi(-1)=(-1)^\kappa$, we have
\begin{equation}
\begin{split}
L\left(f\otimes\chi,\frac{1}{2}\right)L\left(\chi,\frac{1}{2}\right)=  & \ \mathop{\sum\sum}_{n,m\geqslant 1}\frac{\lambda_f(n)\chi(n)\chi(m)}{(nm)^{1/2}}\mathbf{V}_{f,\chi}\left(\frac{nm}{q^{3/2}}\right) \\  + & \ i^{\kappa}\varepsilon_\infty(f,\chi)\varepsilon(f\otimes\chi)\varepsilon(\chi)\mathop{\sum\sum}_{n,m\geqslant 1}\frac{\overline{\lambda_f(n)}\chibar(n)\chibar(m)}{(nm)^{1/2}}\mathbf{V}_{f,\chi}\left(\frac{nm}{q^{3/2}}\right),
\end{split}
\end{equation}
where
\begin{equation}\label{DefinitionVfchi}
\mathbf{V}_{f,\chi}(x):= \frac{1}{2\pi i}\int_{(2)}\frac{L_\infty\left(f\otimes\chi,s+\frac{1}{2}\right)L_\infty\left(\chi,s+\frac{1}{2}\right)}{L_\infty\left(f\otimes\chi,\frac{1}{2}\right)L_\infty\left(\chi,\frac{1}{2}\right)}x^{-s}P(s)\frac{ds}{s},
\end{equation}
for some suitable entire even function $P(s)$ with exponential decay in vertical strips and satisfying $P(0)=1$.
\end{proposition}
\begin{remq} Shifting the $s$-contour on the right in \eqref{DefinitionVfchi} and we obtain that for every $x\geqslant 1$ and any $A>0$, 
$$\mathbf{V}_{f,\chi}(x)\ll_A x^{-A}.$$
By \eqref{KimInfiny}, moving the $s$-line to $\Re e (s)= -1/4$, we catch a simple pole at $s=0$ of residue $1$ and thus
$$\mathbf{V}_{f,\chi}(x)=1+O(x^{1/4}) \ \mathrm{for \ }0<x\leqslant 1.$$
\end{remq}

%%%%%%%%%%%%%%%%%%%%%%%%%%%%%%%%%%%%%%%%%%%%%%%%%%%%%%%%%%%%%%%%%%%%%%%%%%%%%%%%%%%%%%%%%%
%%%%%%%%%%%%%%%%%%%%%%%%%%%%%%%%%%%%%%%%%%%%%%%%%%%%%%%%%%%%%%%%%%%%%%%%%%%%%%%%%%%%%%%%%%%%%%%%%%%%%%%%%%%%%%%%%%%%%%%%%%%%%%%%%%%%%%%%%%%%%%%%%%%%%%%%%%%%%%%%%%%%%%%%%%%%%%%%%%%%%%%%%%%%%%%%%%%%%%%%%%%%%%%%%%%%%%%%%%%%%%%%%%%%%%%%%%%%%%%%%%%%%%%%%%%%%%%%%%%%%%%%%%%%%%
%%%%%%%%%% 				ALGEBRAIC TWIST OF MODULAR FORMS						%%%%%%%%%%
%%%%%%%%%%%%%%%%%%%%%%%%%%%%%%%%%%%%%%%%%%%%%%%%%%%%%%%%%%%%%%%%%%%%%%%%%%%%%%%%%%%%%%%%%%%%%%%%%%%%%%%%%%%%%%%%%%%%%%%%%%%%%%%%%%%%%%%%%%%%%%%%%%%%%%%%%%%%%%%%%%%%%%%%%%%%%%%%%%%%%%%%%%%%%%%%%%%%%%%%%%%%%%%%%%%%%%%%%%%%%%%%%%%%%%%%%%%%%%%%%%%%%%%%%%%%%%%%%%%%%%%%%%%%%%
%%%%%%%%%%%%%%%%%%%%%%%%%%%%%%%%%%%%%%%%%%%%%%%%%%%%%%%%%%%%%%%%%%%%%%%%%%%%%%%%%%%%%%%%%%
\section{$\ell$-adic twists of modular forms}\label{SectionAlgebraicTwsits} In this section, we fix $q>2$ a prime number, $\omega$ a Dirichlet character modulo $q$, $f$ a primitive Hecke cusp form of level $q$ and nebentypus $\omega$ and we denote by $\{\lambda_f(n)\}_{n\geqslant 1}$ the Hecke eigenvalues of $f$. For any $t\in \Rbf$, we also define the twisted divisor function $\lambda_\omega(n,it)$ by
\begin{equation}\label{DefinitionTwistedDivisorFuntion}
\lambda_\omega(n,it):= \sum_{ab=n}\omega(a)\left(\frac{a}{b}\right)^{it},
\end{equation}
which, for $(n,q)=1$, appears as Hecke eigenvalues of Eisenstein series $E_{\amf}(\cdot, 1/2+it)$ of level $q$ and nebentypus $\omega$ for a suitable choice of cusp $\amf$ (c.f. § \ref{SpecialCase} and \eqref{EigenvaluesEisensteinSeries}).
\vspace{0.2cm}

As announced in Section \ref{SectionSketch}, for $K : \Zbf \longrightarrow\Cbf$ a $q$-periodic function, a crucial step in the proof of Theorem \ref{Theorem1} requires non trivial estimates for sums of the shape 
\begin{equation}\label{Coorelation1}
\Scal_V(f,K;q)=\sum_{n\geqslant 1}\lambda_f(n)K(n)V\left(\frac{n}{q}\right),
\end{equation}
\begin{equation}\label{Correlation2}
\Scal_V(\omega,it,K;q)=\sum_{n\geqslant 1}\lambda_\omega(n,it)K(n)V\left(\frac{n}{q}\right),
\end{equation}
where $V$ is a smooth a compactly supported function on $\Rbf_+^*$. Assuming that $|K(n)|\leqslant M$ for every $n\in\Zbf$, we obtain by Cauchy-Schwarz inequality and \eqref{Ramanujan-Petersson-Average},
\begin{equation}\label{TrivialBound}
\Scal_V(\omega,it,K;q),\Scal_V(f,K;q)\ll Mq^{1+\varepsilon},
\end{equation}
with an implied constant depending only on $V$ and the spectral parameter $t_f$ and this bound can be seen as the trivial one. Theorem \ref{Theorem2} improves on \eqref{TrivialBound} with a power saving in the $q$-aspect, namely 
$$\Scal_V(\omega,it,K;q),\Scal_V(f,K;q)\ll q^{1-\frac{1}{16}+\varepsilon},$$
for any $\varepsilon>0$ and with an implied constant depending on $\varepsilon, V,t_f$ (or $t$) and controlled by some invariant of $K$, called the conductor (see Definition \eqref{DefinitionConductor}). As in \cite[Definition 1.1]{twists}, we make the following definition about the test function $V$.
\begin{defi}[Condition $V(C,P,Q)$]\label{DefinitionVCQ} Let $P>0$ and $Q\geqslant 1$ be real numbers and let $C = (C_\nu)_{\nu\geqslant 0}$ be a sequence of non-negative real numbers. A smooth compactly supported function $V$ on $\Rbf$ satisfies Condition $(V (C, P, Q))$ if
\begin{enumerate}
\item [$(1)$] The support of $V$ is contained in the dyadic interval $[P,2P]$;
\item [$(2)$] For all $x>0$ and all integer $\nu\geqslant 0$, we have the inequality 
$$\left|x^\nu V^{(\nu)}(x)\right|\leqslant C_\nu Q^\nu.$$
\end{enumerate}
\end{defi}
In practice, the test function $V$ is not compactly supported, but rather in the Schwartz class. We give here a simple Corollary of Theorem \ref{Theorem2}.
%%%%%%%%%%%%%%%%%%%%%%%%%%%%%%%%%%%%%%%%%%%%%%%%%%%%%%%%%%%%%%%%%%%%%%%%%%%%%%%%%%%%%%%%%%
%%%%%%%%%%					COROLLARY FOR SCHWARTZ FUNCTIONS					  %%%%%%%%%%
%%%%%%%%%%%%%%%%%%%%%%%%%%%%%%%%%%%%%%%%%%%%%%%%%%%%%%%%%%%%%%%%%%%%%%%%%%%%%%%%%%%%%%%%%%
\begin{cor}\label{CorollarySchwartz} Let $q>2$ be a prime number, $\omega$ a Dirichlet character of modulus $q$, $f$ a primitive Hecke cusp form of level $q$, nebentypus $\omega$ and spectral parameter $t_f$. Let $K$ be a non Fourier-exceptional isotypic trace function of conductor $\cond(K)$. Let $Q\geqslant 1$, $C=(C_\nu)_{\nu}$ a sequence of non-negative real numbers and $V$ a smooth function on $\Rbf$ with the property that for any $A>0$,
\begin{equation}\label{HypothesisCorollary}
V(x)\ll_A \frac{1}{(1+|x|)^A} \ \ \mathrm{and} \ \ \left|x^\nu V^{(\nu)}(x)\right|\leqslant C_\nu Q^\nu, \ \ \nu\geqslant 0.
\end{equation}
Then for every $X>0$ and every $\varepsilon>0$, we have
$$\sum_{n\geqslant 1}\lambda_f(n)K(n)V\left(\frac{n}{X}\right)\ll_{C,\varepsilon} (qX)^\varepsilon(1+|t_f|)^A\cond(K)^sXQ^2\left(1+\frac{q}{X}\right)^{1/2}q^{-1/16},$$
$$\sum_{n\geqslant 1}\lambda_\omega(n,it)K(n)V\left(\frac{n}{X}\right)\ll_{C,\varepsilon}(qX)^\varepsilon(1+|t|)^A \cond(K)^s XQ^2\left(1+\frac{q}{X}\right)^{1/2}q^{-1/16},$$
\end{cor}
\begin{proof}
We consider the cuspidal case since the other is completely similar. Applying a partition of unity to $[1,\infty)$ leads to the decomposition
$$\sum_{n\geqslant 1}\lambda_f(n)K(n)V\left(\frac{n}{X}\right)=\sum_N\sum_{n\geqslant 1}\lambda_f(n)K(n)V\left(\frac{n}{X}\right)W\left(\frac{n}{N}\right),$$
where $W$ is a smooth and compactly supported function on $(1/2,2)$ satisfying $|x^{j}W^{(j)}(x)|\leqslant \tilde{c_j}$ for some sequence $\tilde{c}=(\tilde{c_j})$ of non-negative real numbers and $N$ runs over real numbers of the form $2^{i}$ with $i\geqslant 0$. Since $V$ has fast decay at infinity, we can focus on the contribution of $1\leqslant
N\leqslant q^\varepsilon X$ at the cost of an error of size $O(q^{-10})$, (say). Hence, we just need to bound $O(\log(qX))$ sums of the form
$$\sum_{n\geqslant 1}\lambda_f(n)K(n)V\left(\frac{n}{X}\right)W\left(\frac{n}{N}\right).$$
By Mellin inversion formula, we have for any $\varepsilon>0$
$$\sum_{n\geqslant 1}\lambda_f(n)K(n)V\left(\frac{n}{X}\right)W\left(\frac{n}{N}\right)=\frac{1}{2\pi i}\int_{(\varepsilon)}\left(\frac{X}{N}\right)^s\widetilde{V}(s)\left(\sum_{n\geqslant 1}\lambda_f(n)K(n)W_s\left(\frac{n}{N}\right)\right)ds,$$
where the function $W_s(x) :=x^{-s}W(x)$ satisfies 
\begin{equation}\label{QW}
x^jW_s^{(j)}(x) \ll_{\tilde{c},j} \left(1+|s|\right)^j.
\end{equation}
For fixed $s$ with $\Re e (s)=\varepsilon$, we apply Theorem \ref{Theorem2} to the inner sum with the function 
$V(x)=W_s(xq/N)$ which satisfies condition $V(\tilde{C},N/q,1+|s|)$ for some $\tilde{C}$ depending on $\tilde{c}$, obtaining the bound
$$(1+|t_f|)^A\cond(K)^sq^{\frac{1}{2}-\frac{1}{16}+\varepsilon}\left(\frac{X}{N}\right)^\varepsilon\int_{(\varepsilon)}|\widetilde{V}(s)|(N(1+|s|))^{1/2}\left(\frac{N}{q}+1+|s|\right)^{1/2}ds.$$
Using the fact that the Mellin transform $\widetilde{V}(s)$ satisfies
$$\widetilde{V}(s)\ll \left(\frac{Q}{1+|s|}\right)^B,$$
for every $B>0$ with an implied constant depending on $B$ and $\Re e (s)$, we see that we can restrict the integral to $|s|\leqslant q^\varepsilon Q$. Hence replacing $1+|s|$ by its maximal value, maximizing over $N\leqslant q^\varepsilon X$ and average trivially over $|s|\leqslant q^\varepsilon Q$ in the integral yields the desire result.
\end{proof}

%%%%%%%%%%%%%%%%%%%%%%%%%%%%%%%%%%%%%%%%%%%%%%%%%%%%%%%%%%%%%%%%%%%%%%%%%%%%%%%%%%%%%%%%%%%%%%%%%%%%%%%%%%%%%%%%%%%%%%%%%%%%%%%%%%%%%%%%%%%%%%%%%%%%%%%%%%%%%%%%%%%%%%%%%%%%%%%%%%%%%%%%%%%%%%%%%%%%%%%%%%%%%%%%%%%%%%%%%%%%%%%%%%%%%%%%%%%%%%%%%%%%%%%%%%%%%%%%%%%%%%%%%%%%%
%%%%%%%%%%				GOOD FUNCTIONS AND CORRELATING MATRICES				%%%%%%%%%%%%
%%%%%%%%%%%%%%%%%%%%%%%%%%%%%%%%%%%%%%%%%%%%%%%%%%%%%%%%%%%%%%%%%%%%%%%%%%%%%%%%%%%%%%%%%%%%%%%%%%%%%%%%%%%%%%%%%%%%%%%%%%%%%%%%%%%%%%%%%%%%%%%%%%%%%%%%%%%%%%%%%%%%%%%%%%%%%%%%%%%			
\subsection{Good functions and correlating matrices} We follow the same structure as in \cite{twists}. We forget all the specific properties that $K$ might have and proceed abstractly in order to deal with a great level of generality. Therefore we consider the problem of bounding the sum $\Scal_V (f, K; p)$ for $K : \Zbf \longrightarrow \Cbf$ a general $q$-periodic function, assuming only that we know that $||K||_\infty\leqslant M$ for some $M$ that we think as fixed. The strategy is to estimate an amplified second moment of $\Scal_V(g,K;q)$ for $g$ varying in a basis of cusp forms of level $q$ and nebentypus $\omega$. After completing the spectral sum, applying Kuznetsov-Petersson and Poisson formula, we have to confront some sums that we call $\mathit{twisted \ correlation \ sums}$, which we now define.

\vspace{0.2cm}

We denote by $\widehat{K}$ or $\FT(K)$ the normalized Fourier transform of $K$ given by
\begin{equation}\label{DefinitionFourier}
\FT(K)(y)=\widehat{K}(y)=\frac{1}{q^{1/2}}\sum_{x\in\Fbf_q}K(x)e\left(\frac{xy}{q}\right),\end{equation}
and we let $\PGL_2(\Fbfbar_q)$ acts on the projective line $\mathbf{P}^1(\Fbfbar_q)$ by fractional linear transformations.
%%%%%%%%%%%%%%%%%%%%%%%%%%%%%%%%%%%%%%%%%%%%%%%%%%%%%%%%%%%%%%%%%%%%%%%%%%%%%%%%%%%%%%%%%
%%%%%%%%%%					DEFINITION CORRELATION SUM						%%%%%%%%%%
%%%%%%%%%%%%%%%%%%%%%%%%%%%%%%%%%%%%%%%%%%%%%%%%%%%%%%%%%%%%%%%%%%%%%%%%%%%%%%%%%%%%%%%%%
\begin{defi}[Twisted correlation sum]\label{DefinitionCorrelation} Let $\gamma=\left(\begin{smallmatrix}
a & b \\ c & d
\end{smallmatrix}\right)\in\PGL_2(\Fbf_q)$. For $\omega$ a multiplicative character modulo $q$ and $K : \Fbf_q \longrightarrow \Cbf$, we define the twisted correlation sum $\Ccal(K,\omega;\gamma)$ by 
$$\Ccal(K,\omega;\gamma) := \sum_{\substack{z\in\Fbf_q \\ z\neq -d/c}}\omegabar(cz+d)\widehat{K}(\gamma z)\overline{\widehat{K}(z)}.$$
\end{defi}
\begin{remq} Note that for $\gamma\in\PGL_2(\Fbf_q)$, $\Ccal(K,\omega;\gamma)$ is well defined up to multiplication by $\omega(-1)\in\{-1,+1\}$. This is in fact not a problem since only the complex modulus $|\Ccal(K,\omega;\gamma)|$ will be considered later (see for example Definition \ref{Definition(q,M)-good}). We also mention that unlike the definition of correlation sum that the authors made for the original Theorem (see \cite[eq. (1.10)]{twists}), we have the presence here of a twist by the nebentypus of the modular form $f$. This is because the Kloosterman sums that we obtain after the application of Kuznetsov trace formula are also twsited by $\omega$.
\end{remq}
Assuming $||K||_\infty\leqslant M$, Cauchy-Schwarz and Parseval identity leads to 
\begin{equation}\label{Parseval}
|\Ccal(K,\omega;\gamma)|\leqslant M^2q.
\end{equation}
Of course this bound will not be satisfactory. According to the square root cancellation philosophy, one expects bound of the form $|\Ccal(K,\omega;\gamma)|\leqslant Mq^{1/2}$ and that the $\gamma'$s for which this estimate fails is well controlled. By this, we mean that this set of matrices (called correlation matrices) is nicely structured and rather small. Before giving a precise definition of correlation matrices and good functions, we introduce the following notations concerning algebraic subgroups of $\PGL_2$.
\begin{enumerate}
\item We denote by $\Brm\subset \PGL_2$ the subgroup of upper-triangular matrices, i.e. the stabilizer of $\infty\in\Pbfq ;$
\item We write $w=\left( \begin{smallmatrix} 0 & 1 \\ 1 & 0 \end{smallmatrix}\right)$ for the Weyl element, so that $\Brm w$ (resp. $w\Brm$) is the set of matrices mapping $0$ to $\infty$ (resp. $\infty$ to $0$);
\item We denote by $\PGL_{2,par}$ the subset of parabolic matrices in $\PGL_2$, i.e. which have a single fixed point in $\Pbfq;$ 
\item Given $x\neq y$ in $\Pbfq$, the pointwise stabilizer of $x$ and $y$ is denoted $\Trm^{x,y}$ (this is a maximal torus), and its normalizer in $\PGL_2$ (or the stabilizer of the set $\{x,y\}$) is denoted $\Nrm^{x,y}$.
\end{enumerate}
%%%%%%%%%%%%%%%%%%%%%%%%%%%%%%%%%%%%%%%%%%%%%%%%%%%%%%%%%%%%%%%%%%%%%%%%%%%%%%%%%%%%%%%%%%
%%%%%%%%%%					 DEFINITION OF GOOD FUNCTIONS				      %%%%%%%%%%
%%%%%%%%%%%%%%%%%%%%%%%%%%%%%%%%%%%%%%%%%%%%%%%%%%%%%%%%%%%%%%%%%%%%%%%%%%%%%%%%%%%%%%%%%%
\begin{defi}[Good functions]\label{Definition(q,M)-good} Let $q$ be a prime, $K:\Fbf_q\longrightarrow \Cbf$ an arbitrary function and $\omega$ a Dirichlet character modulo $q$. Let $M\geqslant 1$ be such that $||K||_\infty\leqslant M$. 
\begin{enumerate}
\item [$(1)$] We define
$$\Gbf_{K,\omega,M}:=\left\{ \gamma\in\PGL_2(\Fbf_q) \ | \ |\Ccal(K,\omega;\gamma)|> Mq^{1/2}\right\},$$
the set of $M$-$\mathbf{correlation \ matrices}$.
\item [$(2)$] We say that $K$ is $(q,M)$-good if $\Gbf_{K,\omega,M}\cap \PGL_{2,par}=\emptyset$ and there exists at most $M$ pairs $(x_i,y_i)$ of distinct elements in $\mathbf{P}^1(\Fbfbar_q)$ such that 
$$\Gbf_{K,\omega,M}=\Gbf_{K,\omega,M}^b\cup \Gbf_{K,\omega,M}^t\cup \Gbf_{K,\omega,M}^w,$$
where
$$\Gbf_{K,\omega,M}^b\subset \Brm(\Fbf_q)\cup \Brm(\Fbf_q)w\cup w\Brm(\Fbf_q),$$
$$\Gbf_{K,\omega,M}^t\subset\bigcup_{i}\Trm^{x_i,y_i}(\Fbf_q), \ \ \Gbf_{K,\omega,M}^w\subset \bigcup_i\left(\Nrm^{x_i,y_i}-\Trm^{x_i,y_i}\right)(\Fbf_q).$$
\end{enumerate}
\end{defi}
If $K$ is a $(q,M)$-good function, it turns out that the set of matrices $\gamma$ obtained form the amplification method does not intersect the set $\Gbf_{K,\omega,M}$ of correlating matrices in a too large set. More precisely, the technical result is the following generalization of \cite[Theorem 1.9]{twists}.
%%%%%%%%%%%%%%%%%%%%%%%%%%%%%%%%%%%%%%%%%%%%%%%%%%%%%%%%%%%%%%%%%%%%%%%%%%%%%%%%%%%%%%%%%%
%%%%%%%%%%							THEOREM (q,M)-GOOD						%%%%%%%%%%%
%%%%%%%%%%%%%%%%%%%%%%%%%%%%%%%%%%%%%%%%%%%%%%%%%%%%%%%%%%%%%%%%%%%%%%%%%%%%%%%%%%%%%%%%%%
\begin{theorem}\label{Theorem(q,M)-good} Let $q>2$ be a prime number, $f$ a primitive Hecke cusp form of level $q$ and nebentypus $\omega$. Let $V$ be a function satisfying $V(C,P,Q)$. Let $K : \Fbf_q\longrightarrow\Cbf$ be a $(q,M)$-good function with $||K||_\infty\leqslant M$. Then there exists an absolute constant $A\geqslant 1$ such that 
$$\Scal_V(f,K;q)\ll_{\delta,C} (1+|t_f|)^AM^{3/2}q^{1-\delta}(PQ)^{1/2}(P+Q)^{1/2},$$
$$\Scal_V(\omega,it,K;q)\ll_{\delta,C}(1+|t|)^A M^{3/2}q^{1-\delta}(PQ)^{1/2}(P+Q)^{1/2},$$
for any $\delta<1/16$.
\end{theorem}

%%%%%%%%%%%%%%%%%%%%%%%%%%%%%%%%%%%%%%%%%%%%%%%%%%%%%%%%%%%%%%%%%%%%%%%%%%%%%%%%%%%%%%%%%%%%%%%%%%%%%%%%%%%%%%%%%%%%%%%%%%%%%%%%%%%%%%%%%%%%%%%%%%%%%%%%%%%%%%%%%%%%%%%%%%%%%%%%%%%%%%%%%%%%%%%%%%%%%%%%%%%%%%%%%%%%%%%%%%%%%%%%%%%%%%%%%%%%%%%%%%%%%%%%%%%%%%%%%%%%%%%%%%%%%%
%%%%%%%%%%					TRACE FUNCTIONS OF L-ADIC SHEAVES					%%%%%%%%%%%%%%%%%%%%%%%%%%%%%%%%%%%%%%%%%%%%%%%%%%%%%%%%%%%%%%%%%%%%%%%%%%%%%%%%%%%%%%%%%%%%%%%%%%%%%%%%%%%%%%%%%%%%%%%%%%%%%%%%%%%%%%%%%%%%%%%%%%%%%%%%%%%%%%%%%%%%%%%%%%%%%%%%%%%%%%%%%%%%%%%%%%%%%%%%%%%%%%%%%%%%%%%%%%%%%%%%%%%%%%%%%%%%%%%%%%%%%%%%%%%%%%%%%%%%%%%%%%%%%%%%%%%

\subsection{Trace functions of $\ell$-adic sheaves}\label{Sectionl-adic} The functions to which we will apply Theorem \ref{Theorem(q,M)-good} are called $\mathit{trace \ functions}$ modulo $q$, which we now define formally. 

\vspace{0.1cm}

Let $q$ be a prime number and $\ell\neq q$ be an auxiliary prime. To any constructible $\Qlbar$-sheaf $\Fcal$ on $\Abfq$ and any point $x\in\mathbf{A}^1(\Fbf_q)$, we have a linear action of a geometric Frobenius $\Frm_x$ acting on a finite dimensional $\Qlbar$-vector space $\Fcal_{\xbar}$ (we can take $\Fcal_{\xbar}=\Fcal_{\etabar}$ if $\Fcal$ is lisse at $x$, where $\etabar$ is the geometric generic point \cite[§ 4.4]{katz80}). We can thus consider the trace $\Tr(\Frm_x | \Fcal_{\xbar})$. Because this trace takes value in $\Qlbar$, we also fix a field isomorphism 
$$\iota : \Qlbar\overset{\simeq}{\longrightarrow} \Cbf,$$
and we consider functions of the shape
\begin{equation}\label{DefinitionK(x)}
K(x)=\iota\left(\Tr(\Frm_x | \Fcal_{\xbar})\right),
\end{equation}
as defined in \cite[(7.3.7)]{katz90}.

\begin{defi}[Trace sheaves]\label{DefinitionTracesheaf} 1) A constructible $\Qlbar$-sheaf $\Fcal$ on $\Abfq$ is called a $\mathit{trace \ sheaf}$ if it is a middle extension sheaf (in the sens of \cite[Section 1]{inverse}) whose restriction on any non empty open subset $U\subset\Abfq$ where $\Fcal$ is lisse is pointwise pure of weight zero.

\vspace{0.1cm}

\noindent 2) A trace sheaf is called a $\mathit{Fourier \ trace \ sheaf}$ if in addition, it is a Fourier sheaf in the sense of \cite[Definition 8.2.2]{katz88}.

\vspace{0.2cm}

\noindent 3) We say that $\Fcal$ is an $\mathit{isotypic \ trace \ sheaf}$ if it is a Fourier trace sheaf and if for every non empty open subset $U$ as in (1), the associated $\ell$-adic representation $$\pi_1(U\otimes_{\Fbf_q}\Fbfbar_q,\etabar)\longrightarrow \GL(\Fcal_{\etabar}),$$ of the geometric etale fondamental group is an isotypic representation \cite[Chapter 2]{katz88}. We define similarly an $\mathit{irreducible \ trace \ sheaf}$.
\end{defi}

\begin{defi} Let $q$ be a prime number. A function $K : \Fbf_q\longrightarrow\Cbf$ is called a $\mathit{trace \ function}$ (resp. $\mathit{Fourier \ trace \ function, \ isotypic \ trace \ function}$) if there exists a trace sheaf (resp Fourier trace sheaf, isotypic trace sheaf) $\Fcal$ such that $K$ is given by \eqref{DefinitionK(x)}.
\end{defi}
There is an important invariant which measure the complexity of a trace function which we define now.
\begin{defi}[Conductor]\label{DefinitionConductor} Let $\Fcal$ be a constructible $\Qlbar$-sheaf on $\Abfq$ with $n(\Fcal)$ singularities in $\Pbfq$. The $\mathit{conductor}$ of $\Fcal$ is the integer defined by
$$\cond(\Fcal):= \rank(\Fcal)+n(\Fcal)+\sum_{x\in \Pbfq}\Swan_x(\Fcal),$$
where $\Swan_x(\Fcal)=0$ if $\Fcal$ is lisse at $x$ (see for example \cite[Section 4.6]{katz80}). If $K$ is a trace function, the conductor $\cond(K)$ of $K$ is the smallest conductor of a trace sheaf $\Fcal$ with trace function $K$. 
\end{defi}

%%%%%%%%%%%%%%%%%%%%%%%%%%%%%%%%%%%%%%%%%%%%%%%%%%%%%%%%%%%%%%%%%%%%%%%%%%%%%%%%%%%%%%%%%%
%%%%%%%%%%%%%%%%%%%%%%%%%%%%%%%%%%%%%%%%%%%%%%%%%%%%%%%%%%%%%%%%%%%%%%%%%%%%%%%%%%%%%%%%%%
%%%%%%%%%%%%%%%%%%%%%%%%%%%%%%%%%%%%%%%%%%%%%%%%%%%%%%%%%%%%%%%%%%%%%%%%%%%%%%%%%%%%%%%%%%
%%%%%%%%%%						EXAMPLE OF TRACE FUNCTIONS				      %%%%%%%%%%
%%%%%%%%%%%%%%%%%%%%%%%%%%%%%%%%%%%%%%%%%%%%%%%%%%%%%%%%%%%%%%%%%%%%%%%%%%%%%%%%%%%%%%%%%%
%%%%%%%%%%%%%%%%%%%%%%%%%%%%%%%%%%%%%%%%%%%%%%%%%%%%%%%%%%%%%%%%%%%%%%%%%%%%%%%%%%%%%%%%%%
%%%%%%%%%%%%%%%%%%%%%%%%%%%%%%%%%%%%%%%%%%%%%%%%%%%%%%%%%%%%%%%%%%%%%%%%%%%%%%%%%%%%%%%%%%

\subsection{Examples of trace functions} We give in this section basic examples of trace functions that we shall need later.

%%%%%%%%%%%%%%%%%%%%%%%%%%%%%%%%%%%%%%%%%%%%%%%%%%%%%%%%%%%%%%%%%%%%%%%%%%%%%%%%%%%%%%%%%%
%%%%%%%%%%%%%%%%%%%%%%%%%%%%%%%%%%%%%%%%%%%%%%%%%%%%%%%%%%%%%%%%%%%%%%%%%%%%%%%%%%%%%%%%%%
%%%%%%%%%%					KUMMER AND ARTIN-SCHREIER SHEAVES				%%%%%%%%%%%%
%%%%%%%%%%%%%%%%%%%%%%%%%%%%%%%%%%%%%%%%%%%%%%%%%%%%%%%%%%%%%%%%%%%%%%%%%%%%%%%%%%%%%%%%%%%%%%%%%%%%%%%%%%%%%%%%%%%%%%%%%%%%%%%%%%%%%%%%%%%%%%%%%%%%%%%%%%%%%%%%%%%%%%%%%%%%%%%%%%%%
\subsubsection{Kummer and Artin-Schreier sheaves}\label{SectionKummerArtin} Let $\psi : \Fbf_q\longrightarrow \Cbf^\times$ be a non trivial additive character and $f=f_1/f_2\in\Fbf_q(X)$ be a rational function with all poles of order $<q$. Then for every prime $\ell\neq q$, there exists a trace sheaf $\Lcal_{\psi(f)}$ of rank 1 on $\Abfq$ with the following properties :
\begin{enumerate}
\item[$(1)$] $\iota\left(\Tr\left(\Frm_x | (\Lcal_{\psi(f)})_{\xbar}\right)\right)=\psi(f(x))$ for every $x\in\Fbf_q$ (under the convention that $\psi(\infty)=0$);
\item[$(2)$] $\Lcal_{\psi(f)}$ is lisse on $\Abfq-\{\mathrm{poles \ of \ }f\}$ with $\Swan_x$ equal to the order of the pole at $x$;
\item[$(3)$] We have $\cond (\Lcal_{\psi(f)})\leqslant 1+2\deg(f_2)$.
\end{enumerate}
Now let $\omega : \Fbf_q^\times\longrightarrow \Cbf^\times$ be a non-trivial multiplicative character of order $d$ and $f=f_1/f_2\in \Fbf_q(X)$ a rational function with zeros and poles of order not divisible by $d$. Then for every prime $\ell\neq q$, there exists a trace sheaf $\Lcal_{\omega(f)}$ of rank 1 on $\Abfq$ such that 
\begin{enumerate}
\item[$(1)$] $\iota\left(\Tr\left(\Frm_x | (\Lcal_{\omega(f)})_{\xbar}\right)\right)=\omega(f(x))$ for every $x\in\Fbf_q$ (under the convention that $\omega(0)=\omega(\infty)=0$);
\item[$(2)$] $\Lcal_{\omega(f)}$ is lisse on $\Abfq-\{\mathrm{zeros \ and \ poles \ of \ }f\}$ with $\Swan_x=0$ for every $x\in\mathbf{P}^1(\Fbfbar_q)$;
\item[$(3)$] $\cond(\Lcal_{\omega(f)})\leqslant 1+\deg(f_1)+\deg(f_2)$.
\end{enumerate}

\vspace{0.2cm}

\noindent The sheaf $\Lcal_{\psi(f)}$ is called an $\mathit{Artin}$-$\mathit{Schreier \ sheaf}$ and we say that $\Lcal_{\psi(f)}$ has $\mathit{wild \ ramification}$ at $\infty$ because $\Swan_\infty>0$. The sheaf $\Lcal_{\omega(f)}$ is called a $\mathit{Kummer \ sheaf}$ and we say that $\Lcal_{\omega(f)}$ is tamely ramified at the singularities because $\Swan_x=0$. We refer to \cite[Sommes trig.]{sga4} for the construction and properties of such sheaves and also to \cite[Proposition 1.44-1.45]{perret}.

%%%%%%%%%%%%%%%%%%%%%%%%%%%%%%%%%%%%%%%%%%%%%%%%%%%%%%%%%%%%%%%%%%%%%%%%%%%%%%%%%%%%%%%%%%%%%%%%%%%%%%%%%%%%%%%%%%%%%%%%%%%%%%%%%%%%%%%%%%%%%%%%%%%%%%%%%%%%%%%%%%%%%%%%%%%%%%%%%%%%
%%%%%%%%%%						KLOOSTERMAN SHEAVES							%%%%%%%%%%
%%%%%%%%%%%%%%%%%%%%%%%%%%%%%%%%%%%%%%%%%%%%%%%%%%%%%%%%%%%%%%%%%%%%%%%%%%%%%%%%%%%%%%%%%%%%%%%%%%%%%%%%%%%%%%%%%%%%%%%%%%%%%%%%%%%%%%%%%%%%%%%%%%%%%%%%%%%%%%%%%%%%%%%%%%%%%%%%%%%%
\subsubsection{Kloosterman sheaves} Let $k\geqslant 2$ be an integer, $\omega_1,...,\omega_k$ be multiplicative characters on $\Fbf_q^\times.$ The twisted rank $k$ Kloosterman sum $\Kl_k(\omega_1,...,\omega_k;q)$ is the function on $\Fbf_q^\times$ defined by 
\begin{equation}\label{DefinitionKloosterman}
 \Kl_k(a,\omega_1,...,\omega_k;q):=\frac{1}{q^{\frac{k-1}{2}}}\sum_{\substack{x_1,...,x_k\in\Fbf_q^\times \\ x_1\cdots x_k=a}}\omega_1(x_1)\cdots \omega_k(x_k)e\left(\frac{x_1+...+x_k}{q}\right),
\end{equation}
for every $a\in\Fbf_q^\times$. If $\omega_1=...=\omega_k=1$, we write $\Kl_k(a;q)$ instead of $\Kl_k(a,1,...,1:q)$. The main result is the existence of Kloosterman sheaves and it is due to Deligne \cite[Theorem 4.1.1]{katz88}.
%%%%%%%%%%%%%%%%%%%%%%%%%%%%%%%%%%%%%%%%%%%%%%%%%%%%%%%%%%%%%%%%%%%%%%%%%%%%%%%%%%%%%%%%%%
%%%%%%%%%%				THEOREM FOR KLOOSTERMAN SHEAVES						%%%%%%%%%%
%%%%%%%%%%%%%%%%%%%%%%%%%%%%%%%%%%%%%%%%%%%%%%%%%%%%%%%%%%%%%%%%%%%%%%%%%%%%%%%%%%%%%%%%%%
\begin{theorem}[Kloosterman sheaves] For every prime $\ell\neq q$, there exists a constructible $\Qlbar$-sheaf on $\Abfq$, denoted by $\Klcal_k(\omega_1,...,\omega_k;q)$ (or simply $\Klcal$), of rank $k$ satisfying the following properties :
\begin{enumerate}
\item[$(1)$] For every $a\in\Fbf_q^\times,$ we have the equality 
$$\iota\left(\Tr\left(\Frm_a | (\Klcal)_{\overline{a}}\right)\right)=(-1)^{k-1}\Kl_k(a,\omega_1,...,\omega_k;q);$$
\item[$(2)$] $\Klcal$ is geometrically irreducible, lisse on $\Gbf_{m,\Fbf_q}$ and pointwise pure of weight zero;
\item[$(2)$] $\Klcal$ has wild ramification at $\infty$ with $\Swan_\infty(\Klcal)=1$, tamely ramified at $0$ and has conductor $k+3$.
\end{enumerate}
In other words, since the rank is $\geqslant 2$, $\Klcal$ is an irreducible trace sheaf in the sense of Definition \ref{DefinitionTracesheaf} $(3)$.
\end{theorem}
\begin{cor} For every $a\in\Fbf_q^\times,$ we have the sharp bound
\begin{equation}\label{WeilBoundCorollary}
|\Kl_k(a,\omega_1,...,\omega_k;q)|\leqslant k.
\end{equation}
\end{cor}
It will be convenient in Section \ref{SectionProof1}, specially because of the Poisson summation and Fourier inversion formula, to present $\Kl$ as a Fourier transform of a function defined in $\Fbf_q$. For this, we let 
$$\Kl_1(a,\omega_1;q):= \omega_1(a)e\left(\frac{a}{q}\right),$$
and we see that for any $k\geqslant 2$ and $a\in\Fbf_q^\times$, $\Kl_k(a,\omega_1,...,\omega_k;q)$ is given by the formula
\begin{equation}\label{Kloosterman-Fourier}
\Kl_k(a,\omega_1,...,\omega_k;q)=\omega_k(a)\FT\Big(\Fbf_q \ni x\mapsto \omega_k(x)\Kbf_{k-1}(x,\omega_1,...,\omega_{k-1};q)\Big)(a),
\end{equation}
where the function $\Kbf_{k-1}$ is defined by
\begin{equation}\label{DefinitionKbf}
\Kbf_{k-1}(x,\omega_1,...,\omega_{k-1};q):= \left\{ \begin{array}{ccc} \Kl_{k-1}(\xbar,\omega_1,...,\omega_{k-1};q) & \ifm & x\in \Fbf_q^\times, \\
 & & \\
0 & \ifm & x=0.
\end{array} \right.
\end{equation}
\begin{remq}\label{RemarkExtensionByZero}
There are several ways to extend the function $\Kl_k$ to $a=0$. One can choose for example the extension by zero. We choose here the middle extension, i.e. that $\Kl_k(0,\omega_1,...,\omega_k;q)$ coincides with the trace of the Frobenius at $x=0$. It is a deep result of Deligne that the estimate \eqref{WeilBoundCorollary} remains valid for $a=0$ (see \cite[(1.8.9)]{conjecture2}).
\end{remq}

%%%%%%%%%%%%%%%%%%%%%%%%%%%%%%%%%%%%%%%%%%%%%%%%%%%%%%%%%%%%%%%%%%%%%%%%%%%%%%%%%%%%%%%%%%%%%%%%%%%%%%%%%%%%%%%%%%%%%%%%%%%%%%%%%%%%%%%%%%%%%%%%%%%%%%%%%%%%%%%%%%%%%%%%%%%%%%%%%%%%%%%%%%%%%%%%%%%%%%%%%%%%%%%%%%%%%%%%%%%%%%%%%%%%%%%%%%%%%%%%%%%%%%%%%%%%%%%%%%%%%%%%%%%%%%
%%%%%%%%%%			THE l-ADIC FOURIER TRANSFORM AND MOBIUS GROUP			%%%%%%%%%%
%%%%%%%%%%%%%%%%%%%%%%%%%%%%%%%%%%%%%%%%%%%%%%%%%%%%%%%%%%%%%%%%%%%%%%%%%%%%%%%%%%%%%%%%%%%%%%%%%%%%%%%%%%%%%%%%%%%%%%%%%%%%%%%%%%%%%%%%%%%%%%%%%%%%%%%%%%%%%%%%%%%%%%%%%%%%%%%%%%%%%%%%%%%%%%%%%%%%%%%%%%%%%%%%%%%%%%%%%%%%%%%%%%%%%%%%%%%%%%%%%%%%%%%%%%%%%%%%%%%%%%%%%%%%%%
\subsection{The $\ell$-adic Fourier transform and the Kummer-Möbius group}
There is a deep interpretation of the Fourier transform at the level of sheaves who was discovered by Deligne and developed by Laumon, specially in \cite{laumon}. To be precise, let $q$ be a prime, $\ell\neq q$ another prime, $\Fcal$ a Fourier sheaf and $\psi : \Fbf_q\longrightarrow \Qlbar^\times$ a non-trivial additive character. Then there exists a Fourier sheaf $\Gcal_\psi$ with the property that 
$$\Tr\left(\Frm_x | (\Gcal_\psi)_{\xbar}\right) = -\frac{1}{q^{1/2}}\sum_{y\in \Fbf_q^\times}\Tr\left(\Frm_x | \Fcal_{\xbar}\right)\psi(xy),$$
for every $x\in \Fbf_q$ and where $q^{1/2}$ is the choice of a square root of $q$ in $\Qlbar^\times$ which maps to $q^{1/2}$ under the fixed isomorphism $\iota$. In particular,if we choose $\psi$ such that $\iota(\psi(x))=e\left(x/q\right)$ and $K$ is the trace function corresponding to $\Fcal$, we obtain 
$$\iota\left(\Tr\left(\Frm_x | \Gcal_{\xbar}\right)\right) = -\widehat{K}(x).$$
Moreover, the sheaf $\Gcal_\psi$ is geometrically isotypic (resp. geometrically irreducible) if and only if $\Fcal$ has this property \cite[Lemma 8.1]{twists}.

\vspace{0.2cm}

Finally, taking in account the twist by the character in the Definition \ref{DefinitionCorrelation}, we adapt the Definition of the Fourier-Möbius group given in \cite[Definition 1.15]{twists}, which is a geometric interpretation of the set of correlating matrices, as follows :
%%%%%%%%%%%%%%%%%%%%%%%%%%%%%%%%%%%%%%%%%%%%%%%%%%%%%%%%%%%%%%%%%%%%%%%%%%%%%%%%%%%%%%%%%%
%%%%%%%%%%					  DEFINITION KUMMER MOBIUS GROUP				     %%%%%%%%%%
%%%%%%%%%%%%%%%%%%%%%%%%%%%%%%%%%%%%%%%%%%%%%%%%%%%%%%%%%%%%%%%%%%%%%%%%%%%%%%%%%%%%%%%%%%
\begin{defi}\label{DefinitionKummerMobius} Let $q$ be prime number and $\Fcal$ an isotypic trace sheaf on $\Abfq$ with Fourier transform $\Gcal$ with respect to a non trivial additive character $\psi$. Let $\omega$ be a multiplicative character. The Kummer-Möbius group $\Gbf_{\Fcal,\omega}$ is the subgroup of $\PGL_2(\Fbf_q)$ defined by 
$$\Gbf_{\Fcal,\omega} :=\left\{\gamma=\left(\begin{matrix} a & b \\ c & d \end{matrix}\right) \in\PGL_2(\Fbf_q) \ | \ \Gcal\simeq_{\mathrm{geom}} \gamma^*\Gcal\otimes\Lcal_{\omegabar(cX+d)}\right\}.$$
\end{defi}
\begin{remq} Note that Definition \ref{DefinitionKummerMobius} makes sense in the sense that if $\gamma,\gamma'\in\GL_2(\Fbf_q)$ are equal in $\PGL_2(\Fbf_q)$, then $\gamma=\pm\mathrm{I}_2\gamma'$ and thus
$$\gamma^*\Gcal\otimes \Lcal_{\omegabar(cX+d)}=\gamma'^*\Gcal\otimes\Lcal_{\omegabar(-c'X-d')}\simeq_{\mathrm{geom}}\gamma'^*\Gcal\otimes\Lcal_{\omegabar(c'X+d')}\otimes\Lcal_{\omegabar(-1)}\simeq_{\mathrm{geom}} \gamma'^*\Gcal\otimes\Lcal_{\omegabar(c'X+d')}.$$
\end{remq}
The key property here is that $\Gbf_{\Fcal,\omega}$ is indeed a subgroup of $\PGL_2(\Fbf_q)$.
%%%%%%%%%%%%%%%%%%%%%%%%%%%%%%%%%%%%%%%%%%%%%%%%%%%%%%%%%%%%%%%%%%%%%%%%%%%%%%%%%%%%%%%%%%
%%%%%%%%%%						PROPOSITION SUBGROUP							%%%%%%%%%%
%%%%%%%%%%%%%%%%%%%%%%%%%%%%%%%%%%%%%%%%%%%%%%%%%%%%%%%%%%%%%%%%%%%%%%%%%%%%%%%%%%%%%%%%%%
\begin{proposition} $\Gbf_{\Fcal,\omega}$ is a subgroup of $\PGL_2(\Fbf_q)$.
\end{proposition}
\begin{proof}
Let $\Fscr$ be the set of geometric isomorphism classes of trace sheaves. To show that $\Gbf_{\Fcal,\omega}$ is a subgroup, it is enough to prove that the map
$\Fscr\times \PGL_2(\Fbf_q) \longrightarrow \Fscr$ given by 
\begin{equation}\label{DefinitionAction}
(\Gcal,\gamma) \longmapsto  \Gcal\cdot\gamma := \gamma^*\Gcal\otimes\Lcal_{\omegabar(cX+d)},\end{equation}
defines a right group action because $\Gbf_{\Fcal,\omega}$ will be the stabilizer of $\Gcal$. For this, we will use the fact that we have geometric isomorphisms (we use the notation $\simeq$ instead of $\simeq_{\mathrm{geom}}$)
\begin{equation}\label{GeometricIso}
\Lcal_{\omegabar(d)}\simeq \Qlbar \simeq \Lcal_{\omegabar(cX+d)}\otimes \Lcal_{\omega(cX+d)},
\end{equation}
where $\Qlbar$ denotes the constant sheaf. The first isomorphism implies that the identity matrix acts trivially. For the second part, note that for $\gamma_1=\left( \begin{smallmatrix} a & b \\ c & d \end{smallmatrix} \right)$ and $\gamma_2= \left( \begin{smallmatrix} a' & b' \\ c' & d' \end{smallmatrix} \right) \in \PGL_2(\Fbf_q)$, we have
$$\Lcal_{\omegabar(cX+d)}\simeq j(\gamma)^*\Lcal_\chi, \ \ j(\gamma):= \left( \begin{matrix}
0 & 1 \\ c & d 
\end{matrix}\right),$$
and 
$$j(\gamma_1\gamma_2)=\left(\begin{matrix} 0 & 1 \\ ca'+dc' & cb'+ dd' \end{matrix} \right) \ \ , \ \ j(\gamma_1)\gamma_2= \left(\begin{matrix} c' & d' \\ ca'+dc' & cb'+ dd' \end{matrix} \right).$$
Combining the above equality with the second isomorphism in \eqref{GeometricIso} leads to
$$j(\gamma_1\gamma_2)^*\Lcal_\omega\simeq (j(\gamma_1)\gamma_2)^*\Lcal_\omega\otimes j(\gamma_2)^*\Lcal_\omega.$$
Hence we obtain
\begin{alignat*}{1}
\Gcal\cdot (\gamma_1\gamma_2)\simeq & \ (\gamma_1\gamma_2)^*\Gcal\otimes j(\gamma_1\gamma_2)^*\Lcal_\omega \simeq \gamma_2^*\gamma_1^*\Gcal\otimes \gamma_2^*j(\gamma_1)^*\Lcal_\omega\otimes j(\gamma_2)^*\Lcal_\omega \\ \simeq & \ \gamma_2^*\Big(\gamma_1^*\Gcal\otimes j(\gamma_1)^*\Lcal_\omega\Big)\otimes j(\gamma_2)^*\Lcal_\omega \\
\simeq & \ (\Gcal\cdot\gamma_1)\cdot\gamma_2,
\end{alignat*}
which completes the proof of this Proposition.
\end{proof}
We will also need the following fact about the conductor of $\Fcal\cdot\gamma$.
%%%%%%%%%%%%%%%%%%%%%%%%%%%%%%%%%%%%%%%%%%%%%%%%%%%%%%%%%%%%%%%%%%%%%%%%%%%%%%%%%%%%%%%%%%
%%%%%%%%%%					LEMMA ABOUT THE CONDUCTOR						%%%%%%%%%%
%%%%%%%%%%%%%%%%%%%%%%%%%%%%%%%%%%%%%%%%%%%%%%%%%%%%%%%%%%%%%%%%%%%%%%%%%%%%%%%%%%%%%%%%%%
\begin{lemme}\label{LemmeCond} Let $\Fcal$ be a trace sheaf and $\gamma\in \PGL_2(\Fbf_q)$. Then 
$$\left|\cond(\Fcal\cdot\gamma)-\cond(\Fcal)\right|\leqslant 2.$$
\end{lemme}
\begin{proof}
Since the Kummer sheaves are of rank one and tamely ramified at the singularities (c.f. § \ref{SectionKummerArtin}), we have for any $x\in\Pbfq$ (by definition \eqref{DefinitionAction} of the action of $\gamma$ on $\Fcal$),  
$$\rank(\gamma^*\Fcal\otimes\Lcal_{\omegabar(cX+d)})=\rank(\gamma^*\Fcal)=\rank(\Fcal), \ \ \Swan_x(\Fcal\cdot \gamma)=\Swan_x(\Fcal),$$
see \cite[4.6 (iv)]{katz80}. Moreover, if $n(\Fcal)$ (resp $n(\Lcal_{\omegabar(cX+d)})$) denotes the number of singularities of $\Fcal$ (resp. of $\Lcal_{\omegabar(cX+d)}$), the tensor product $\gamma^*\Fcal\otimes\Lcal_{\omegabar(cX+d)}$ satisfies (see for example \cite[Proposition 1.23]{perret})
\begin{alignat*}{1}
n(\Fcal)-n(\Lcal_{\omegabar(cX+d)})= & \ n(\gamma^*\Fcal)-n(\Lcal_{\omegabar(cX+d)}) \\ \leqslant & \  n(\gamma^*\Fcal\otimes \Lcal_{\omegabar(cX+d)}) \\ \leqslant & \ n(\gamma^*\Fcal) + n( \Lcal_{\omegabar(cX+d)})= n(\Fcal)+n(\Lcal_{\omegabar(cX+d)}),
\end{alignat*}
which completes the proof since $n(\Lcal_{\omegabar(cX+d)})=0$ or $2$ depending on whether $c=0$ or not.
\end{proof}
Here is a key Lemma.
%%%%%%%%%%%%%%%%%%%%%%%%%%%%%%%%%%%%%%%%%%%%%%%%%%%%%%%%%%%%%%%%%%%%%%%%%%%%%%%%%%%%%%%%%%
%%%%%%%%%%					PROPOSITION ON THE INCLUSION						%%%%%%%%%%
%%%%%%%%%%%%%%%%%%%%%%%%%%%%%%%%%%%%%%%%%%%%%%%%%%%%%%%%%%%%%%%%%%%%%%%%%%%%%%%%%%%%%%%%%%
\begin{lemme}\label{PropositionInclusion} Let $q$ be a prime number , $\ell\neq q$ and $\Fcal$ an isotypic trace sheaf on $\Abfq$ with corresponding isotypic trace function $K$. Then for 
$$M\geqslant 8 \ \cond(\Fcal)^5+31 \ \cond(\Fcal)^6,$$
we have the inclusion
$$\Gbf_{K,\omega,M}\subset \Gbf_{\Fcal,\omega}.$$
\end{lemme}
\begin{proof} Let $\Gcal$ be the $\ell$-adic Fourier transform of $\Fcal$. The proof of this Lemma is a consequence of the Riemann hypothesis over finite field and we sketch here how it works.

As in \cite[Theorem 9.1]{twists}, the Lefschetz trace formula can be used to derive that for any $\gamma\notin \Gbf_{\Fcal,\omega}$, we have
$$\left|\Ccal(K,\omega;\gamma)\right|\leqslant M_1+M_2q^{1/2},$$
with $$M_1=\Big(n(\gamma^*\Gcal\otimes\Lcal_{\omegabar(cX+d)})+n(\Gcal)\Big)\rank(\Gcal)^2$$
and 
$$M_2=\dim\Hrm^1_c(U_\gamma\times_{\Fbf_q} \Fbfbar_q,\gamma^*\Gcal\otimes\Lcal_{\omegabar(cX+d)}\otimes\Gcal^{\vee})$$
and where $\Gcal^{\vee}$ is the dual sheaf, $\Hrm^1_c$ is the first étale cohomology group with compact support. Here $U_\gamma\subset \Abfq$ is an open subset where the sheaf $\gamma^*\Gcal\otimes\Lcal_{\omegabar(cX+d)}\otimes\Gcal^{\vee}$ is lisse and can be taken such that 
$$|(\mathbf{P}^1-U_\gamma)(\Fbfbar_q)|\leqslant n(\gamma^*\Gcal\otimes\Lcal_{\omegabar(cX+d)})+n(\Gcal^{\vee})\leqslant 2n(\Gcal)+2,$$
where we used Lemma \ref{LemmeCond} for the last inequality. Using \cite[Proposition 8.2]{twists} which is a consequence of the Grothendieck-Ogg-Schavarevich formula and again Lemma \ref{LemmeCond}, we obtain
$$\dim\Hrm^1_c(U_\gamma\times_{\Fbf_q} \Fbfbar_q,\gamma^*\Gcal\otimes\Lcal_{\omegabar(cX+d)}\otimes\Gcal^{\vee})\leqslant \rank(\Gcal)^2\left(5+2n(\Gcal)+2\cond(\Gcal)\right).$$
Now according to the bounds (see \cite[Proposition 8.2 (1), (8.4)-(8.5)]{twists})
\begin{equation}\label{BoundCondFourier}
\cond (\Gcal)\leqslant 10 \ \cond(\Fcal)^2, \ \ \rank(\Gcal)\leqslant \cond(\Fcal)^2, \ \ n(\Gcal)\leqslant 3 \ \cond(\Fcal),
\end{equation}
we obtain
$$M_1\leqslant (2n(\Gcal)+2)\rank(\Gcal)^2\leqslant 8 \ \cond(\Fcal)^5 \ \ \mathrm{and} \ \ M_2\leqslant 31 \ \cond(\Fcal)^6$$
and the result follows.
\end{proof}
Let $\Fcal$ be an isotypic trace sheaf with corresponding trace function $K$ and $N\geqslant 1$ with $\cond(\Fcal)\leqslant N$. Suppose that its Fourier-Möbius group $\Gbf_{\Fcal,\omega}$ satisfies condition $(2)$ in Definition \ref{Definition(q,M)-good} for some $M'$ with $|K|\leqslant M'$. Then by Proposition \ref{PropositionInclusion}, choosing $M\geqslant \max(39N^6,M')$ and we have the inclusion $\Gbf_{K,\omega,M}\subset\Gbf_{\Fcal,\omega}$, so $K$ is $(q,M)$-good and Theorem \ref{Theorem2} follows from Theorem \ref{Theorem(q,M)-good}. The following Proposition, which is a simple variant of \cite[Lemma 9.3]{twists} and follows from classification of Artin-Schreier sheaves, shows that $\Gbf_{\Fcal,\omega}$ contains parabolic elements if the Fourier transform $\Gcal$ contains in its irreducible component a sheaf of the form $\Lcal_{\psi}\cdot\gamma$.
%%%%%%%%%%%%%%%%%%%%%%%%%%%%%%%%%%%%%%%%%%%%%%%%%%%%%%%%%%%%%%%%%%%%%%%%%%%%%%%%%%%%%%%%%%
%%%%%%%%%%			     	LEMMA ON ARTIN-SCHREIER SHEAVES					%%%%%%%%%%
%%%%%%%%%%%%%%%%%%%%%%%%%%%%%%%%%%%%%%%%%%%%%%%%%%%%%%%%%%%%%%%%%%%%%%%%%%%%%%%%%%%%%%%%%%
\begin{lemme}\label{Classification-Artin-Schreier} Let $q>2$ be a prime and $\ell\neq q$ an auxiliary prime. Let $\psi$ be a non trivial $\ell$-adic additive character of $\Fbf_q$ and let $\Fcal=\oplus \Lcal_\psi$ be a finite direct sum of copies of $\Lcal_\psi$. Then for $\gamma\in \PGL_2(\Fbf_q)$, we have a geometric isomorphism $\Fcal\cdot\gamma \simeq \Fcal$ if and only if $\gamma \in \Urm^\infty(\Fbf_q)$, where for any $x\in \mathbf{P}^1(\Fbfbar_q)$, $\Urm^x$ denotes the unipotent radical of the Borel subgroup of $\PGL_2$ fixing $x$.
\end{lemme}
\begin{proof}
If $\gamma\in \Urm^\infty(\Fbf_q)$, then by definition of the action \eqref{DefinitionAction}, we have geometrically 
$$\Fcal\cdot \gamma\simeq \gamma^*\Fcal=\bigoplus \gamma^*\Lcal_\psi \simeq \bigoplus \Lcal_\psi = \Fcal.$$
Conversely, assume that $\Fcal\cdot \gamma \simeq \Fcal$. We show first that we also have $\Lcal_\psi\cdot\gamma \simeq \Lcal_\psi.$ Indeed, writing $\iota_i : \Lcal_\psi \hookrightarrow \Fcal$ (resp. $\mathrm{pr}_i : \Fcal\cdot\gamma \rightarrow \Lcal_\psi\cdot\gamma$) for the inclusion of (resp. the projection in) the $i$-th component, we can choose $j$ such that the morphism
$$\Lcal_\psi \overset{\iota_1}{\hookrightarrow}\Fcal \overset{\simeq}{\longrightarrow}\Fcal\cdot\gamma=\bigoplus \Lcal_\psi\cdot\gamma \overset{\mathrm{pr}_j}{\longrightarrow}\Lcal_\psi\cdot\gamma$$
is not the zero map. Since the rank is one on both side, we conclude that it is an isomorphism. Hence we obtain
$$\Lcal_\psi \simeq \gamma^*\Lcal_\psi\otimes \Lcal_{\omegabar(cX+d)}, \ \ \gamma=\left( \begin{matrix} a & b \\ c & d 
\end{matrix} \right).$$
The sheaf $\Lcal_\psi$ has wild ramification at $\infty$ and $\Lcal_{\omegabar(cX+d)}$ is tame at this point (see § \ref{SectionKummerArtin}), so we must have $\gamma\infty=\infty$, which implies that $c=0$ and thus $\Lcal_\psi\simeq \gamma^*\Lcal_\psi$. The conclusion follows now from the classification of Artin-Schreier sheaves.
\end{proof}
It turns out that if $\Gbf_{\Fcal,\omega,M}$ contains parabolic elements, then contrary to the original Theorem of \cite{twists}, the size of its intersection with the matrices constructed from the amplification method may be too large. We therefore formulate the last definition :
\begin{defi}\label{DefinitionExceptional} An isotypic trace sheaf $\Fcal$ is called \textit{Fourier-exceptional} if the irreducible component of its Fourier transform $\Gcal$ is of the form $\gamma^*\Lcal_\psi\otimes\Lcal_{\omegabar(cX+d)}$ for some $\gamma\in\PGL_2(\Fbf_q)$ and some (possibly trivial) additive character $\psi$.
\end{defi}
 With all these definitions, we can now formulate the last result which, together with Theorem \ref{Theorem(q,M)-good}, immediately implies Theorem \ref{Theorem2}.

\begin{theorem}\label{Theorem3}Let $q$ be a prime number, $N\geqslant 1$ and $\Fcal$ a non-exceptional isotypic trace sheaf, as in Definition \ref{DefinitionExceptional}, on $\Abfq$ with $\cond(\Fcal)\leqslant N$. Let $K$ be the corresponding isotypic trace function. Then there exists absolute constants $s\geqslant 1$ and $a\geqslant 1$ such that $K$ is $(q,aN^s)$-good as in Definition \ref{Definition(q,M)-good}.
\end{theorem}
%%%%%%%%%%%%%%%%%%%%%%%%%%%%%%%%%%%%%%%%%%%%%%%%%%%%%%%%%%%%%%%%%%%%%%%%%%%%%%%%%%%%%%%%%%
%%%%%%%%%%					PREUVE DU THEOREM 1.16							%%%%%%%%%%
%%%%%%%%%%%%%%%%%%%%%%%%%%%%%%%%%%%%%%%%%%%%%%%%%%%%%%%%%%%%%%%%%%%%%%%%%%%%%%%%%%%%%%%%%%
\begin{proof}By Proposition \ref{PropositionInclusion}, there exists $M\leqslant 39 N^6$ such that 
$$\Gbf_{K,\omega,M}\subset G:= \Gbf_{\Fcal,\omega},$$
which is a subgroup of $\PGL_2(\Fbf_q)$. If the order of $G$ is coprime with $q$, then we proceed as in the first paragraph of the proof of \cite[Theorem 1.14]{twists}.

\vspace{0.2cm}

We now show that under the non-exceptional assumption, the order of $G$ cannot be divisible by $q$. Assume by contradiction that it is the case and fix an element $\gamma_0\in G$ of order $q$. Then $\gamma_0$ is necessarily parabolic, so it has a unique fixed point $x\in \mathbf{P}^1(\Fbf_q)$. Let $\sigma \in \PGL_2(\Fbf_q)$ be such that 
$$\sigma \left( \begin{matrix} 1 & 1 \\ 0 & 1 
\end{matrix} \right)\sigma^{-1} = \gamma_0,$$
and define $\Gcal_1:= \Gcal\cdot \sigma = \sigma^*\Gcal\otimes \Lcal_{\omegabar(cX+d)}$ for $\sigma=\left(\begin{smallmatrix} a & b \\ c & d \end{smallmatrix}\right)$. Since geometrically we have $[+1]^*\Fcal\simeq \Fcal\cdot \left(\begin{smallmatrix} 1 & 1 \\ 0 & 1 \end{smallmatrix}\right)$ for any trace sheaf $\Fcal$, we see that we have a geometric isomorphism
$$[+1]^*\Gcal_1 \simeq \Gcal_1.$$
Suppose first that $\Gcal_1$ is ramified at some $y\in \mathbf{A}^1(\Fbfbar_q)$, then by the above, $\Gcal_1$ is also ramified at $y+1,...,y+p-1$ and we obtain by Lemma \ref{LemmeCond}
$$\cond(\Gcal)\geqslant \cond(\Gcal_1)-2\geqslant q-2+\rank(\Gcal_1)=q-2+\rank(\Gcal),$$
and in this case $K$ is $(q,N)$-good for trivial reasons.

\vspace{0.2cm}

Assume now that $\Gcal_1$ is lisse on $\mathbf{A}^1(\Fbfbar_q)$. Since $\Gcal$ is geometrically isotypic, the same is true for $\Gcal_1$ and the geometrically irreducible component $\Gcal_2$ of $\Gcal_1$ also satisfies $[+1]^*\Gcal_2\simeq \Gcal_2$. Using \cite[Lemma 5.4, (2)]{inverse} with $G=\Fbf_q$ and $P_h=0$, we have either
$$\cond(\Gcal_1)\geqslant \Swan_\infty(\Gcal_2)\geqslant q+\rank(\Gcal_2)$$
and in this case we are done as before, or $\Gcal_2$ is geometrically isomorphic to some Artin-Schreier sheaf $\Lcal_\psi$ for some additive character $\psi$. It follows that $\Gcal_1$ is geometrically isomorphic to a direct sum of copies of $\Lcal_\psi$ and thus, by definition of $\Gcal_1$, we have a geometric isomorphism
$$\Gcal \simeq \left(\bigoplus \Lcal_\psi \right)\cdot\sigma^{-1}=\bigoplus \Lcal_\psi \cdot \sigma^{-1},$$
which contradicts the fact that $\Fcal$ is not Fourier-exceptional.
\end{proof}

%%%%%%%%%%%%%%%%%%%%%%%%%%%%%%%%%%%%%%%%%%%%%%%%%%%%%%%%%%%%%%%%%%%%%%%%%%%%%%%%%%%%%%%%%%
%%%%%%%%%%%%%%%%%%%%%%%%%%%%%%%%%%%%%%%%%%%%%%%%%%%%%%%%%%%%%%%%%%%%%%%%%%%%%%%%%%%%%%%%%%
%%%%%%%%%%%%%%%%%%%%%%%%%%%%%%%%%%%%%%%%%%%%%%%%%%%%%%%%%%%%%%%%%%%%%%%%%%%%%%%%%%%%%%%%%%
%%%%%%%%%%%%%%%%%%%%%%%%%%%%%%%%%%%%%%%%%%%%%%%%%%%%%%%%%%%%%%%%%%%%%%%%%%%%%%%%%%%%%%%%%%
%%%%%%%%%%			BILINEAR FORMS INVOLVING TRACE FUNCTIONS					%%%%%%%%%%
%%%%%%%%%%%%%%%%%%%%%%%%%%%%%%%%%%%%%%%%%%%%%%%%%%%%%%%%%%%%%%%%%%%%%%%%%%%%%%%%%%%%%%%%%%
%%%%%%%%%%%%%%%%%%%%%%%%%%%%%%%%%%%%%%%%%%%%%%%%%%%%%%%%%%%%%%%%%%%%%%%%%%%%%%%%%%%%%%%%%%
%%%%%%%%%%%%%%%%%%%%%%%%%%%%%%%%%%%%%%%%%%%%%%%%%%%%%%%%%%%%%%%%%%%%%%%%%%%%%%%%%%%%%%%%%%
%%%%%%%%%%%%%%%%%%%%%%%%%%%%%%%%%%%%%%%%%%%%%%%%%%%%%%%%%%%%%%%%%%%%%%%%%%%%%%%%%%%%%%%%%%
\section{Bilinear forms involving trace functions} 
We begin with a classical result.
\begin{proposition}[Poly\'a-Vinogradov method]\label{PropositionPolya}Let $q$ be a prime number and $\Fcal$ be a Fourier trace sheaf on $\Abf^1_{\Fbf_q}$ with corresponding trace function $K$ modulo $q$. Let $f$ be a smooth and compactly supported function on $\Rbf$ and $N>0$ be a real number. Then for any $\varepsilon>0$, we have
$$\sum_{n\in\Zbf}K(n)f\left(\frac{n}{N}\right)\ll \min\left\{N,\frac{N}{q^{1/2}}\left(1+\frac{q^{1+\varepsilon}}{N}\right)\right\},$$
where the implied constant depends on $\varepsilon,f$ and the conductor of $\Fcal$.
\end{proposition}
\begin{proof}
The first bound is clear. For the second, an application of Poisson summation formula gives
\begin{equation}\label{Gives}
\sum_{n\in\Zbf}K(n)f\left(\frac{n}{N}\right)=\frac{N}{q^{1/2}}\sum_{n\in\Zbf}\widehat{K}(n)\widehat{f}\left(\frac{nN}{q}\right),
\end{equation}
where $\widehat{K}$ is the Fourier transform defined in \eqref{DefinitionFourier} and $\widehat{f}$ is the analytic Fourier transform
$$\widehat{f}(y)=\int_{\Rbf}f(x)e^{-2\pi ixy}dx.$$
Since $\widehat{f}$ has fast decay at infinity, the sum is essentially supported on $|n|\leqslant q^{1+\varepsilon}/N$ and thus, using the fact that the infinity norm of $\widehat{K}$ is bounded in terms of the conductor of $\Fcal$ (c.f. \eqref{BoundCondFourier}) and estimating trivially the right handside of \eqref{Gives} yields the second bound.
\end{proof}
A more elaborate treatment of the Poly\'a-Vinogradov method can be used to obtain bounds for bilinear sums \cite[Theorem 1.17]{prime}.
%%%%%%%%%%%%%%%%%%%%%%%%%%%%%%%%%%%%%%%%%%%%%%%%%%%%%%%%%%%%%%%%%%%%%%%%%%%%%%%%%%%%%%%%%%
%%%%%%%%%%								THEOREM FKM			   				 %%%%%%%%%%
%%%%%%%%%%%%%%%%%%%%%%%%%%%%%%%%%%%%%%%%%%%%%%%%%%%%%%%%%%%%%%%%%%%%%%%%%%%%%%%%%%%%%%%%%%
\begin{theorem}\label{TheoremPrime} Let $K$ be an isotypic trace function modulo $q$ associated to an isotypic $\ell$-adic sheaf $\Fcal$ such that $\Fcal$ does not contain a sheaf of the form $\Lcal_\omega\otimes\Lcal_\psi$ in his irreducible component. Let $M,N\geqslant 1$ be parameters and $(\alpha_m)_m, (\beta_n)_n$ two sequences of complex numbers supported on $[M/2,2M]$ and $[N/2,2N]$ respectively. 
\begin{enumerate}
\item[$(1)$] We have
$$\mathop{\sum\sum}_{\substack{n,m \\ (m,q)=1}}\alpha_m\beta_n K(mn) \ll ||\alpha||_2||\beta||_2(NM)^{1/2}\left(\frac{1}{q^{1/4}}+\frac{1}{M^{1/2}}+\frac{q^{1/4}\log^{1/2}q}{N^{1/2}}\right),$$
with
$$||\alpha||_2^2=\sum_{m}|\alpha_m|^2 \ , \ ||\beta||_2^2=\sum_n|\beta_n|^2.$$
\item[$(2)$] We have
$$\sum_{(m,q)=1}\alpha_m\sum_{n\leqslant N}K(mn) \ll \left(\sum_m|\alpha_m|\right)N\left(\frac{1}{q^{1/2}}+\frac{q^{1/2}\log q}{N}\right).$$
\end{enumerate}
In both estimates, the implicit constants depend only, and at most polynomially, on the conductor of $\Fcal.$
\end{theorem}
The above theorem beats the trivial bound and gives a power saving in the error term as long as $\max(N,M)\geqslant q^{1/2+\delta}$ and $\min(M,N)\geqslant q^\delta$ for some $\delta>0$. In the critical case where $N\sim M\sim q^{1/2}$, we have the powerful result of Kowalski, Michel and Sawin, which still saves a small power of $q$, but has been proved in the special case of classical Kloosterman sums \cite[Theorem 1.3]{sawin} and \cite[Theorem 5.1]{moments}. 
\begin{theorem}\label{TheoremSawin} Let $q$ be a prime number and $a$ an integer coprime with $q$. Let $M,N\geqslant 1$ be such that 
\begin{equation}\label{AssumptionTheoremBilinear}
1\leqslant M\leqslant N^2, \ \ N<q, \ \ MN<q^{3/2}.
\end{equation}
Let $(\alpha_m)_{m\leqslant M}$ be a sequence of complex numbers and $\mathcal{N}\subset [1,q-1]$ be an interval of length $N$. Then for any $\varepsilon>0$, we have
\begin{equation}\label{TypeI}
\sum_{n\in\mathcal{N}}\sum_{1\leqslant m\leqslant M}\alpha_m \Kl_k(anm;q) \ll q^\varepsilon ||\alpha||_1^{1/2}||\alpha||_2^{1/2}M^{1/4}N\left(\frac{M^2N^5}{q^{3}}\right)^{-1/12},
\end{equation}
with
$$||\alpha||_1= \sum_{1\leqslant m\leqslant M}|\alpha_m|$$
where the implied constant in \eqref{TypeI} only depends on $\varepsilon$ and $k$.
\end{theorem}

%%%%%%%%%%%%%%%%%%%%%%%%%%%%%%%%%%%%%%%%%%%%%%%%%%%%%%%%%%%%%%%%%%%%%%%%%%%%%%%%%%%%%%%%
%%%%%%%%%%%%%%%%%%%%%%%%%%%%%%%%%%%%%%%%%%%%%%%%%%%%%%%%%%%%%%%%%%%%%%%%%%%%%%%%%%%%%%%%
%%%%%%%%%%%%%%%%%%%%%%%%%%%%%%%%%%%%%%%%%%%%%%%%%%%%%%%%%%%%%%%%%%%%%%%%%%%%%%%%%%%%%%%%
%%%%%%%%%%%%%%%%%%%%%%%%%%%%%%%%%%%%%%%%%%%%%%%%%%%%%%%%%%%%%%%%%%%%%%%%%%%%%%%%%%%%%%%%
%%%%%%%%%%							PROOF OF THEOREM 2						%%%%%%%%%%
%%%%%%%%%%%%%%%%%%%%%%%%%%%%%%%%%%%%%%%%%%%%%%%%%%%%%%%%%%%%%%%%%%%%%%%%%%%%%%%%%%%%%%%%
%%%%%%%%%%%%%%%%%%%%%%%%%%%%%%%%%%%%%%%%%%%%%%%%%%%%%%%%%%%%%%%%%%%%%%%%%%%%%%%%%%%%%%%%
%%%%%%%%%%%%%%%%%%%%%%%%%%%%%%%%%%%%%%%%%%%%%%%%%%%%%%%%%%%%%%%%%%%%%%%%%%%%%%%%%%%%%%%%
\section{Proof of Theorem \ref{Theorem(q,M)-good}}\label{ProofTwist} This section is devoted to the proof of Theorem \ref{Theorem(q,M)-good}. For the cuspidal case, we will indicate the necessary changes in \cite[Sections 4,5,6]{twists} due to the level $q$ and the presence of a nebentypus. Finally, we will explain how to adapt \cite[Section 2]{prime} and put together with Section \ref{SectionCuspidalecase} to obtain the conclusion of Theorem \ref{Theorem(q,M)-good} in the Eisenstein case.

%%%%%%%%%%%%%%%%%%%%%%%%%%%%%%%%%%%%%%%%%%%%%%%%%%%%%%%%%%%%%%%%%%%%%%%%%%%%%%%%%%%%%%%%%%
%%%%%%%%%%%%%%%%%%%%%%%%%%%%%%%%%%%%%%%%%%%%%%%%%%%%%%%%%%%%%%%%%%%%%%%%%%%%%%%%%%%%%%%%%%
%%%%%%%%%%%%%%%%%%%%%%%%%%%%%%%%%%%%%%%%%%%%%%%%%%%%%%%%%%%%%%%%%%%%%%%%%%%%%%%%%%%%%%%%%%
%%%%%%%%%%							THE CUSPIDALE CASE								   %%%%%%%%%%
%%%%%%%%%%%%%%%%%%%%%%%%%%%%%%%%%%%%%%%%%%%%%%%%%%%%%%%%%%%%%%%%%%%%%%%%%%%%%%%%%%%%%%%%%%
%%%%%%%%%%%%%%%%%%%%%%%%%%%%%%%%%%%%%%%%%%%%%%%%%%%%%%%%%%%%%%%%%%%%%%%%%%%%%%%%%%%%%%%%%%
%%%%%%%%%%%%%%%%%%%%%%%%%%%%%%%%%%%%%%%%%%%%%%%%%%%%%%%%%%%%%%%%%%%%%%%%%%%%%%%%%%%%%%%%%%

\subsection{The cuspidal case}\label{SectionCuspidalecase}

%%%%%%%%%%%%%%%%%%%%%%%%%%%%%%%%%%%%%%%%%%%%%%%%%%%%%%%%%%%%%%%%%%%%%%%%%%%%%%%%%%%%%%%%%%
%%%%%%%%%%%%%%%%%%%%%%%%%%%%%%%%%%%%%%%%%%%%%%%%%%%%%%%%%%%%%%%%%%%%%%%%%%%%%%%%%%%%%%%%%%
%%%%%%%%%%						THE AMPLIFICATION METHOD						   %%%%%%%%%%
%%%%%%%%%%%%%%%%%%%%%%%%%%%%%%%%%%%%%%%%%%%%%%%%%%%%%%%%%%%%%%%%%%%%%%%%%%%%%%%%%%%%%%%%%%
%%%%%%%%%%%%%%%%%%%%%%%%%%%%%%%%%%%%%%%%%%%%%%%%%%%%%%%%%%%%%%%%%%%%%%%%%%%%%%%%%%%%%%%%%%

\subsubsection{The amplification method} Let $q>2$ be a prime number, $\omega$ a Dirichlet character of modulus $q$ and $\kappa\in\{0,1\}$ such that $\omega(-1)=(-1)^\kappa$. Let $f$ be a $L^2$-normalized primitive Hecke cusp form of weight $k_f\equiv \kappa \ (\modm 2)$ (resp. with Laplace eigenvalue $1/4+t_f^2$) if $f$ is holomorphic (resp. if $f$ is a Maass form) of level $q$ and nebentypus $\omega$. For some technical reasons, it is convenient to view $f$ as a modular form of level $2q$ (see the beginning of § \ref{SpecialCase}) under the isometric embedding (with respect to the Petersson inner product) 
$$f(z)\mapsto \frac{f(z)}{[\Gamma_0(q): \Gamma_0(2q)]^{1/2}}=\frac{f(z)}{\sqrt{3}},$$
which can be embedded in a suitable orthonormal basis of modular cusp forms of level $2q$, i.e. either $\Bcal_{k_f}(2q,\omega)$ or $\Bcal(2q,\omega)$. The strategy is therefore to estimate an amplified second moment of the sum $\Scal_V(g,K;q)$ where $g$ runs over a basis of $\Bcal_{k_f}(2q,\omega)$ and $\Bcal(2q,\omega)$.

\vspace{0.2cm}

To be precise, let $L\geqslant 1$ and $(b_\ell)$ a sequence of coefficients supported on $1\leqslant \ell\leqslant 2L$. For any modular form $g$, we let 
$$B(g):= \sum_{1\leqslant\ell\leqslant2L}b_\ell \lambda_g(\ell).$$
For an Eisenstein series $E_{\amf}(\cdot,1/2+it)$, we set
$$B(\amf,it):= \sum_{1\leqslant \ell\leqslant2L}b_\ell \lambda_{\amf}(\ell,it),$$
where for any singular cusp $\amf$, $\lambda_{\amf}(\ell,it)$ is given by \eqref{EigenvaluesEisensteinSeries}. Since the original form is of level $q$ and $L$ will be at the end a small power of $q$, we cannot choose the standard coefficients $b_\ell=\overline{\lambda_f(\ell)}$ for $\ell$ a prime $p\sim L$, but rather the less obvious amplifier found by Iwaniec in \cite{duke},
\begin{equation}\label{DefinitionAmplifier}
b_\ell = \left\{ \begin{array}{ccc} \lambda_f(p)\omegabar(p) & \ifm & \ell = p \sim L^{1/2} \ \mathrm{and} \ (p,2q)=1, \\
 & &  \\
 -\omegabar(p) & \ifm & \ell=p^2 \sim L \ \mathrm{and} \ (p,2q)=1, \\ 
 & & \\
0 & \mathrm{else}.
\end{array} \right.
\end{equation}
Since we will apply the trace formula, it is also better to consider the Fourier coefficients $\rho_g(n)$ instead of the Hecke eigenvalues $\lambda_g(n)$ in the definition of $\Scal_V(g,k;q)$. For this, we define 
$$\tilde{\Scal}_V(g,K;q)= \sum_n \rho_g(n)K(n)V\left(\frac{n}{q}\right)$$
and note that for $g$ primitive,  it is related to the original sum $\Scal_V(g,K;q)$ by the simple relation (c.f. \eqref{RelationHecke-Fourier2}),
\begin{equation}\label{SimpleRelation}
\tilde{\Scal}_V(g,K;q)=\rho_g(1)\Scal_V(g,K;q).
\end{equation}
We then let
\begin{equation}
\begin{split}
M(L) := & \ \sum_{\substack{k\equiv \kappa \ (\modm 2) \\ k>\kappa}}\dot{\phi}(k)(k-1)M(L;k) \\
+ & \ \sum_{g\in \Bcal(2q,\omega)}\tilde{\phi}(t_g)\frac{4\pi}{\cosh(\pi t_g)}|B(g)|^2 \left| \tilde{\Scal}_V(g,K;q)\right|^2 \\ 
+ & \ \sum_{\amf}\int_{-\infty}^\infty \tilde{\phi}(t)\frac{1}{\cosh(\pi t)}|B(\amf,it)|^2\left|\tilde{\Scal}_V(E_{\amf}(\cdot,1/2+it),K;q)\right|^2dt,
\end{split}
\end{equation}
where for any $k\equiv \kappa \ (\modm 2)$ with $k>\kappa,$
\begin{equation}
M(L;k):= \frac{(k-2)!}{\pi (4\pi)^{k-1}}\sum_{g\in\Bcal_k(2q,\omega)}|B(g)|^2\left|\tilde{\Scal}_V(g,K;q)\right|^2,
\end{equation}
and we refer to \cite[Section 3.2]{twists} or \cite[(2.9)]{blomer} for the choice and properties of the test function $\phi=\phi_{a,b}$. The key Proposition is the following \cite[Proposition 4.1]{twists}.
%%%%%%%%%%%%%%%%%%%%%%%%%%%%%%%%%%%%%%%%%%%%%%%%%%%%%%%%%%%%%%%%%%%%%%%%%%%%%%%%%%%%%%%%%%
%%%%%%%%%%								THE KEY PROPOSITION						%%%%%%%%%%
%%%%%%%%%%%%%%%%%%%%%%%%%%%%%%%%%%%%%%%%%%%%%%%%%%%%%%%%%%%%%%%%%%%%%%%%%%%%%%%%%%%%%%%%%%	
\begin{proposition}\label{Proposition4.1} Assume that $M\geqslant 1$ is such that $K$ is $(q,M)$-good. Let $V$ be a smooth compactly supported function satisfying Condition $V(C,P,Q)$. Let $(b_\ell)$ be the sequence of complex numbers defined by \eqref{DefinitionAmplifier}.

For any $\varepsilon >0$, there exists $k(\varepsilon)>\kappa$ such that for any $k\geqslant k(\varepsilon)$ and any integers $a>b>2$ satisfying 
$$a-b\geqslant k(\varepsilon), \ \  a-b\equiv \kappa \ (\modm 2),$$
we have the bound
\begin{equation}\label{BoundProposition4.1}
M(L), \ M(L;k) \ll \left\{ q^{1+\varepsilon}L^{1/2}P(P+Q)+q^{1/2}L^{2}PQ^2(P+Q)\right\}M^3,
\end{equation}
provided
\begin{equation}\label{ConditionL}
q^\varepsilon LQ<q^{1/4}
\end{equation}
and where the implied constant depends on $C,\varepsilon,a,b,k$ and polynomially on the archimedean parameter of $f$.
\end{proposition}
Theorem \ref{Theorem(q,M)-good} can be deduced from Proposition \ref{Proposition4.1} exactly in the same way as in \cite[Section 4.2]{twists}. The only changes is to use \eqref{SimpleRelation} to pass from $|B(f)|^2|\Scal_V(f,K;q)|^2$ to $|B(f)|^2|\tilde{\Scal}_V(f,K;q)|^2$ and then \eqref{LowerBound} for the upper bound on $|\rho_f(1)|^{-2}$. Finally, since for any prime $p$ different from $q$ we have the elementary relation
$$\lambda_f(p)^2-\lambda_f(p^2)=\omega(p),$$ 
we obtain the lower bound
$$B(f) \gg \frac{L^{1/2}}{\log L},$$
simply using the prime number Theorem. Hence it remains to prove Proposition \ref{Proposition4.1}. 

\vspace{0.2cm}

Expanding the square in $|B(g)|^2$ and $|\tilde{\Scal}_V(g,K;q)|$ (choosing the variables $\ell_1,\ell_2$ for those comming from the amplifier), we get a first decomposition of $M(L)$ and $M(L;k)$
\begin{equation}\label{FirstDecomposition}
M(L)=M_d(L)+M_{nd}(L) \ \ \mathrm{and} \ \ M(L;k)= M_{d}(L;k)+M_{nd}(L;k),
\end{equation}
depending on weither $(\ell_1,\ell_2)>1$ (the diagonal term) or not. For the diagonal term, we have the following lemma which is the analogous of \cite[Lemma 5.1]{twists}.
%%%%%%%%%%%%%%%%%%%%%%%%%%%%%%%%%%%%%%%%%%%%%%%%%%%%%%%%%%%%%%%%%%%%%%%%%%%%%%%%%%%%%%%%%%
%%%%%%%%%%						LEMMA DIAGONAL TERM							      %%%%%%%%%%
%%%%%%%%%%%%%%%%%%%%%%%%%%%%%%%%%%%%%%%%%%%%%%%%%%%%%%%%%%%%%%%%%%%%%%%%%%%%%%%%%%%%%%%%%%
\begin{lemme} Assume $|K|\leqslant M$. For any $\varepsilon >0$, we have
$$M_d(L), M_d(L;k) \ll M^2q^{1+\varepsilon}L^{1/2}P(P+1),$$
where the implied constant depends only on $\varepsilon$.
\end{lemme}
\begin{proof}
We consider $M_d(L)$ which decomposes as a sum of the holomorphic, Maass and Eisenstein contributions
$$M_d(L)=M_{d,\mathrm{hol}}(L)+M_{d,\mathrm{
Maass}}(L)+M_{d,\mathrm{Eis}}(L).$$
We treat only $M_{d,\mathrm{Maass}}(L)$ since the others contributions are the same and even simpler. For instance, we have
$$M_{d,\mathrm{Maass}}(L)=\sum_{g\in \Bcal(2q,\omega)}\tilde{\phi}(t_g)\frac{4\pi}{\cosh(\pi t_g)}\mathscr{C}(g,L)\left|\sum_n K(n)\rho_g(n)V\left(\frac{n}{q}\right)\right|^2,$$
with
$$\mathscr{C}(g,L):= \sum_{(\ell_1,\ell_2)>1}b_{\ell_1}\overline{b_{\ell_2}}\lambda_g(\ell_1)\overline{\lambda_g(\ell_2)}.$$
By definition of the coefficients $b_\ell$ (c.f. \eqref{DefinitionAmplifier}), the case $(\ell_1,\ell_2)>1$ appears when $\ell_1=\ell_2=p\sim L^{1/2}$, $\ell_1=p^2=\ell_2^2\sim L$ (or the inverse) and $\ell_1=\ell_2=p^2\sim L$. We write $\mathscr{C}(g,L)=\mathscr{C}_1(g,L)+\mathscr{C}_2(g,L)+\mathscr{C}_3(g,L)$ according to the different possibilities and we estimate the three quantities individually. We first have by Cauchy-Schwarz inequality and \eqref{FourthPower},
\begin{alignat*}{1}
\mathscr{C}_1(g,L)= & \ \sum_{\substack{p\sim L^{1/2} \\ p \ \mathrm{prime}}} |\lambda_f(p)|^2|\lambda_g(p)|^2 \\
\leqslant & \ \left(\sum_{\substack{p\sim L^{1/2}}}|\lambda_f(p)|^4\right)^{1/2} \left(\sum_{\substack{p\sim L^{1/2}}}|\lambda_g(p)|^4\right)^{1/2} \\ \ll & \ (qL(1+|t_f|)(1+t_g))^\varepsilon L^{1/2},
\end{alignat*}
where the implied constant only depends on $\varepsilon$. For the second case, we have using $|\lambda(p^2)|\leqslant 1+|\lambda(p)|^2$ (c.f. \eqref{HeckeMult1}), Hölder and again \eqref{FourthPower},
\begin{alignat*}{1}
|\mathscr{C}_2(g,L)|\leqslant & \  \sum_{\substack{p\sim L^{1/2} \\ p \ \mathrm{prime}}}|\lambda_f(p)||\lambda_g(p)||\lambda_g(p^2)|\leqslant \sum_{\substack{p\sim L^{1/2} \\ p \ \mathrm{prime}}}|\lambda_f(p)||\lambda_g(p)|(1+|\lambda_g(p)|^2) \\ 
\leqslant & \  \left(\sum_{\substack{p\sim L^{1/2}}}|\lambda_f(p)|^4\right)^{1/4} \left(\sum_{\substack{p\sim L^{1/2}}}|\lambda_g(p)|^4\right)^{1/4}\left(\sum_{p\sim L^{1/2}}(1+|\lambda_g(p)|^2)^2\right)^{1/2} \\ \ll & \  (qL(1+|t_f|)(1+t_g))^\varepsilon L^{1/2}.
\end{alignat*}
Using the inequality $|\lambda_g(p^2)|^2 \leqslant 2(1+|\lambda_g(p)|^4)$, we treat in the same way $\mathscr{C}_3(g,L)$. The rest of the proof is exaclty the same as \cite[Lemma 5.1]{twists}, except that we must use Proposition \ref{PropositionSieve} for the spectral large sieve (possible since the conductor of $\omega$ is either $1$ or a prime $q$) instead of the original version of Deshouillers-Iwaniec \cite[Theorem 2, (1.29)]{82}.
\end{proof}
Now comes the contribution of the $\ell_1,\ell_2$ such that $(\ell_1,\ell_2)=1$. We first change the complex conjugate $\overline{\lambda_g(\ell_2)}=\omegabar(\ell_2)\lambda_g(\ell_2)$ in $M_{nd}(L)$ and $M_{nd}(L;k)$ appearing in the decomposition \eqref{FirstDecomposition} (c.f. \eqref{HeckeMult2}). By the primiality condition, we use the multiplicativity of the Hecke eigenvalues \eqref{HeckeMult1} followed by the relation \eqref{RelationHecke-Fourier} to obtain
$$\lambda_g(\ell_1\ell_2)\rho_g(n_1)=\sum_{d|(\ell_1\ell_2,n_1)}\omega(d)\rho_g\left(\frac{\ell_1\ell_2n_1}{d^2}\right).$$
Once we have done this, we apply the Petersson trace formula \eqref{Petersson} to $M_{nd}(L;k)$ in \eqref{FirstDecomposition}, obtaining $$\pi M_{nd}(L;k)=M_1(L;k)+M_2(L;k),$$ where $M_1(L;k)$ corresponds to the diagonal term $\delta(\ell_1\ell_2n_1d^{-2},n_2)$. Similarily, we apply Kuznetsov formula \eqref{Kuznetsov} to $M_{nd}(L)$ and since there is no diagonal term, we write $M_{nd}(L)=M_2(L).$ 

The treatment of the diagonal term $M_1(L;k)$ is contained in \cite[Lemma 5.3]{twists}, with the appropriate changes using \eqref{Ramanujan-Petersson-Average} for the coefficients of the amplifier,
\begin{equation}\label{BoundDiagonalTerm}
M_1(L;k) \ll (q(1+|t_f|))^\varepsilon qL^{1/2}PM^2,
\end{equation}
with an implied constant depending only on $\varepsilon$.

%%%%%%%%%%%%%%%%%%%%%%%%%%%%%%%%%%%%%%%%%%%%%%%%%%%%%%%%%%%%%%%%%%%%%%%%%%%%%%%%%%%%%%%%%%
%%%%%%%%%%%%%%%%%%%%%%%%%%%%%%%%%%%%%%%%%%%%%%%%%%%%%%%%%%%%%%%%%%%%%%%%%%%%%%%%%%%%%%%%%%
%%%%%%%%%%							THE OFF-DIAGONAL TERM						%%%%%%%%
%%%%%%%%%%%%%%%%%%%%%%%%%%%%%%%%%%%%%%%%%%%%%%%%%%%%%%%%%%%%%%%%%%%%%%%%%%%%%%%%%%%%%%%%%%
%%%%%%%%%%%%%%%%%%%%%%%%%%%%%%%%%%%%%%%%%%%%%%%%%%%%%%%%%%%%%%%%%%%%%%%%%%%%%%%%%%%%%%%%%%

\subsubsection{The off-diagonal terms} This is the most important case of $M_2(L)$ and $M_2(L;k)$ and thus we write explicitly the quantities to study. For $\phi$ an arbitrary function, we write
\begin{equation}
\begin{split}
M_2[\phi]= & \ \frac{1}{2q}\sum_{(\ell_1,\ell_2)=1}b_{\ell_1}\overline{b_{\ell_2}}\overline{\omega}(\ell_2)\sum_{d|\ell_1\ell_2}\omega(d)\sum_{\substack{n_1,n_2 \\ d|n_1}}K(n_1)\overline{K(n_2)}V\left(\frac{n_1}{q}\right)V\left(\frac{n_2}{q}\right) \\ 
\times & \ \sum_{c\geqslant 1}c^{-1}S_{\omega}(\ell_1\ell_2n_1d^{-2},n_2;2cq)\phi\left(\frac{4\pi}{2cq}\sqrt{\frac{\ell_1\ell_2n_1n_2}{d^2}}\right),
\end{split}
\end{equation}
in order to have 
$$M_2(L)=M_2[\phi_{a,b}] \ \ \mathrm{and} \ \ M_2(L;k)=M_2[\phi_k]$$
where $\phi_k=2\pi i^{-k}J_{k-1}$ is the Bessel function. We transform the sum as
\begin{equation}\label{M2}
M_2[\phi]=\sum_{(\ell_1,\ell_2)=1}b_{\ell_1}\overline{b_{\ell_2}}\omegabar(\ell_2)\sum_{de=\ell_1\ell_2}\omega(d)M_2[\phi;d,e],
\end{equation}
where 
$$M_2[\phi;d,e]=\frac{1}{2q}\sum_{c\geqslant 1}c^{-1}\Ecal_\phi(c,d,e)$$
and
\begin{alignat*}{1}
\Ecal_\phi(c,d,e)= & \ \sum_{n_1,n_2}S_\omega(en_1,n_2;2cq)K(dn_1)\overline{K(n_2)}\phi\left(\frac{4\pi \sqrt{en_1n_2}}{2cq}\right)V\left(\frac{dn_1}{q}\right)V\left(\frac{n_2}{q}\right) \\ 
= & \ \sum_{n_1,n_2}S_\omega(en_1,n_2;2cq)K(dn_1)\overline{K(n_2)}H_\phi(n_1,n_2),
\end{alignat*}
with
\begin{equation}\label{DefinitionH}
H_\phi(x,y)=\phi\left(\frac{4\pi\sqrt{exy}}{2cq}\right)V\left(\frac{dx}{q}\right)V\left(\frac{y}{q}\right).
\end{equation}
As in \cite[Section 5.4]{twists}, we truncate the parameter $c$ in $M_2[\phi;d,e]$ by writing $M_2[\phi;d,e]=M_{2,C}[\phi;d,e]+M_3[\phi;d,e]$ where $M_{2,C}[\phi;d,e]$ denotes the contribution of ther terms with $c>C$ for some $C=C(d,e)\geqslant 1/2$ and correspondingly
\begin{equation}\label{tail}
M_2[\phi]=M_{2,\mathrm{tail}}[\phi]+M_{3}[\phi].
\end{equation}
It turns out that with the choice
\begin{equation}\label{ChoiceC}
C=\max\left(\frac{1}{2},q^{\delta}P\sqrt{\frac{e}{d}}\right)\ll q^\delta LP,
\end{equation}
the contribution of $c>C$ is negligible (see \cite[(5.9)]{twists}), so we focus on the complementary sum, which is given by
\begin{equation}\label{ComplementarySum}
M_3[\phi;d,e]=\frac{1}{2q}\sum_{1\leqslant c\leqslant C}c^{-1}\Ecal_\phi(c,d,e).
\end{equation}
In particular, the above expression is zero if $C<1$. 

Recall that we factored the product $\ell_1\ell_2$ as $de=\ell_1\ell_2$. Since we allow $\ell_1$ and $\ell_2$ to be square of primes, there are more different type of factorization to consider. We distinguish three types.
\begin{enumerate}
\item [$\bullet$] Type I (balanced case) : this is when both $d$ and $e$ are $\neq 1$ and $d/e\sim 1$, so $d$ and $e$ are either primes $\sim L^{1/2}$ (type ($L^{1/2},L^{1/2}$)) or square of primes $\sim L$ (type $(L,L)$) with $(d,e)=1$ in each case.
\item [$\bullet$] Type II (unbalanced case) : this is when $(d,e)$ satisfies $e/d\gg L^{1/2}$, i.e. is of type $(1,L),(1,L^{3/2}),(1,L^2),$ $(L^{1/2},L)$ and $(L^{1/2},L^{3/2})$.
\item [$\bullet$] Type III (unbalanced case) : this is when $(d,e)$ satisfies $d/e\gg L^{1/2}$, so is of type $(L,1),(L^{3/2},1),(L^2,1),(L,L^{1/2})$ and $(L^{3/2},L^{1/2}).$
\end{enumerate}
Assuming the harmless condition
\begin{equation}\label{HarmlessCondition}
q^\delta P\ll L^{1/2},
\end{equation}
we obtain by \eqref{ChoiceC} :
\begin{lemme} Suppose that $(d,e)$ is of Type III and that \eqref{HarmlessCondition} is satisfied. Then we have the equality
$$M_{3}[\phi;d,e]=0.$$
\end{lemme}
It remains to deal with the types I and II. The goal now is to transform the sums $\Ecal_\phi(c,d,e)$ to connect them with the correlation sums $\Ccal(K,\omega;\gamma)$ defined in \eqref{DefinitionCorrelation} for suitable matrices $\gamma$. This is the content of \cite[Section 5.5]{twists} and it is achieved using principally twisted multiplicativity of the Kloosterman sums and Poisson summation formula. The only difference here is the appearance of the nebentypus $\omega$ when we open the Kloosterman sum. We also mention that it is in this treatment that we use the fact that the level is $2q$ and not $q$. The result is that for any $c\geqslant 1$, we have the identity
\begin{equation}\label{ConnectionCorrelation}
\Ecal_\phi(c,d,e)=\frac{\omega(-d)}{q}\mathop{\sum\sum}_{\substack{n_1n_2\neq 0, \ (n_2,2c)=1 \\ n_1n_2\equiv e (\modm 2c)}}\widehat{H}_\phi\left(\frac{n_1}{2cq},\frac{n_2}{2cq}\right)\Ccal(K,\omega;\gamma(c,d,e,n_1,n_2)),
\end{equation}
where $\widehat{H}_\phi$ is the Fourier transform of $H_\phi$ and
\begin{equation}\label{DefinitionCorrelationMatrix}
\gamma(c,d,e,n_1,n_2):= \left( \begin{matrix}
n_1 & \frac{n_1n_2-e}{2c} \\ 2cd & dn_2
\end{matrix} \right) \in \mathrm{M}_2(\Zbf)\cap \GL_2(\Qbf).
\end{equation}
\begin{remq} Observe that $\det(\gamma(c,d,e,n_1,n_2))=de$ which is coprime with $q$. Hence the reduction of $\gamma(c,d,e,n_1,n_2)$ modulo $q$ provides a well defined element of $\PGL_2(\Fbf_q)$.
\end{remq}

\subsubsection{Analysis of $\Ecal_\phi(c,d,e)$} The first step in the analysis of \eqref{ConnectionCorrelation} passes by the study of the Fourier transform of $H_\phi(x,y)$. This is the content of \cite[Sections 5.6-5.7]{twists} and it is contained in Lemmas 5.7,5.9. One of the consequences is that it allows to truncate the $n_1,n_2$-sum in $\Ecal_\phi(c,d,e)$ to 
\begin{equation}\label{Restriction}
0\neq |n_1|\leqslant \Ncal_1:= q^\varepsilon cd \frac{\left(Q+\frac{P}{2c}\sqrt{\frac{e}{d}}\right)}{P}, \ \ 0\neq|n_2|\leqslant \Ncal_2:= \frac{\Ncal_1}{d},
\end{equation}
(see \cite[(5.21)]{twists}). The final strategy is to separate the terms in \eqref{ConnectionCorrelation} (with the restriction \eqref{Restriction} on $n_1,n_2$) according to whether
$$|\Ccal(K,\omega;\gamma(c,d,e,n_1,n_2))|\leqslant Mq^{1/2}$$
or not, i.e., as to whether the reduction modulo $q$ of the matrix $\gamma(c,d,e,n_1,n_2)$ is in the set $\Gbf_{K,\omega,M}$ of $M$-correlation matrix or not. We thus write
$$\Ecal_\phi(c,d,e)=\Ecal_\phi^{c}(c,d,e)+\Ecal_\phi^n(c,d,e),$$
where $\Ecal_\phi^c(c,d,e)$ is the subsum of \eqref{ConnectionCorrelation} where we restrict to the variables $n_1,n_2$ such that the reduction modulo $q$ of $\gamma(c,d,e,n_1,n_2)$ belongs to $\Gbf_{K,\omega,M}$ and $\Ecal_\phi^n(c,d,e)$ is the contribution of the remaining terms. According to \eqref{ComplementarySum}, \eqref{tail} and \eqref{M2}, we also write
\begin{alignat*}{1}
M_3[\phi;d,e]= & \ \frac{1}{2q}\sum_{c\leqslant C}c^{-1}\left(\Ecal_\phi^c(c,d,e)+\Ecal_\phi^n(c,d,e)\right) \\ 
= & \ M_3^c[\phi;d,e]+M_3^n[\phi;d,e],
\end{alignat*}
and
\begin{alignat*}{1}
M_3[\phi]= & \ \sum_{(\ell_1,\ell_2)=1}b_{\ell_1}\overline{b_{\ell_2}}\omegabar(\ell_2)\sum_{de=\ell_1\ell_2}\omega(d)\left(M_3^c[\phi;d,e]+M_3^n[\phi;d,e]\right)  \\ 
= & \ M_3^c[\phi]+M_3^n[\phi].
\end{alignat*}

\begin{lemme} With the above notations, we have
$$M_3^n[\phi_{a,b}]\ll Mq^{1/2+\varepsilon}L^2PQ^2(P+Q), \ \ M_3^n[\phi_k]\ll Mk^3q^{1/2+\varepsilon}L^2PQ^2(P+Q),$$
for any $\varepsilon>0$ where the implied constant depends on $\varepsilon,a,b$ for $\phi=\phi_{a,b}$ and on $\varepsilon$ for $\phi=\phi_k$.
\end{lemme}
\begin{proof}
This is the content of \cite[pp. 625-626]{twists}, with minimal changes due to the different nature of pairs $(d,e)$ of type I and II.
\end{proof}

%%%%%%%%%%%%%%%%%%%%%%%%%%%%%%%%%%%%%%%%%%%%%%%%%%%%%%%%%%%%%%%%%%%%%%%%%%%%%%%%%%%%%%%%%%%%%%%%%%%%%%%%%%%%%%%%%%%%%%%%%%%%%%%%%%%%%%%%%%%%%%%%%%%%%%%%%%%%%%%%%%%%%%%%%%%%%%%%%%%%
%%%%%%%%%%				CONTRIBUTION OF THE CORRELATING MATRICES				%%%%%%%%%%
%%%%%%%%%%%%%%%%%%%%%%%%%%%%%%%%%%%%%%%%%%%%%%%%%%%%%%%%%%%%%%%%%%%%%%%%%%%%%%%%%%%%%%%%%%
%%%%%%%%%%%%%%%%%%%%%%%%%%%%%%%%%%%%%%%%%%%%%%%%%%%%%%%%%%%%%%%%%%%%%%%%%%%%%%%%%%%%%%%%%To 
To conclude the proof of Proposition \ref{Proposition4.1}, it remains to evaluate the contribution $M_3^c[\phi;d,e]$ corresponding to the matrices whose reduction modulo $q$ is a correlating matrix, i.e.
such that $\gamma(c,d,e,n_1,n_2) \modm q \in \Gbf_{K,\omega,M}$. The final lemma is the following :
\begin{lemme} Under the assumption
\begin{equation}\label{assumption}
q^{3\varepsilon}LQ<q^{1/4},
\end{equation}
we have
$$M_3^c[\phi_k]\ll M^3k^3q^{1+\varepsilon}L^{1/2}PQ, \ \ M_3^c[\phi_{a,b}]\ll M^3q^{1+\varepsilon}L^{1/2}PQ,$$
where the implied constant depends on $\varepsilon,a,b$.
\end{lemme}
\begin{proof}
The proof is \cite[Sections 6.1,6.3,6.5]{twists} (recall that here there are no parabolic elements). Various arguments use the fact that the discriminant of certain binary quadratic form is not zero. For example, if $\gamma=\gamma(c,d,e,n_1,n_2)$ is a toric matrix, then we need to have $(n_1+dn_2)^2-4de\neq 0$ and we cannot say that $de=\ell_1\ell_2$ is squarefree since we allow square of primes in the amplifier. This is not a problem here because if $(n_1+dn_2)^2=4de$, then we would get (see \eqref{DefinitionCorrelationMatrix})
$$\Tr(\gamma)^2-4\det(\gamma)=0 \ \mathrm{in} \ \Fbf_q.$$
This means that $\gamma$ has only one distinct eigenvalue, so its is necessarily scalar since not parabolic by assumption. But for $\gamma$ scalar, we have $cd\equiv 0 \ (\modm q)$, wich is not possible by \eqref{assumption} and \eqref{ChoiceC}.

\vspace{0.2cm}

Another argument uses the fact that $dn_2^2-e\neq 0$ in a situation where $n_1-dn_2=0$ (c.f. \cite[Section 6.3, p.p 629]{twists}). Again, $d$ and $e$ are not necessarily coprime here so we cannot argue in the same way. However, if $dn_2^2=e$, then since $n_1=dn_2$, we obtain $n_1n_2-e=dn_2^2-e=0$ and the matrix $\gamma(c,d,e,n_1,n_2)$ takes the form
$$\gamma(c,d,e,n_1,n_2)=\left( \begin{matrix} n_1 & 0 \\ 2cd & dn_2 \end{matrix} \right).$$
Since $dn_2=n_1$, this matrix is parabolic with single fixed point $z=0$, which contradicts the fact that $\Gbf_{k,\omega,M}$ does not contain parabolic elements.
\end{proof}

%%%%%%%%%%%%%%%%%%%%%%%%%%%%%%%%%%%%%%%%%%%%%%%%%%%%%%%%%%%%%%%%%%%%%%%%%%%%%%%%%%%%%%%%%%%%%%%%%%%%%%%%%%%%%%%%%%%%%%%%%%%%%%%%%%%%%%%%%%%%%%%%%%%%%%%%%%%%%%%%%%%%%%%%%%%%%%%%%%%%%%%%%%%%%%%%%%%%%%%%%%%%%%%%%%%%%%%%%%%%%%%%%%%%%%%%%%%%%%%%%%%%%%%%%%%%%%%%%%%%%%%%%%%%%%
%%%%%%%%%%						THE EISENSTEIN CASE							%%%%%%%%%%
%%%%%%%%%%%%%%%%%%%%%%%%%%%%%%%%%%%%%%%%%%%%%%%%%%%%%%%%%%%%%%%%%%%%%%%%%%%%%%%%%%%%%%%%%%%%%%%%%%%%%%%%%%%%%%%%%%%%%%%%%%%%%%%%%%%%%%%%%%%%%%%%%%%%%%%%%%%%%%%%%%%%%%%%%%%%%%%%%%%%%%%%%%%%%%%%%%%%%%%%%%%%%%%%%%%%%%%%%%%%%%%%%%%%%%%%%%%%%%%%%%%%%%%%%%%%%%%%%%%%%%%%%%%%%%

\subsection{The Eisenstein case}  We recall the notations from Section \ref{Section1}. For $q>2$ prime, $\omega$ a Dirichlet character modulo $q$ and $t\in\Rbf$, we set
$$\lambda_\omega(n,it)= \sum_{ab=n}\omega(a)\left(\frac{a}{b}\right)^{it},$$
and
$$\Scal_V(\omega,it,K;q)=\sum_{n\geqslant 1}\lambda_{\omega}(n,it)K(n)V\left(\frac{n}{q}\right)$$
for $K$ a non Fourier-exceptional isotypic trace function and $V$ satisfying condition $V(C,P,Q)$. Since $\lambda_{\omega}(n,it)$ appears as Hecke eigenvalues (for $(n,q)=1$) of the Eisenstein series $E_1(\cdot,1/2+it)$ (the cusp $\amf=1$) lying  in the continuous spectrum of the Laplacian on the space of modular forms of level $q$ (and thus also of level $2q$ after a normalization) and nebentypus $\omega$ (see § \ref{SpecialCase} and Remark \ref{RemarkEisenstein}), we may estimate an amplified second moment of $\Scal_V(\omega,it,K;q)$ by embedding in the Eisentsein spectrum and using Kuznetsov trace formula as in the cuspidal case.

\vspace{0.1cm}

For $\tau\in\Rbf$, we define as in \eqref{DefinitionAmplifier}
\begin{equation}\label{DefinitionAmplifierEisenstein}
b_{\ell}(\tau) := \left\{ \begin{array}{ccc} \lambda_{\omega}(\ell,i\tau)\omegabar(\ell) & \ifm & \ell=p\sim L^{1/2} \ \mathrm{and} \ (p,2q)=1 \\ 
 & & \\
-\omegabar(\ell) & \ifm & \ell=p^2\sim L \ \mathrm{and} \ (p,2q)=1 \\ 
 & & \\
0 & \mathrm{else},    &
\end{array} \right.
\end{equation}
and for $g$ a cuspidal form, we set
$$B_{\tau}(g) = \sum_{1\leqslant \ell\leqslant 2L}b_\ell(\tau)\lambda_g(\ell).$$
For an Eisenstein series $E_{\amf}(\cdot, 1/2+it)$, we let
$$B_\tau(\amf,it)= \sum_{1\leqslant\ell\leqslant 2L}b_\ell(\tau)\lambda_{\amf}(\ell,it)$$
and we also write $B_\tau(\omega,it)$ so that it corresponds to the Eisenstein series $E_1(\cdot,1/2+it)$ having $\lambda_\omega(n,it)$ as Hecke eigenvalues. Since the trace formula involves Fourier coefficients instead of Hecke eigenvalues, we define as in Section \ref{SectionCuspidalecase}
$$\tilde{\Scal}_V(\omega,it,K;q)=\sum_{n\geqslant 1}\rho_\omega (n,it)K(n)V\left(\frac{n}{q}\right),$$
with the relation
\begin{equation}\label{RelationEisenstein}
\Scal_V(\omega,it,K;q)=\rho_\omega(1,it)^{-1}\tilde{\Scal}_V(\omega,it,K;q).
\end{equation}
\begin{remq} Actually, the relation \eqref{RelationEisenstein} is true if we restrict the $n$-summation in $\Scal_V(\omega,it,K;q)$ to $(n,q)=1$. However, we could consider directly this restriction at the beginning since the error to pass from one to the other is given by
$$\Scal_V(\omega,it,K;q)=\sum_{(n,q)=1}\lambda_\omega(n,it)K(n)V\left(\frac{n}{q}\right)+O\left(q^\varepsilon M(P+1)\right).$$
\end{remq}
Using the lower bound for $\rho_\omega(1,it)$ given by \eqref{LowerBoundEisenstein} and $\tilde{\phi}_{a,b}(t)\asymp (1+|t|)^{\kappa-2b-2}$ (c.f. \cite[(2.21)]{blomer}), we obtain exactly as in Section \ref{SectionCuspidalecase} (see Proposition \ref{Proposition4.1}),
\begin{equation}\label{FullSpectrumEisenstein}
\begin{split}
 \int_{\Rbf}\frac{\left|\Scal_V(\omega,it,K;q)\right|^2}{(1+|t|)^{2b+2-\varepsilon}}\left|B_\tau(\omega,it)\right|^2dt  & \ \ll q^{1+\varepsilon}\int_{\Rbf}\frac{(1+|t|)^{\kappa-2b-2}}{\cosh(\pi t)}\left|\tilde{\Scal}_V(\omega,it,K;q)\right|^2\left|B_\tau(\omega,it)\right|^2dt \\ & \ \ll q^{1+\varepsilon} \int_{\Rbf}\frac{\tilde{\phi}_{a,b}(t)}{\cosh(\pi t)}\left|\tilde{\Scal}_V(\omega,it,K;q)\right|^2\left|B_\tau(\omega,it)\right|^2dt \\ & \ \ll M^3\left\{ q^{2+\varepsilon}L^{1/2}P(P+Q)+q^{3/2}L^2PQ^2(P+Q)\right\}.
\end{split}
\end{equation}
In order to apply \eqref{FullSpectrumEisenstein}, the following Lemma gets a suitable lower bound for the amplifier $B_\tau(\omega,it)$ when $\tau$ is close enough to $t$ (see \cite[Lemma 2.4]{prime}).
%%%%%%%%%%%%%%%%%%%%%%%%%%%%%%%%%%%%%%%%%%%%%%%%%%%%%%%%%%%%%%%%%%%%%%%%%%%%%%%%%%%%%%%%%%
%%%%%%%%%						LEMMA LOWER BOUND							  %%%%%%%%%%
%%%%%%%%%%%%%%%%%%%%%%%%%%%%%%%%%%%%%%%%%%%%%%%%%%%%%%%%%%%%%%%%%%%%%%%%%%%%%%%%%%%%%%%%%%
\begin{lemme}\label{LemmaLowerBound} For $L$ large enough, we have
$$B_\tau(\omega,it)\gg \frac{L^{1/2}}{\log L},$$
uniformly in $t,\tau \in \Rbf$ satisfying
$$|t-\tau|\leqslant \frac{1}{\log ^2 L}.$$
\end{lemme}
\begin{proof}
Observe that since $\ell$ has at most three divisors, we have $|b_\ell(\tau)|\leqslant 3$ and thus
\begin{alignat*}{1}
|B_\tau(\omega,it)-B_\tau(\omega,i\tau)|\leqslant & \ \sum_{\ell}|b_\ell(\tau)||\lambda_\omega(\ell,it)-\lambda_\omega(\ell,i\tau)| \\ 
\leqslant & \ 3\sum_{\substack{p\sim L^{1/2} \\ p \ \mathrm{prime}}}\left\{ |\lambda_\omega(p,it)-\lambda_\omega(p,i\tau)|+|\lambda_\omega(p^2,it)-\lambda_\omega(p^2,i\tau)| \right\} \\ \leqslant & \ 6\sum_{\substack{p\sim L^{1/2} \\ p \ \mathrm{prime}}} \left\{ |p^{it}-p^{i\tau}|+|p^{2it}-p^{2i\tau}|\right\} \\ 
\leqslant & \ 36 |t-\tau|\sum_{\substack{p\sim L^{1/2} \\ p \ \mathrm{prime}}}\log(p) \ll \frac{L^{1/2}}{\log^2 L}.
\end{alignat*}
It is therefore enough to prove the result for $t=\tau$. But this is a consequence of the elementary relation
$$\lambda_\omega (p,it)^2\omegabar(p)-\omegabar(p)\lambda_\omega(p^2,it)=1,$$
valid for $(p,q)=1$,
and the prime number Theorem.
\end{proof}
The above Lemma combining with the average bound \eqref{FullSpectrumEisenstein} allows us to deduce a first upper-bound for short averages of twists of Eisenstein series. For this, we introduce the notations
$$I(\tau,q) :=\left\{ t\in\Rbf \ | \ |t-\tau|\leqslant \frac{1}{\log^2 q}\right\}$$
and
$$M(P,Q;q):= M^{3/2} q^{1-\frac{1}{16}}(PQ)^{1/2}(P+Q)^{1/2},$$
so that Theorem \ref{Theorem(q,M)-good} claims that
$$\Scal_V(\omega,it,K;q)\ll_\varepsilon q^\varepsilon (1+|t|)^A M(P,Q;q)$$
for any $\varepsilon>0$ and some $A\geqslant 1$ depending on $\varepsilon$.
%%%%%%%%%%%%%%%%%%%%%%%%%%%%%%%%%%%%%%%%%%%%%%%%%%%%%%%%%%%%%%%%%%%%%%%%%%%%%%%%%%%%%%%%%%
%%%%%%%%%%					PROPOSITION FOR SHORT INTERVALS					  %%%%%%%%%%
%%%%%%%%%%%%%%%%%%%%%%%%%%%%%%%%%%%%%%%%%%%%%%%%%%%%%%%%%%%%%%%%%%%%%%%%%%%%%%%%%%%%%%%%%%
\begin{proposition}\label{PropositionShortInterval} For any $\varepsilon>0$, there exists $B\geqslant 1$, depending only on $\varepsilon$, such that for any $\tau\in\Rbf$ we have
\begin{equation}
\int_{I(\tau,q)}\left|\Scal_V(\omega,it,K;q)\right|^2dt \ll_\varepsilon q^{\varepsilon}(1+|\tau|)^B M(P,Q;q)^2,
\end{equation}
where the implied constant depends only on $\varepsilon$.
\end{proposition}
\begin{proof}Using Lemma \ref{LemmaLowerBound} and \eqref{FullSpectrumEisenstein}, we obtain
\begin{alignat*}{1}
\frac{L}{\log^2 L}\int_{I(\tau,q)}& \left|\Scal_V(\omega,it,K;q)\right|^2dt \ll \int_{I(\tau,q)}\frac{(1+|\tau|)^{2b+2}}{(1+|t|)^{2b+2-\varepsilon}}\left|\Scal_V(\omega,it,K;q)\right|^2\left|B_\tau(\omega,it)\right|^2dt \\ \ll & \ (1+|\tau|)^{2b+2}M^3\left\{ q^{2+\varepsilon}L^{1/2}P(P+Q)+q^{3/2}L^2PQ^2(P+Q)\right\},
\end{alignat*}
and we conclude as in \cite[Section 4.2]{twists} by choosing 
$$L= \frac{1}{2}q^{1/4-\varepsilon}Q^{-1},$$
and $B=2b+2$ which depends on $\varepsilon$.
\end{proof}
The last step is to derive a pointwise bound for $\Scal_V(\omega,it,K;q)$. For this, we separate the variables $n,m$ in the twisted divisor function $\lambda_\omega(n,it)$ and using a partition of unity, we can decompose $\Scal_V(\omega,it,K;q)$ into $O(\log q)$ sums of the shape
$$\Scal_{V,M,N}(\omega,it,K;q)= \sum_{n,m\geqslant 1}K(mn)\omega(m)\left(\frac{m}{n}\right)^{it}W_1\left(\frac{m}{M}\right)W_2\left(\frac{n}{N}\right)V\left(\frac{nm}{q}\right),$$
where the parameters $(M,N)$ belongs to the set 
\begin{equation}\label{RangeMN}
\mathbf{P}:= \left\{ (M,N) \ | \ \frac{Pq}{4}\leqslant NM \leqslant 4Pq \ , \ \ 1\leqslant N,M\right\}
\end{equation}
and $W_1,W_2$ are smooth and compactly supported functions on $[-1/2,2]$ satisfying $x^jW_{i}^{(j)}(x)\ll_j 1$ for every $j\geqslant 0$.
It follows that 
\begin{equation}\label{Max}
\Scal_V(\omega,it,K;q)\ll \log (q) \max_{(M,N)\in\mathbf{P}}\left|\Scal_{V,M,N}(\omega,it,K;q)\right|.
\end{equation}
The relation between $\Scal_{V,M,N}(\omega,it,K;q)$ and and an average of $\Scal_V(\omega,it,K;q)$ is given through the Mellin transform (see \cite[Lemma 2.1]{prime}).
%%%%%%%%%%%%%%%%%%%%%%%%%%%%%%%%%%%%%%%%%%%%%%%%%%%%%%%%%%%%%%%%%%%%%%%%%%%%%%%%%%%%%%%%%%
%%%%%%%%%%				LEMMA RELATION THROUGH THE MELLIN TRANSFORM			  %%%%%%%%%%
%%%%%%%%%%%%%%%%%%%%%%%%%%%%%%%%%%%%%%%%%%%%%%%%%%%%%%%%%%%%%%%%%%%%%%%%%%%%%%%%%%%%%%%%%%
\begin{lemme}\label{LemmaMellin} Given $s\in\Cbf$ and $x>0$, we define
$$V_s(x):= V(x)x^{-s}.$$
Then for every $\varepsilon >0$, we have
\begin{equation}\label{EquationMellin}
\Scal_{V,M,N}(\omega,it,K;q)\ll_\varepsilon \iint_{|t_1|,|t_2|\leqslant q^{\varepsilon}}\left| \Scal_{V_{it_1}}(\omega,it_2+it,K;q)\right|dt_1dt_2 + O\left(q^{-100}\right).
\end{equation}
\end{lemme}
\begin{proof}
Using Mellin inversion formula for $W_1$ and $W_2$, we can write
$$\Scal_{V,M,N}(\omega,it,K;q)=\frac{1}{(2\pi i)^2}\int_{(0)}\int_{(0)}\widehat{W}_1(s_1)\widehat{W}_2(s_2)\mathcal{T}_V(s_1,s_2)M^{s_1}N^{s_2}ds_1ds_2,$$
where $\widehat{W}_1,\widehat{W}_2$ denote the Mellin transform of the smooth functions $W_1,W_2$ and 
$$\mathcal{T}_{V}(s_1,s_2)=\sum_{n,m\geqslant 1}K(nm)\omega(m)m^{it-s_1}n^{-it-s_2}V\left(\frac{nm}{q}\right).$$
Note that this sum can be expressed as a twist of Eisenstein series, namely
$$\mathcal{T}_V(s_1,s_2)=q^{-\theta_1}\Scal_{V_{\theta_1}}(\omega,\theta_2+it,K;q),$$
with
$$\theta_1=\frac{s_1+s_2}{2} \, \ \ \theta_2=\frac{-s_1+s_2}{2}.$$
For $\Re e (\theta_1),$ the smooth function $V_{\theta_1}$ satisfies condition $V(C,P,Q(\theta_1))$ with 
\begin{equation}\label{Q(theta1)}
Q(\theta_1)\ll Q+|\theta_1|,
\end{equation}
where the implied constant is absolute. Thus by a change of variables, we get
\begin{equation}\label{Int1}
\begin{split}
\Scal_{V,M,N}(\omega,it,K;q)=\frac{1}{(2\pi i)^2}\int\limits_{(0)}\int\limits_{(0)}&\widehat{W}_1(\theta_1-\theta_2)\widehat{W}_2(\theta_1+\theta_2)\left(\frac{M}{N}\right)^{\theta_2} \\ & \times \left(\frac{MN}{q}\right)^{\theta_1}\Scal_{V_{\theta_1}}(\omega,\theta_2+it,K;q)d\theta_1d\theta_2.
\end{split}
\end{equation}
Because we have the estimations
$$\widehat{W}_1(s),\widehat{W}_2(s)\ll \frac{1}{(1+|s|)^C},$$
with an implied constant depending on $C$ and $\Re e (s)$, we can truncate the integral \eqref{Int1} to $|\theta_1|\leqslant q^{\varepsilon}$, $|\theta_2|\leqslant q^\varepsilon$ for a cost of $O(q^{-100})$ by taking $C$ large enough in term of $\varepsilon$ and using the trivial bound for $\Re e (\theta_1)=\Re e (\theta_2)=0$
$$\Scal_{V_{\theta_1}}(\omega,\theta_2+it,K;q)\ll MPq\log q.$$
\end{proof}

%%%%%%%%%%%%%%%%%%%%%%%%%%%%%%%%%%%%%%%%%%%%%%%%%%%%%%%%%%%%%%%%%%%%%%%%%%%%%%%%%%%%%%%%
%%%%%%%%%%%%%%%%%%%%%%%%%%%%%%%%%%%%%%%%%%%%%%%%%%%%%%%%%%%%%%%%%%%%%%%%%%%%%%%%%%%%%%%%
%%%%%%%%%%								CONCLUSION							%%%%%%%%%%
%%%%%%%%%%%%%%%%%%%%%%%%%%%%%%%%%%%%%%%%%%%%%%%%%%%%%%%%%%%%%%%%%%%%%%%%%%%%%%%%%%%%%%%%
%%%%%%%%%%%%%%%%%%%%%%%%%%%%%%%%%%%%%%%%%%%%%%%%%%%%%%%%%%%%%%%%%%%%%%%%%%%%%%%%%%%%%%%%
\subsubsection{Conclusion} We are now in position to obtain the conclusion of Theorem \ref{Theorem(q,M)-good} in the Eisenstein case. Indeed, fix $\varepsilon>0$ and take $B=B(\varepsilon)$ as in Proposition \ref{PropositionShortInterval}. By \eqref{Max}, it is enough to estimate $\Scal_{V,M,N}(\omega,it,K;q)$ for $(M,N)\in\mathbf{P}$. Now let $\varepsilon'=\varepsilon/B$ such that we have the estimate \eqref{EquationMellin} of Lemma \ref{LemmaMellin}. We thus get
$$\Scal_{V,M,N}(\omega,it,K;q)\ll_{\varepsilon '} q^{\varepsilon'}\max_{|t_1|\leqslant q^{\varepsilon'}}\int_{|t_2|\leqslant q^{\varepsilon'}}\left| \Scal_{V_{it_1}}(\omega,it_2+it,K;q)\right|dt_1dt_2 + O\left(q^{-100}\right).$$
We split the above integral into $O(q^{\varepsilon'})$ integrals over intervals of length $\log^{-2} q$. For such interval $I$ centered at $\tau$, we obtain by Proposition \ref{PropositionShortInterval}, the value \eqref{Q(theta1)} and Cauchy-Schwarz inequality, the bound
$$\Scal_{V,M,N}(\omega,it,K;q)\ll q^\varepsilon(1+|\tau|)^{B/2} M(P,Q+q^{\varepsilon'};q),$$
(the function $Q\mapsto M(P,Q;q)$ is increasing). Finally, taking the maximal value $|\tau|\leqslant |t|+q^{\varepsilon'}$ yields the desire result.

%%%%%%%%%%%%%%%%%%%%%%%%%%%%%%%%%%%%%%%%%%%%%%%%%%%%%%%%%%%%%%%%%%%%%%%%%%%%%%%%%%%%%%%%
%%%%%%%%%%%%%%%%%%%%%%%%%%%%%%%%%%%%%%%%%%%%%%%%%%%%%%%%%%%%%%%%%%%%%%%%%%%%%%%%%%%%%%%%
%%%%%%%%%%%%%%%%%%%%%%%%%%%%%%%%%%%%%%%%%%%%%%%%%%%%%%%%%%%%%%%%%%%%%%%%%%%%%%%%%%%%%%%%
%%%%%%%%%%%%%%%%%%%%%%%%%%%%%%%%%%%%%%%%%%%%%%%%%%%%%%%%%%%%%%%%%%%%%%%%%%%%%%%%%%%%%%%%
%%%%%%%%%%							PROOF OF THEOREM 1						%%%%%%%%%%
%%%%%%%%%%%%%%%%%%%%%%%%%%%%%%%%%%%%%%%%%%%%%%%%%%%%%%%%%%%%%%%%%%%%%%%%%%%%%%%%%%%%%%%%
%%%%%%%%%%%%%%%%%%%%%%%%%%%%%%%%%%%%%%%%%%%%%%%%%%%%%%%%%%%%%%%%%%%%%%%%%%%%%%%%%%%%%%%%
%%%%%%%%%%%%%%%%%%%%%%%%%%%%%%%%%%%%%%%%%%%%%%%%%%%%%%%%%%%%%%%%%%%%%%%%%%%%%%%%%%%%%%%%
%%%%%%%%%%%%%%%%%%%%%%%%%%%%%%%%%%%%%%%%%%%%%%%%%%%%%%%%%%%%%%%%%%%%%%%%%%%%%%%%%%%%%%%%
\section{Proof of Theorem \ref{Theorem1}}\label{SectionProof1}

%%%%%%%%%%%%%%%%%%%%%%%%%%%%%%%%%%%%%%%%%%%%%%%%%%%%%%%%%%%%%%%%%%%%%%%%%%%%%%%%%%%%%%%%
%%%%%%%%%%%%%%%%%%%%%%%%%%%%%%%%%%%%%%%%%%%%%%%%%%%%%%%%%%%%%%%%%%%%%%%%%%%%%%%%%%%%%%%%
%%%%%%%%%%%%%%%%%%%%%%%%%%%%%%%%%%%%%%%%%%%%%%%%%%%%%%%%%%%%%%%%%%%%%%%%%%%%%%%%%%%%%%%%
%%%%%%%%%%						THE EISENSTEIN CASE						    %%%%%%%%%%
%%%%%%%%%%%%%%%%%%%%%%%%%%%%%%%%%%%%%%%%%%%%%%%%%%%%%%%%%%%%%%%%%%%%%%%%%%%%%%%%%%%%%%%%
%%%%%%%%%%%%%%%%%%%%%%%%%%%%%%%%%%%%%%%%%%%%%%%%%%%%%%%%%%%%%%%%%%%%%%%%%%%%%%%%%%%%%%%%
%%%%%%%%%%%%%%%%%%%%%%%%%%%%%%%%%%%%%%%%%%%%%%%%%%%%%%%%%%%%%%%%%%%%%%%%%%%%%%%%%%%%%%%%
\subsection{The Eisenstein case}\label{SectionEisenstein}
It is natural to separate the sum in \eqref{DefinitionMomentDirichlet} into even and odd primitives characters because they have different gamma factors in their functional equations. We will treat only the case of even characters since the odd case is completely similar. We therefore consider 
\begin{equation}\label{TwistedEven}
\mathscr{T}_{\mathrm{even}}^3(\omega_1,\omega_2,\ell;q):=\frac{2}{q-1}\sideset{}{^+}\sum_{\substack{\chi \ (\modm q) \\ \chi\neq 1,\overline{\omega}_1,\overline{\omega}_2}}L\left(\chi,\frac{1}{2}\right)L\left(\chi\omega_1,\frac{1}{2}\right)L\left(\chi\omega_2,\frac{1}{2}\right)\chi(\ell).
\end{equation}

%%%%%%%%%%%%%%%%%%%%%%%%%%%%%%%%%%%%%%%%%%%%%%%%%%%%%%%%%%%%%%%%%%%%%%%%%%%%%%%%%%%%%%%%
%%%%%%%%%%%%%%%%%%%%%%%%%%%%%%%%%%%%%%%%%%%%%%%%%%%%%%%%%%%%%%%%%%%%%%%%%%%%%%%%%%%%%%%%
%%%%%%%%%			APPLYING THE APPROXIMATE FUNCTIONAL EQUATION				%%%%%%%%%%	%%%%%%%%%%%%%%%%%%%%%%%%%%%%%%%%%%%%%%%%%%%%%%%%%%%%%%%%%%%%%%%%%%%%%%%%%%%%%%%%%%%%%%%%	
\subsubsection{Applying the approximate functional equation}\label{SectionApplicationApproximateFunctionalEquation}
Applying the approximate function equation provided by Lemma \ref{LemmeApproximate}, we decompose \eqref{TwistedEven} into two terms
$$\mathscr{T}_{\mathrm{even}}^3(\omega_1,\omega_2,\ell;q) = \mathcal{S}_1(\omega_1,\omega_2,\ell;q)+i^{\kappa_1+\kappa_2}\mathcal{S}_2(\omega_1,\omega_2,\ell;q),$$
with
\begin{equation}\label{FirstTerm}
\mathcal{S}_1(\omega_1,\omega_2,\ell;q):=\frac{2}{q-1}\sideset{}{^+}\sum_{\substack{\chi \  (\modm q) \\ \chi\neq 1,\overline{\omega}_1,\overline{\omega}_2}}\sideset{}{^*}\sum_{n_0,n_1,n_2\geqslant 1}\frac{\chi(n_0n_1n_2\ell)\omega_1(n_1)\omega_2(n_2)}{(n_0n_1n_2)^{1/2}}\mathbf{V}_\chi\left(\frac{n_0n_1n_2}{q^{3/2}}\right),
\end{equation}
and
\begin{equation}\label{FourthTerm}
\begin{split}
\mathcal{S}_2(\omega_1,\omega_2,\ell;q):=\frac{2}{q-1}\sideset{}{^+}\sum_{\substack{\chi \  (\modm q) \\ \chi\neq 1,\overline{\omega}_1,\overline{\omega}_2}}\sideset{}{^*}\sum_{n_0,n_1,n_2\geqslant 1}&\frac{\chi(\overline{n_0n_1n_2}\ell)\overline{\omega}_1(n_1)\overline{\omega}_2(n_2)}{(n_0n_1n_2)^{1/2}} \\ & \ \times\varepsilon(\chi)\varepsilon(\chi\omega_1)\varepsilon(\chi\omega_2)\mathbf{V}_\chi\left(\frac{n_0n_1n_2}{q^{3/2}}\right),
\end{split}
\end{equation}
where the symbol $^*$ over the $n_i's$ sum means that we restrict to $(n_0n_1n_2,q)=1$ and the function $\mathbf{V}_\chi$ is defined in \eqref{DefinitionV2}. In particular, since we sum over even characters, this function is constant on the average and we write $\mathbf{V}$ instead of $\mathbf{V}_\chi.$
\begin{rmk}\label{RemarkFastDecay}
The function $\mathbf{V}$ has rapid decay at infinity by Remark \ref{Remark}, so the sums \eqref{FirstTerm}-\eqref{FourthTerm} are essentially supported on $1\leqslant n_0n_1n_2\leqslant q^{3/2+\varepsilon}$. It follows that the sum over $n_0,n_1,n_2$ is trivially bounded by $O(q^{3/4+\varepsilon})$, so we can remove as it suits us the contribution of $\chi=1,\overline{\omega}_1$ or $\chi=\overline{\omega}_2$ for an error of size $O(q^{-1/4+\varepsilon})$. 
\end{rmk}

%%%%%%%%%%%%%%%%%%%%%%%%%%%%%%%%%%%%%%%%%%%%%%%%%%%%%%%%%%%%%%%%%%%%%%%%%%%%%%%%%%%%%%%%
%%%%%%%%%%%%%%%%%%%%%%%%%%%%%%%%%%%%%%%%%%%%%%%%%%%%%%%%%%%%%%%%%%%%%%%%%%%%%%%%%%%%%%%%
%%%%%%%%%%					AVERAGE OVER THE CHARACTERS						%%%%%%%%%%
%%%%%%%%%%%%%%%%%%%%%%%%%%%%%%%%%%%%%%%%%%%%%%%%%%%%%%%%%%%%%%%%%%%%%%%%%%%%%%%%%%%%%%%%
%%%%%%%%%%%%%%%%%%%%%%%%%%%%%%%%%%%%%%%%%%%%%%%%%%%%%%%%%%%%%%%%%%%%%%%%%%%%%%%%%%%%%%%%
\subsubsection{Average over the primitive and even characters}\label{SectionAverageCharacter} We need to average the sum over the characters in \eqref{FirstTerm}-\eqref{FourthTerm}. For this, we use some orthogonality relations asserting that for any prime $q>2$ and any integer $a$ corpime with $q$, we have (c.f. \cite[(3.2)-(3.4)]{Iw-S}) 
\begin{equation}\label{Orthogonality1}
\sideset{}{^+}\sum_{\substack{\chi \ (\modm q) \\ \chi\neq 1}}\chi(a)=\frac{q-1}{2}\delta_{a\equiv \pm 1 (q)}-1,
\end{equation}
and for $\kappa\in\{0,1\},$
\begin{equation}\label{Orthogonality2}
\sideset{}{^\kappa}\sum_{\substack{\chi \ (\modm q) \\ \chi\neq 1}}\chi(m)\varepsilon(\chi) = \frac{q-1}{2q^{1/2}}\sum_{\pm}(\pm 1)^{\kappa}\left(e\left(\pm \frac{\overline{m}}{q}\right)+\frac{1}{(q-1)}\right),
\end{equation}
where the supscript $\kappa$ means that we sum over $\chi$ such that $\chi(-1)=(-1)^\kappa$. In \eqref{FirstTerm}, we remove the contribution of $\chi=\overline{\omega}_1,\overline{\omega}_2$ (see Remark \ref{RemarkFastDecay}) and after applying \eqref{Orthogonality1}, we get $\mathcal{S}_1=\mathcal{S}_1^++\mathcal{S}_1^- + O(q^{-1/4+\varepsilon})$ with 
\begin{equation}\label{Average1}
\mathcal{S}_1^\pm(\omega_1,\omega_2,\ell;q)=\mathop{\sum\sum\sum}_{\substack{n_0,n_1,n_2\geqslant 1 \\ n_0n_1n_2\ell\equiv \pm 1 \ (\modm q)}}\frac{\omega_1(n_1)\omega_2(n_2)}{(n_0n_1n_2)^{1/2}}\mathbf{V}\left(\frac{n_0n_1n_2}{q^{3/2}}\right).
\end{equation}
For \eqref{FourthTerm}, we remove the contribution of $\chi=1,\overline{\omega}_2$ and note that for $(m,q)=1$ we have, opening the Gauss sum $\varepsilon(\chi\omega_2)$ and using \eqref{Orthogonality2},
\begin{equation}\label{AverageTwoGaussSums}
\begin{split}
\frac{2}{q-1}\sideset{}{^+}\sum_{\chi\neq\omegabar_1}\chi(m)\varepsilon(\chi\omega_1)&\varepsilon(\chi\omega_2)=\frac{1}{q^{1/2}}\sum_{a\in\Fbf_q^\times}\omega_2(a)e\left(\frac{a}{q}\right)\left(\frac{2}{q-1}\sideset{}{^+}\sum_{\chi\neq\omegabar_1}\chi(am)\varepsilon(\chi\omega_1)\right) \\ 
= & \ \frac{\omegabar_1(m)}{q^{1/2}}\sum_{a\in\Fbf_q^\times}\omegabar_1\omega_2(a)e\left(\frac{a}{q}\right)\left(\frac{2}{q-1}\sideset{}{^{\kappa_1}}\sum_{\chi\neq 1}\chi(am)\varepsilon(\chi)\right) \\
= & \ \frac{1}{q^{1/2}}\sum_{\pm}\omegabar_1(\pm m)\left(\frac{1}{q^{1/2}}\sum_{a\in\Fbf_q^\times}\omegabar_1\omega_2(a)e\left(\frac{a}{q}\right)\left(e\left(\frac{\pm\overline{am}}{q}\right)+\frac{1}{q-1}\right)\right).
\end{split}
\end{equation} 
The second expression in the above parenthesis is easily computed as a Gauss sum. For the first term, we have 
\begin{alignat*}{1}
\frac{1}{q^{1/2}}\sum_{a\in\Fbf_q^\times}\omegabar_1\chi_2(a)e\left(\frac{a}{q}\right)e\left(\frac{\pm\overline{am}}{q}\right) = & \ \omega_1\omegabar_2(\pm m)\frac{1}{q^{1/2}}\sum_{a\in\Fbf_q^\times}\omegabar_1\omega_2(a)e\left(\frac{\overline{a}}{q}\right)e\left(\frac{\pm \overline{m}a}{q}\right) \\ 
= & \ \omega_1(\pm m)\Kl_2(\pm\overline{m},\omega_1,\omega_2;q),
\end{alignat*}
where the twisted rank $2$ Kloosterman sum is defined by \eqref{DefinitionKloosterman} (see also \eqref{Kloosterman-Fourier}). Hence we see that \eqref{AverageTwoGaussSums} equals
\begin{equation}\label{DoubleGauss}
\frac{2}{q-1}\sideset{}{^+}\sum_{\chi\neq\omegabar_1}\chi(m)\varepsilon(\chi\omega_1)\varepsilon(\chi\omega_2)=\frac{1}{q^{1/2}}\Kl_2(\pm\overline{m},\omega_1,\omega_2;q)+\frac{\varepsilon(\omegabar_1\omega_2)\omegabar_1(m)(1+(-1)^{\kappa_1})}{q^{1/2}(q-1)}.
\end{equation}
Now opening the Gauss sum $\varepsilon(\chi)$ and using \eqref{DoubleGauss}, we obtain for every $(m,q)=1$,
\begin{equation}\label{TripleGauss}
\begin{split}
\frac{2}{q-1}\sideset{}{^+}\sum_{\chi\neq\omegabar_1}\chi(m)\varepsilon(\chi)\varepsilon(\chi\omega_1)\varepsilon(\chi\omega_2)= & \ \frac{1}{q^{1/2}}\sum_{a\in\Fbf_q^\times}e\left(\frac{a}{q}\right)\left(\frac{2}{q-1}\sideset{}{^+}\sum_{\chi\neq\omegabar_1}\chi(am)\varepsilon(\chi\omega_1)\varepsilon(\chi\omega_2)\right) \\ 
= & \ \frac{1}{q}\sum_{a\in\Fbf_q^\times}\Kl_2(\pm \overline{am},\omega_1,\omega_2;q)e\left(\frac{a}{q}\right)+O\left(q^{-3/2}\right) \\ 
= & \ \frac{1}{q^{1/2}}\Kl_3(\pm \overline{m},\omega_1,\omega_2,1;q)+O\left(q^{-3/2}\right).
\end{split}
\end{equation}
Finally, applying \eqref{TripleGauss} in \eqref{FourthTerm} with $m=\overline{n_0n_1n_2}\ell$ yields $\Scal_2=\Scal_2^++\Scal_2^-+O(q^{-1/4+\varepsilon})$ (recall Remark \ref{RemarkFastDecay}) with
\begin{equation}\label{FinalValueS2}
\Scal_2^{\pm}(\omega_1,\omega_2,\ell;q)= \frac{1}{q^{1/2}}\sumstar_{n_0,n_1,n_2 \geqslant 1}\frac{\omegabar_1(n_1)\omegabar_2(n_2)}{(n_0n_1n_2)^{1/2}}\Kl_3(\pm n_0n_1n_2 \overline{\ell},\omega_1,\omega_2,1;q)\mathbf{V}\left(\frac{n_0n_1n_2}{q^{3/2}}\right).
\end{equation}
We will evaluate each of these two terms (\eqref{Average1}-\eqref{FinalValueS2}) separately and find that a main term appears only in $\mathcal{S}_1^+(\omega,\omega_2,\ell;q)$ when $\ell=1$. The others will contribute as an error term.

%%%%%%%%%%%%%%%%%%%%%%%%%%%%%%%%%%%%%%%%%%%%%%%%%%%%%%%%%%%%%%%%%%%%%%%%%%%%%%%%%%%%%%%%
%%%%%%%%%%%%%%%%%%%%%%%%%%%%%%%%%%%%%%%%%%%%%%%%%%%%%%%%%%%%%%%%%%%%%%%%%%%%%%%%%%%%%%%%
%%%%%%%%%%							THE MAIN TERM							%%%%%%%%%%
%%%%%%%%%%%%%%%%%%%%%%%%%%%%%%%%%%%%%%%%%%%%%%%%%%%%%%%%%%%%%%%%%%%%%%%%%%%%%%%%%%%%%%%%
%%%%%%%%%%%%%%%%%%%%%%%%%%%%%%%%%%%%%%%%%%%%%%%%%%%%%%%%%%%%%%%%%%%%%%%%%%%%%%%%%%%%%%%%
\subsubsection{The Main Term}\label{SectionMainTerm} The main contribution comes from $n_0=n_1=n_2=\ell=1$ in \eqref{Average1}. Indeed, assuming $n_0n_1n_2\ell=1$, we obtain by the Remark \ref{Remark}
$$\mathbf{V}\left(\frac{1}{q^{3/2}}\right)=1+O\left(q^{-3/4+\varepsilon}\right).$$
When $n_0n_1n_2\ell\equiv \pm 1$ (mod $q$) with $n_0n_1n_2\ell\neq 1$, we write the congruence equation in the form $n_0n_1n_2\ell = \pm 1 + kq$ with $1\leqslant k\leqslant \ell q^{1/2+\varepsilon}+1$. Therefore, we get that the contribution of $n_0n_1n_2\ell\neq 1$ is at most
$$\ell^{1/2}q^{\varepsilon-1/2}\sum_{1\leqslant k\leqslant \ell q^{1/2+\varepsilon}+1}\frac{1}{k^{1/2}}\ll \ell q^{-1/4+\varepsilon}.$$
We conclude with
$$\mathcal{S}_1^+(\omega_1,\omega_2,\ell;q)=\delta_{\ell=1}+O\left(\ell q^{-1/4+\varepsilon}\right) \ \ \mathrm{and} \ \ \mathcal{S}_1^-(\omega_1,\omega_2,\ell;q)=O\left(\ell q^{-1/4+\varepsilon}\right),$$
which gives the desired main term of Theorem \ref{Theorem1} provided
\begin{equation}\label{Condition1ell}
\ell\leqslant q^{\frac{1}{4}-\frac{1}{64}}=q^{\frac{15}{64}}.
\end{equation}

%%%%%%%%%%%%%%%%%%%%%%%%%%%%%%%%%%%%%%%%%%%%%%%%%%%%%%%%%%%%%%%%%%%%%%%%%%%%%%%%%%%%%%%%
%%%%%%%%%%%%%%%%%%%%%%%%%%%%%%%%%%%%%%%%%%%%%%%%%%%%%%%%%%%%%%%%%%%%%%%%%%%%%%%%%%%%%%%%
%%%%%%%%%%							THE ERROR TERM							%%%%%%%%%%
%%%%%%%%%%%%%%%%%%%%%%%%%%%%%%%%%%%%%%%%%%%%%%%%%%%%%%%%%%%%%%%%%%%%%%%%%%%%%%%%%%%%%%%%
%%%%%%%%%%%%%%%%%%%%%%%%%%%%%%%%%%%%%%%%%%%%%%%%%%%%%%%%%%%%%%%%%%%%%%%%%%%%%%%%%%%%%%%%
\subsubsection{The error term}\label{SectionErrorTerm}In this section, we analyze the expression \eqref{FinalValueS2} and will find that it contributes as an error term. Applying a partition of unity to $[1,\infty)^3$ in order to locate the variables $n_0,n_1,n_2$ and we obtain 
$\mathcal{S}_2^\pm (\omega_1,\omega_2,\ell;q)=\sum_{N_0,N_1,N_2}\mathcal{S}_2^\pm (\ell,N_0,N_1,N_2;q)$ with 
\begin{equation}\label{LastTerm}
\begin{split}
\mathcal{S}_2^\pm (\ell,N_0,N_1,N_2;q)= & \ \frac{1}{(qN_0N_1N_2)^{1/2}}\sideset{}{^*}\sum_{n_0,n_1,n_2\in\mathbf{Z}}\overline{\omega}_1(n_1)\overline{\omega}_2(n_2)f_1\left(\frac{n_1}{N_1}\right)f_2\left(\frac{n_2}{N_2}\right)\\ & \times  \mathrm{Kl}_{3}(\pm n_0n_1n_2\overline{\ell},\omega_1,\omega_2,1;q) f_0\left(\frac{n_0}{N_0}\right)\mathbf{V}\left(\frac{n_0n_1n_2}{q^{3/2}}\right),
\end{split}
\end{equation}
where the functions $f_i$ are smooth and compactly supported on $(1/2,2)$ and the $N_i'$s runs over real numbers of the form $2^i$, $i\geqslant 0$. By the fast decay at infinity of $\mathbf{V}$, we can assume that 
$$1\leqslant N_0,N_1,N_2 \ \ \mathrm{and} \ \  N_0N_1N_2\leqslant q^{3/2+\varepsilon}.$$
Hence it remains to bound $O(\log^3 q)$ sums of the shape \eqref{LastTerm}. It is also convenient to separate the variables $n_0n_1n_2$ in the test function $\mathbf{V}$. This can be done using its integral representation \eqref{DefinitionV2} and we refer to \cite[Section 4.1]{zacharias} for more details. We keep the same notation $\mathcal{S}_2^\pm (\ell,N_0,N_1,N_2;q)$, but with the factor $\mathbf{V}$ removed, and also for the functions $f_i$, i.e.
\begin{equation}\label{LastTerm2}
\begin{split}
\mathcal{S}_2^\pm (\ell,N_0,N_1,N_2;q)=\frac{1}{(qN_0N_1N_2)^{1/2}}\sideset{}{^*}\sum_{n_0,n_1,n_2\in\mathbf{Z}}&\overline{\omega}_1(n_1)\overline{\omega}_2(n_2)\mathrm{Kl}_{3}(\pm n_0n_1n_2\overline{\ell},\omega_1,\omega_2,1;q)  \\ & \ \times   f_0\left(\frac{n_0}{N_0}\right)f_1\left(\frac{n_1}{N_1}\right)f_2\left(\frac{n_2}{N_2}\right),
\end{split}
\end{equation}
with 
\begin{equation}\label{fi}
x^j f_i^{(j)}(x) \ll_j q^{\varepsilon j}.
\end{equation}
Note finally that the trivial estimate is
\begin{equation}\label{TrivialEstimateS4}
\mathcal{S}_2^{\pm}(\ell,N_0,N_1,N_2;q)\ll \left(\frac{N_0N_1N_2}{q}\right)^{1/2}.
\end{equation}

%%%%%%%%%%%%%%%%%%%%%%%%%%%%%%%%%%%%%%%%%%%%%%%%%%%%%%%%%%%%%%%%%%%%%%%%%%%%%%%%%%%%%%%%
%%%%%%%%%%%%%%%%%%%%%%%%%%%%%%%%%%%%%%%%%%%%%%%%%%%%%%%%%%%%%%%%%%%%%%%%%%%%%%%%%%%%%%%%
%%%%%%%%%						POLYA-VINOGRADOV BOUND						%%%%%%%%%%
%%%%%%%%%%%%%%%%%%%%%%%%%%%%%%%%%%%%%%%%%%%%%%%%%%%%%%%%%%%%%%%%%%%%%%%%%%%%%%%%%%%%%%%%
%%%%%%%%%%%%%%%%%%%%%%%%%%%%%%%%%%%%%%%%%%%%%%%%%%%%%%%%%%%%%%%%%%%%%%%%%%%%%%%%%%%%%%%%
\subsubsection{Poly\'a-Vinogradov bound}\label{SectionPolya} We show here that \eqref{LastTerm2} is very small if we assume that one of the parameters $N_i$ is greater than $q$. Indeed, since the argument is the same, we suppose that $N_1\geqslant q/2$. In this case, for fixed $(n_0n_2,q)=1$ we focus on the $n_1$-sum 
$$\mathcal{P}(N_1;q) = \sumstar_{n_1\in\Zbf}\omegabar(n_1)\Kl_3(\pm n_0n_1n_2\overline{\ell},\omega_1,\omega_2,1;q)f_1\left(\frac{n_1}{N_1}\right).$$
By Remark \ref{RemarkExtensionByZero}, we can add the contribution of $q|n_1$ for an error of size $O(N_1/q)$ (since $N_1\geqslant q/2$). Hence, applying Proposition \ref{PropositionPolya} with the Fourier trace sheaf
$$\Lcal_{\omegabar}\otimes \left[\times \left(\pm n_0n_2\overline{\ell}\right)\right]^*\Klcal_3(\omega_1,\omega_2,1;q),$$
we get
$$\mathcal{P}(N_1;q)=O\left(q^{\varepsilon}\frac{N_1}{q^{1/2}}+\frac{N_1}{q}\right)=O\left(q^{\varepsilon}\frac{N_1}{q^{1/2}}\right).$$
Finally, averaging trivially over $n_0$ and $n_2$ in \eqref{LastTerm2} yields
$$\Scal_2^\pm(\ell,N_0,N_1,N_2;q)\ll q^\varepsilon \left(\frac{N_0N_1N_2}{q^2}\right)^{1/2}.$$
Since $N_0N_1N_2\leqslant q^{3/2+\varepsilon}$, we obtain
\begin{proposition}\label{PropositionPolyBound} Assume that $N_i\geqslant q/2$ for some $i\in\{0,1,2\}.$ Then for any $\varepsilon>0$, we have
$$\Scal_2^\pm(\ell,N_0,N_1,N_2;q) = O\left(q^{-1/4+\varepsilon}\right),$$
with an implied constant depending only on $\varepsilon$.
\end{proposition}

%%%%%%%%%%%%%%%%%%%%%%%%%%%%%%%%%%%%%%%%%%%%%%%%%%%%%%%%%%%%%%%%%%%%%%%%%%%%%%%%%%%%%%%%
%%%%%%%%%%%%%%%%%%%%%%%%%%%%%%%%%%%%%%%%%%%%%%%%%%%%%%%%%%%%%%%%%%%%%%%%%%%%%%%%%%%%%%%%
%%%%%%%%%%					POISSON IN THE THREE VARIABLES					%%%%%%%%%%
%%%%%%%%%%%%%%%%%%%%%%%%%%%%%%%%%%%%%%%%%%%%%%%%%%%%%%%%%%%%%%%%%%%%%%%%%%%%%%%%%%%%%%%%
%%%%%%%%%%%%%%%%%%%%%%%%%%%%%%%%%%%%%%%%%%%%%%%%%%%%%%%%%%%%%%%%%%%%%%%%%%%%%%%%%%%%%%%%
\subsubsection{Applying Poisson summation in the three variables}\label{Section3Poisson} In this section, we obtain an estimate for $\mathcal{S}_2^\pm$ which is satisfactory if the product of the three variables $N_0N_1N_2$ is greater that $q$. This can be done using successive applications of Poisson summation formula. Before this, we just note that by Proposition \ref{PropositionPolyBound}, we can assume that $N_i<q/2$ for $i=0,1,2$, which allows us to ignore the primality condition $(n_0n_1n_2,q)=1$ in \eqref{LastTerm2} since we also have $N_i\geqslant 1$.

\vspace{0.1cm}

We begin with the $n_0$-variable. In \eqref{LastTerm2}, we write the Kloosterman sum as the Fourier transform of the function $x\mapsto \mathbf{K}_{2}(x,\omega_1,\omega_2;q)$ defined in \eqref{DefinitionKbf}. Hence, an application of Poisson summation in $n_0$ and Fourier inversion formula gives (recall that $(n_1n_2,q)=1$)
\begin{alignat*}{1}
\sum_{n_0\in \Zbf}\widehat{\mathbf{K}_2(\omega_1,\omega_2;q)}(\pm n_0n_1n_2\overline{\ell})f_0\left(\frac{n_0}{N_0}\right) = & \ \frac{N_0}{q^{1/2}}\sum_{n_0\in\Zbf}\widehat{\widehat{\mathbf{K}_2(\omega_1,\omega_2;q)}}(\pm n_0\overline{n_1n_2}\ell)\widehat{f_0}\left(\frac{n_0 N_0}{q}\right) \\ 
= & \ \frac{N_0}{q^{1/2}}\sum_{n_0\in\Zbf}\mathbf{K}_2(\mp n_0\overline{n_1n_2}\ell,\omega_1,\omega_2;q)\widehat{f_0}\left(\frac{n_0N_0}{q}\right).
\end{alignat*}
Since by Definition $\mathbf{K}_2(x)=0$ for $q|x$, we obtain
\begin{equation}\label{Poisson2}
\begin{split}
\Scal_{2}^\pm (\ell,N_0,N_1,N_2;q)= \frac{N_0^{1/2}}{q(N_1N_2)^{1/2}}\sum_{\substack{n_0,n_1,n_2\in\Zbf \\ (n_0,q)=1}}&\omegabar_1(n_1)\omegabar_2(n_2)\Kl_2(\mp n_1n_2\overline{n_0\ell},\omega_1,\omega_2;q) \\
 & \ \times \widehat{f_0}\left(\frac{n_0N_0}{q}\right)f_1\left(\frac{n_1}{N_1}\right)f_2\left(\frac{n_2}{N_2}\right).
\end{split}
\end{equation}
We continue with the $n_2$-variable. As before, since the argument of $\Kl_2$ is non zero modulo $q$, we can express the Kloosterman sum as suitable Fourier transform, namely (see \eqref{DefinitionKbf})
\begin{equation}\label{FourierKl2}
\Kl_2(\mp n_1n_2\overline{\ell},\omega_1,\omega_2;q)=\omega_2(\mp n_1n_2\overline{n_0\ell})\widehat{\left[\omega_2\mathbf{K}_1(\omega_1;q)\right]}(\mp n_1n_2\overline{n_0\ell}).
\end{equation}
Using exactly the same argument as before, after replacing $\Kl_2$ by \eqref{FourierKl2} in \eqref{Poisson2}, we get
\begin{alignat*}{1}
\Scal_2^{\pm}(\ell,N_0,N_1,N_2;q)=\frac{\omegabar_2(\mp\ell)N_0^{1/2}}{q(N_1N_2)^{1/2}}\sum_{\substack{n_0,n_1,n_2 \in\Zbf \\ (n_0n_2,q)=1}} & \omegabar_1\chi_2(n_1)\omegabar_2(n_0)[\widehat{\omega_2\mathbf{K}_1(\omega_1)}](\mp n_1n_2\overline{n_0\ell}) \\ 
\times & \ \widehat{f_0}\left(\frac{n_0N_0}{q}\right)f_1\left(\frac{n_1}{N_1}\right)f_2\left(\frac{n_2}{N_2}\right).
\end{alignat*}
Applying Poisson in the $n_2$-variable yields 
\begin{alignat*}{1}
\sum_{n_2\in\Zbf}\widehat{\omega_2\mathbf{K}_1(\omega_1)}(\mp n_1n_2\overline{n_0\ell})f_2\left(\frac{n_2}{N_2}\right)= & \ \frac{N_2}{q^{1/2}}\sum_{n_2\in\Zbf}\widehat{\widehat{\omega_2\mathbf{K}_1(\omega_1)}}(\mp n_2\overline{n_1}n_0\ell)\widehat{f_2}\left(\frac{n_2N_2}{q}\right) \\ = & \ \frac{N_2}{q^{1/2}}\sum_{n_2\in\Zbf}\omega_2(\pm n_2\overline{n_1}n_0\ell)\mathbf{K}_1(\pm n_2\overline{n_1}n_0\ell,\omega_1;q)\widehat{f_2}\left(\frac{n_2N_2}{q}\right) \\ 
= & \ \frac{N_2}{q^{1/2}}\sum_{(n_2,q)=1}\omegabar_1\omega_2(\pm n_2\overline{n_1}n_0\ell)e\left(\frac{\pm n_1\overline{n_2n_0\ell}}{q}\right)\widehat{f_2}\left(\frac{n_2N_2}{q}\right).
\end{alignat*}
Hence 
\begin{alignat}{1}
\Scal_2^\pm(\ell,N_0,N_1,N_2;q)=\omegabar_1(\pm\ell)\omega_2(-1)\left(\frac{N_0N_2}{q^3N_1}\right)^{1/2}\sum_{\substack{ n_0,n_1,n_2 \in\Zbf \\ (n_0n_2,q)=1}} & \omegabar_1(n_2n_0)\omega_2(n_2)e\left(\frac{\pm n_1\overline{n_2n_0\ell}}{q}\right)   \nonumber \\ \times & \ f_1\left(\frac{n_1}{N_1}\right) \widehat{f_0}\left(\frac{n_0N_0}{q}\right)\widehat{f_2}\left(\frac{n_2N_2}{q}\right) \label{S2Final}.
\end{alignat}
It remains to do Poisson in the $n_1$-variable. Let $a\in\Fbf_q$, we denote by $\bm{\delta}_{a}$ the Dirac function on $\Fbf_q$ defined by $\bm{\delta}_a(x)=1$ if $x=a$ and zero else. Then the exponential map $$n_1\mapsto e\left(\frac{\pm n_1\overline{n_0n_2\ell}}{q}\right)$$ is the additive Fourier transform of the Dirac function $x\mapsto q^{1/2}\bm{\delta}_{\pm\overline{n_0n_2\ell}}(x)$. It follows that the $n_1$ sum in \eqref{S2Final} equals
$$\sum_{n_1\in\Zbf}\widehat{q^{1/2}\bm{\delta}_{\pm\overline{n_0n_2\ell}}}(n_1)f_1\left(\frac{n_1}{N_1}\right)=N_1 \sum_{n_1\in\Zbf}\bm{\delta}_{\pm\overline{n_0n_2\ell}}(-n_1)\widehat{f_1}\left(\frac{n_1N_1}{q}\right).$$
Summarizing all the previous computations yields the bound
$$\Scal_{2}^\pm(\ell,N_0,N_1,N_2;q)\ll \left(\frac{N_0N_1N_2}{q^3}\right)^{1/2}\sum_{n_0n_1n_2\ell\equiv \mp 1 (\modm q)}\left|\widehat{f_0}\left(\frac{n_0N_0}{q}\right)\widehat{f_1}\left(\frac{n_1N_1}{q}\right)\widehat{f_2}\left(\frac{n_2N_2}{q}\right)\right|.$$
Finally, using the fact that all these Fourier transform have fast decay at infinity, we see that the above sum is essentially supported on $|n_i|\leqslant \frac{q^{1+\varepsilon}}{N_i}$ and thus, a trivial estimate leads to
\begin{equation}\label{FinalBoundPoisson}
\Scal_2^{\pm}(\ell,N_0,N_1,N_2;q)\ll q^\varepsilon\left(\frac{N_0N_1N_2}{q^3}\right)^{1/2}\left(\frac{q^2}{N_0N_1N_2}+1\right)\ll q^\varepsilon \left(\frac{q}{N_0N_1N_2}\right)^{1/2}.
\end{equation}

%%%%%%%%%%%%%%%%%%%%%%%%%%%%%%%%%%%%%%%%%%%%%%%%%%%%%%%%%%%%%%%%%%%%%%%%%%%%%%%%%%%%%%%%
%%%%%%%%%%%%%%%%%%%%%%%%%%%%%%%%%%%%%%%%%%%%%%%%%%%%%%%%%%%%%%%%%%%%%%%%%%%%%%%%%%%%%%%%
%%%%%%%%%%					ESTIMATION USING THE TWO THEOREMS				%%%%%%%%%%
%%%%%%%%%%%%%%%%%%%%%%%%%%%%%%%%%%%%%%%%%%%%%%%%%%%%%%%%%%%%%%%%%%%%%%%%%%%%%%%%%%%%%%%%
%%%%%%%%%%%%%%%%%%%%%%%%%%%%%%%%%%%%%%%%%%%%%%%%%%%%%%%%%%%%%%%%%%%%%%%%%%%%%%%%%%%%%%%%
\subsubsection{Estimation of $\Scal_2^\pm$ using Theorems \ref{Theorem2} and \ref{TheoremPrime}} We return to expression \eqref{LastTerm2}. The combination of \eqref{TrivialEstimateS4} and \eqref{FinalBoundPoisson} shows that it remains to deal with the case where the product $N_0N_1N_2$ is of length about $q$. The strategy is the following : if one of the variables $N_i$ is very small, then we factorize the two others to form a new long variable and apply Theorem \ref{Theorem2} for the twist of Eisenstein series. If all $N_i$ are not too small, then it is possible to factorize two variables and form a bilinear sum in such a way that an application of Theorem \ref{TheoremPrime} is beneficial. 

\vspace{0.1cm}

We prove in this section :
%%%%%%%%%%%%%%%%%%%%%%%%%%%%%%%%%%%%%%%%%%%%%%%%%%%%%%%%%%%%%%%%%%%%%%%%%%%%%%%%%%%%%%%%
%%%%%%%%%%						FINAL PROPOSITION							%%%%%%%%%%
%%%%%%%%%%%%%%%%%%%%%%%%%%%%%%%%%%%%%%%%%%%%%%%%%%%%%%%%%%%%%%%%%%%%%%%%%%%%%%%%%%%%%%%%
\begin{proposition}\label{FinalProposition} Let $N=\max(N_0,N_1,N_2),$ $M=\min (N_0,N_1,N_2)$ and write $D$ for the remaining parameter, i.e. $M\leqslant D\leqslant N$. Then for every $\varepsilon>0$, we have
\begin{equation}\label{FinalBoundTheorem}
\Scal_{2}^\pm (\ell,N_0,N_1,N_2;q)\ll_\varepsilon q^\varepsilon \left(\frac{N_0N_1N_2}{q}\right)^{1/2} \left\{ \begin{array}{c}  \left(1+\frac{qM}{N_0N_1N_2}\right)^{1/2}q^{-1/16} \\  \\ \frac{1}{q^{1/4}}+\frac{1}{D^{1/2}}+\frac{q^{1/4}}{(NM)^{1/2}}.\end{array} \right.
\end{equation}
\end{proposition}
\vspace{0.1cm}
\begin{proof}
To fix the ideas, we assume that 
$$M=N_0 \leqslant N_1=D\leqslant N_2=N,$$
and we leave it to the reader to ensure that the other cases treated with minimal changes. 

We first focus on the $n_1,n_2$-sum in \eqref{LastTerm2} and write it in the form (recall that $N_1,N_2<q/2$ so the primality condition is satisfied)
\begin{equation}\label{n2n1-sum}
\sum_{n_1,n_2\geqslant 1}\omegabar_1\omega_2(n_1)\omegabar_2(n_1n_2)\Kl_3(n_0n_1n_2\overline{\ell},\omega_1,\omega_2;q)f_1\left(\frac{n_1}{N_1}\right)f_2\left(\frac{n_2}{N_2}\right).
\end{equation}
We show now how to transform this expression in order to obtain the same as in Corollary \ref{CorollarySchwartz}. To simplify notations, we define 
\begin{equation}\label{DefinitionTraceK}
K(n):= \omegabar_2(n)\Kl_3(\pm nn_0\overline{\ell},\omega_1,\omega_2;q).
\end{equation} 
Using Mellin inversion on $f_1$ and $f_2$ in \eqref{n2n1-sum} leads to
$$\frac{1}{(2\pi i)^2}\int\limits_{(0)}\int\limits_{(0)}\widetilde{f_1}(s_1)\widetilde{f_2}(s_2)N_1^{s_1}N_2^{s_2}\sum_{n_1,n_2\geqslant 1}\omegabar_1\omega_2(n_1)K(n_1n_2)n_1^{-s_1}n_2^{-s_2}ds_1ds_2.$$
Making the change of variables
$$\theta_1=\frac{s_1+s_2}{2} \ , \ \ \theta_2=\frac{-s_1+s_2}{2},$$
and we see that the above integral takes the form
\begin{equation}\label{aboveIntegral}
\begin{split}
\frac{2}{(2\pi i)^2}\int\limits_{(0)}\int\limits_{(0)}\widetilde{f_1}&(\theta_1-\theta_2)\widetilde{f_2}(\theta_1+\theta_2)\left(\frac{N_2}{N_1}\right)^{\theta_2} \\  \times & \sum_{n_1,n_2\geqslant 1}\omegabar_1\omega_2(n_1)\left(\frac{n_1}{n_2}\right)^{\theta_2}K(n_1n_2)\left(\frac{N_1N_2}{n_1n_2}\right)^{\theta_1}d\theta_1d\theta_2 \\ 
= & \ \frac{2}{(2\pi i)^2}\int_{(0)}\left(\frac{N_2}{N_1}\right)^{\theta_2}\sum_{n\geqslant 1}\lambda_{\omegabar_1\omega_2}(n,\theta_2)K(n)V\left(\frac{n}{N_1N_2},\theta_2\right)d\theta_2,
\end{split}
\end{equation}
where for any $x\geqslant 0$ and $\Re e (\theta_2)=0$, we defined
\begin{equation}\label{DefinitionV(theta2)}
V(x,\theta_2) := \int_{(0)}\widetilde{f_1}(\theta_1-\theta_2)\widetilde{f_2}(\theta_1+\theta_2)x^{-\theta_1}d\theta_1.
\end{equation}
Because the Mellin transforms satisfy (c.f. \eqref{fi})
\begin{equation}\label{MellinTransformBehaviour}
\widetilde{f_1}(s),\widetilde{f_2}(s) \ll \left(\frac{q^\varepsilon}{1+|s|}\right)^B,
\end{equation}
with an implied constant depending on $\varepsilon$, $B$ and $\Re e (s)$, the function $V(x,\theta_2)$ satisfies 
$$V(x,\theta_2)\ll_B \frac{1}{(1+x)^B} \ \ \mathrm{and} \ \ x^\nu V^{(\nu)}(x,\theta_2) \ll_{\nu,\varepsilon} q^{\nu\varepsilon},$$
uniformly in $\Re e (\theta_2)=0$. Since we want to estimate the inner sum in \eqref{aboveIntegral} using Theorem \ref{Theorem2} and then average trivially over the $\theta_2$-integral, we also need to control the function $V(x,\theta_2)$ with respect to the $\theta_2$-variable. By \eqref{MellinTransformBehaviour}, for any $B\geqslant 1$, we have uniformly on $x>0$ and with an implied constant depending only on $B$,
$$V(x,\theta_2)\ll \int_{(0)}\left(\frac{q^\varepsilon}{(1+|\theta_1-\theta_2|)(1+|\theta_1+\theta_2|)}\right)^Bd\theta_1.$$
Note the identity
\begin{alignat*}{1}
(1+|\theta_1-\theta_2|)(1+|\theta_1+\theta_2|)= & \ 1+|\theta_1^2-\theta_2^2|+|\theta_1-\theta_2|+|\theta_1+\theta_2| \\ \geqslant & \ 1+|\theta_1^2-\theta_2^2|+2\max(|\theta_1|,|\theta_2|).
\end{alignat*}
Hence, splitting the integral depending on whether $|\theta_1|\leqslant |\theta_2|$ or not and we get
\begin{alignat*}{1}
V(x,\theta_2) \ll & \ \int\limits_{\substack{\Re e (\theta_1)=0 \\ |\theta_1|\geqslant |\theta_2|}}\left(\frac{q^\varepsilon}{1+|\theta_1^2-\theta_2^2|+2|\theta_1|}\right)^Bd\theta_1+\int\limits_{\substack{\Re e (\theta_1)=0 \\ |\theta_1|\leqslant |\theta_2|}}\left(\frac{q^\varepsilon}{1+|\theta_1^2-\theta_2^2|+2|\theta_2|}\right)^Bd\theta_1 \\ 
\leqslant & \ \int_{|t|\geqslant |\theta_2|} \left(\frac{q^\varepsilon}{1+2|t|}\right)^Bdt+ \int_{|t|\leqslant |\theta_2|}\left(\frac{q^\varepsilon}{1+2|\theta_2|}\right)^Bdt \ll \left(\frac{q^\varepsilon}{1+|\theta_2|}\right)^{B-1}.
\end{alignat*}
Therefore, for any $\varepsilon'>0$, we obtain that \eqref{aboveIntegral} is bounded, up to a constant which depends only on $\varepsilon'$, by 
\begin{equation}\label{Maximum}
q^{\varepsilon'} \max_{\substack{|\theta_2|\leqslant q^{\varepsilon'} \\ \Re e (\theta_2)=0}}\left|\sum_{n\geqslant 1}\lambda_{\omegabar_1\omega_2}(n,\theta_2)K(n)V\left(\frac{n}{N_1N_2},\theta_2\right)\right|.
\end{equation}
We now apply Corollary \ref{CorollarySchwartz} with the Schwartz function $V(x,\theta)$ and with the sheaf
$$\Fcal:= \Lcal_{\omegabar_2}\otimes [\pm n_0\overline{\ell}]^*\Klcal_3(\omega_1,\omega_2,1;q)$$
having trace function \eqref{DefinitionTraceK}. Note that since $\Kl_3(\cdot,\omega_1,\omega_2,1;q)$ is invariant under permutation of the triple $(\omega_1,\omega_2,1)$, we have by \eqref{Kloosterman-Fourier} a geometric isomorphism
$$\Fcal\simeq [\times(\pm n_0\overline{\ell})]^*\FT\left(\Lcal_{\omega_2}\otimes [x\mapsto x^{-1}]^*\Klcal_2(\omega_1,1;q)\right)$$
and hence $\Fcal$ is not Fourier-exceptional since by Fourier inversion, its $\ell$-adic Fourier transform is a rank $2$ irreducible sheaf. It follows that for any $\varepsilon>0$, we can estimate \eqref{Maximum} by
$$q^{\varepsilon'}\max_{\substack{|\theta_2|\leqslant q^{\varepsilon'}}}(qN_1N_2)^\varepsilon (1+|\theta_2|)^A N_1N_2\left(1+\frac{q}{N_1N_2}\right)^{1/2}q^{-1/16}.$$
Choosing $\varepsilon'=\varepsilon/A$, maximizing the above quantity by setting $\theta_2=q^{\varepsilon'}$, replacing the obtained bound in \eqref{n2n1-sum} and finally, averaging trivially over $n_0$ in \eqref{LastTerm2} yields the first estimate of \eqref{FinalBoundTheorem}
\begin{equation}\label{EstimateEisenstein}
\Scal_{2}^\pm(\ell,N_0,N_1,N_2;q) \ll_\varepsilon q^\varepsilon \left(\frac{N_0N_1N_2}{q}\right)^{1/2}\left(1+\frac{q}{N_1N_2}\right)^{1/2}q^{-1/16}.
\end{equation}
For the second bound, we group together the variables $n_0n_2=m$ in \eqref{LastTerm2} and we obtain
\begin{equation}\label{Bilinear}
\Scal_2^\pm(\ell,N_0,N_1,N_2;q)=\frac{1}{(qN_0N_1N_2)^{1/2}}\sum_{n,n_1}\alpha_m\beta_{n_1}\Kl_3(\pm nn_1\overline{\ell},\omega_1,\omega_2,1;q),
\end{equation}
with 
$$\alpha_m := \sum_{n_0n_2=m}\omegabar_2(n_2)f_0\left(\frac{n_0}{N_0}\right)f_2\left(\frac{n_2}{N_2}\right) \ \ \mathrm{and} \ \ \beta_{n_1}:= \omegabar_1(n_1)f_1\left(\frac{n_1}{N_1}\right).$$
Applying Theorem \ref{TheoremPrime} (1) with $N=N_0N_2$ and $M=N_2$ gives
\begin{equation}\label{BoundBilinear}
\Scal_2^\pm (\ell,N_0,N_1,N_2;q)\ll q^\varepsilon \left(\frac{N_0N_1N_2}{q}\right)^{1/2}\left(\frac{1}{q^{1/4}}+\frac{1}{N_1^{1/2}}+\frac{q^{1/4}}{(N_0N_2)^{1/2}}\right),
\end{equation}
as wishes.
\end{proof}
%%%%%%%%%%%%%%%%%%%%%%%%%%%%%%%%%%%%%%%%%%%%%%%%%%%%%%%%%%%%%%%%%%%%%%%%%%%%%%%%%%%%%%%%
%%%%%%%%%%							END PROPOSITION							%%%%%%%%%%
%%%%%%%%%%%%%%%%%%%%%%%%%%%%%%%%%%%%%%%%%%%%%%%%%%%%%%%%%%%%%%%%%%%%%%%%%%%%%%%%%%%%%%%%

%%%%%%%%%%%%%%%%%%%%%%%%%%%%%%%%%%%%%%%%%%%%%%%%%%%%%%%%%%%%%%%%%%%%%%%%%%%%%%%%%%%%%%%%
%%%%%%%%%%%%%%%%%%%%%%%%%%%%%%%%%%%%%%%%%%%%%%%%%%%%%%%%%%%%%%%%%%%%%%%%%%%%%%%%%%%%%%%%
%%%%%%%%%%								CONCLUSION							%%%%%%%%%%
%%%%%%%%%%%%%%%%%%%%%%%%%%%%%%%%%%%%%%%%%%%%%%%%%%%%%%%%%%%%%%%%%%%%%%%%%%%%%%%%%%%%%%%%
%%%%%%%%%%%%%%%%%%%%%%%%%%%%%%%%%%%%%%%%%%%%%%%%%%%%%%%%%%%%%%%%%%%%%%%%%%%%%%%%%%%%%%%%
\subsubsection{Conclusion of the Eisenstein case} Write $N_i=q^{\mu_i}$ with $\mu_i\geqslant 0$ and let $\eta>0$ be a parameter. If $\mu_0+\mu_1+\mu_2<1-2\eta$ or $\mu_0+\mu_1+\mu_2>1+2\eta$, we use the trivial bound \eqref{TrivialEstimateS4} or the estimate \eqref{FinalBoundPoisson} to obtain
$$\Scal_{2}^\pm(\ell,N_0,N_1,N_2;q) = O\left(q^{-\eta+\varepsilon}\right).$$
We therefore assume that we are in the range
\begin{equation}\label{Range1}
1-2\eta\leqslant \mu_0+\mu_1+\mu_2\leqslant 1+2\eta.
\end{equation}
Let $\delta>0$ be an auxiliary parameter. As we already see, there is no loose of generality assuming that $\mu_0\leqslant \mu_1\leqslant \mu_2$. Suppose first that 
\begin{equation}\label{Assume1}
\mu_0\leqslant \delta.
\end{equation}
In this case, we apply \eqref{EstimateEisenstein} which, combining with \eqref{Range1} and \eqref{Assume1} gives
$$\Scal_2^\pm(\ell,N_0,N_1,N_2;q)\ll_\varepsilon q^{\varepsilon}\left(q^{\eta-\frac{1}{16}}+q^{\frac{\delta}{2}-\frac{1}{16}}\right)\ll_\varepsilon q^{-\eta+\varepsilon},$$
provided
\begin{equation}\label{Condition1}
\eta\leqslant \frac{1}{32} \ \ \mathrm{and} \ \ \delta\leqslant \frac{1}{8}-2\eta,
\end{equation}
which condition we henceforth assume to hold. 

\vspace{0.1cm}

Suppose now that we are in the case
\begin{equation}\label{Assume2}
\mu_0\geqslant \delta.
\end{equation}
The estimate \eqref{BoundBilinear} leads to
$$\Scal_2^\pm(\ell,N_0,N_1,N_2;q)\ll q^\varepsilon \left(q^{\eta-\frac{1}{4}}+q^{\frac{1}{2}(\mu_0+\mu_2-1)}+q^{\frac{1}{2}(\mu_1-\frac{1}{2})}\right).$$
The first term is clearly smaller than $q^{-\eta+\varepsilon}$ by \eqref{Condition1}. For the second, note that $\mu_1\geqslant \mu_0\geqslant \delta$ and thus, by \eqref{Range1}
$$\mu_0+\mu_2-1\leqslant 2\eta-\delta.$$
It follows that 
$$q^{\varepsilon +\frac{1}{2}(\mu_0+\mu_2-1)}\leqslant q^{\varepsilon+\eta-\frac{\delta}{2}}\leqslant q^{-\eta+\varepsilon},$$
under the assumption that
\begin{equation}\label{Condition2}
\delta\geqslant 4\eta.
\end{equation}
Finally, the combination of \eqref{Range1}, $\mu_1\leqslant \mu_2$ and \eqref{Assume2} gives
$$\mu_1\leqslant \frac{1}{2}+\eta-\frac{\delta}{2}$$
and hence
$$q^{\varepsilon+\frac{1}{2}(\mu_1-\frac{1}{2})}\leqslant q^{\varepsilon +\frac{\eta}{2}-\frac{\delta}{4}}\leqslant q^{-\eta+\varepsilon},$$
provided
\begin{equation}\label{Condition3}
\delta\geqslant 6\eta,
\end{equation}
which is more restrictive than \eqref{Condition2}. To finalize the computations, we just note that the second condition in \eqref{Condition1} and \eqref{Condition3} are simultaneously satisfied as long as $\eta\leqslant \frac{1}{64}$, which gives the correct exponent of the error term in Theorem \ref{Theorem1}.

%%%%%%%%%%%%%%%%%%%%%%%%%%%%%%%%%%%%%%%%%%%%%%%%%%%%%%%%%%%%%%%%%%%%%%%%%%%%%%%%%%%%%%%%
%%%%%%%%%%%%%%%%%%%%%%%%%%%%%%%%%%%%%%%%%%%%%%%%%%%%%%%%%%%%%%%%%%%%%%%%%%%%%%%%%%%%%%%%
%%%%%%%%%%%%%%%%%%%%%%%%%%%%%%%%%%%%%%%%%%%%%%%%%%%%%%%%%%%%%%%%%%%%%%%%%%%%%%%%%%%%%%%%
%%%%%%%%%%						THE CUSPIDAL CASE							%%%%%%%%%%
%%%%%%%%%%%%%%%%%%%%%%%%%%%%%%%%%%%%%%%%%%%%%%%%%%%%%%%%%%%%%%%%%%%%%%%%%%%%%%%%%%%%%%%%
%%%%%%%%%%%%%%%%%%%%%%%%%%%%%%%%%%%%%%%%%%%%%%%%%%%%%%%%%%%%%%%%%%%%%%%%%%%%%%%%%%%%%%%%
%%%%%%%%%%%%%%%%%%%%%%%%%%%%%%%%%%%%%%%%%%%%%%%%%%%%%%%%%%%%%%%%%%%%%%%%%%%%%%%%%%%%%%%%

\subsection{The cuspidal case}
We consider as in Section \ref{SectionEisenstein} the average over primitive and even characters (recall that the nebentypus is trivial)
\begin{equation}\label{MomentEvenCuspidal}
\Tscr_{\mathrm{even}}^3(f,\ell;q) := \frac{2}{q-1}\sideset{}{^+}\sum_{\substack{\chi \ (\modm q) \\ \chi\neq 1}}L\left(f\otimes\chi,\frac{1}{2}\right)L\left(\chi,\frac{1}{2}\right)\chi(\ell).
\end{equation}

%%%%%%%%%%%%%%%%%%%%%%%%%%%%%%%%%%%%%%%%%%%%%%%%%%%%%%%%%%%%%%%%%%%%%%%%%%%%%%%%%%%%%%%%
%%%%%%%%%%%%%%%%%%%%%%%%%%%%%%%%%%%%%%%%%%%%%%%%%%%%%%%%%%%%%%%%%%%%%%%%%%%%%%%%%%%%%%%%
%%%%%%%%%%			APPLYING THE APPROXIMATE FUNCTIONAL EQUATION				%%%%%%%%%%
%%%%%%%%%%%%%%%%%%%%%%%%%%%%%%%%%%%%%%%%%%%%%%%%%%%%%%%%%%%%%%%%%%%%%%%%%%%%%%%%%%%%%%%%
%%%%%%%%%%%%%%%%%%%%%%%%%%%%%%%%%%%%%%%%%%%%%%%%%%%%%%%%%%%%%%%%%%%%%%%%%%%%%%%%%%%%%%%%
\subsubsection{Applying the approximate functional equation} Using Proposition \ref{PropositionApproximateFunctionalEquationTwisted}, we can write \eqref{MomentEvenCuspidal} in the form
$$\Tscr_{\mathrm{even}}^3(f,\ell;q)=\Ccal_1(f,\ell;q)+\varepsilon_\infty(f,+1)\Ccal_2(f,\ell;q)$$
with
$$
\Ccal_1(f,\ell;q) = \frac{2}{q-1}\sideset{}{^+}\sum_{\substack{\chi \ (\modm q) \\ \chi\neq 1}}\sum_{n.m\geqslant 1}\frac{\lambda_f(n)\chi(nm\ell)}{(nm)^{1/2}}\mathbf{V}_{f,\chi}\left(\frac{nm}{q^{3/2}}\right),
$$
$$
\Ccal_2(f,\ell;q) = \frac{2}{q-1}\sideset{}{^+}\sum_{\substack{\chi \ (\modm q) \\ \chi\neq 1}}\sum_{n.m\geqslant 1}\frac{\overline{\lambda_f(n)}\chibar(nm)\chi(\ell)}{(nm)^{1/2}}\varepsilon(\chi)^3\mathbf{V}_{f,\chi}\left(\frac{nm}{q^{3/2}}\right),
$$
where we recall that $\mathbf{V}_{f,\chi}$ depends on $\chi$ only through its parity. Since we assume that $f$ satisfies the Ramanujan-Petersson conjecture, we have $|\lambda_f(n)|\leqslant \tau(n)$. Hence, proceeding as in § \ref{SectionAverageCharacter} for the average over the characters and writing $\mathbf{V} = \mathbf{V}_{f,\chi}$, we find $$\Ccal_i(f,\ell;q)=\sum_\pm \Ccal_i^\pm(f,\ell;q)+O\left(q^{-1/4+\varepsilon}\right),$$
where
\begin{equation}\label{Cal1}
\Ccal_1^\pm(f,\ell;q)= \sumstar_{nm\ell\equiv \pm 1 \ (\modm q)}\frac{\lambda_f(n)}{(nm)^{1/2}}\mathbf{V}\left(\frac{nm}{q^{3/2}}\right),
\end{equation}
and
\begin{equation}\label{Cal2}
\Ccal_2^\pm(f,\ell;q) = \frac{1}{q^{1/2}}\sumstar_{n,m\geqslant 1}\frac{\overline{\lambda_f(n)}}{(nm)^{1/2}}\Kl_3(\pm nm\overline{\ell};q)\mathbf{V}\left(\frac{nm}{q^{3/2}}\right).
\end{equation}

%%%%%%%%%%%%%%%%%%%%%%%%%%%%%%%%%%%%%%%%%%%%%%%%%%%%%%%%%%%%%%%%%%%%%%%%%%%%%%%%%%%%%%%%
%%%%%%%%%%%%%%%%%%%%%%%%%%%%%%%%%%%%%%%%%%%%%%%%%%%%%%%%%%%%%%%%%%%%%%%%%%%%%%%%%%%%%%%%
%%%%%%%%%%							THE MAIN TERM							%%%%%%%%%%
%%%%%%%%%%%%%%%%%%%%%%%%%%%%%%%%%%%%%%%%%%%%%%%%%%%%%%%%%%%%%%%%%%%%%%%%%%%%%%%%%%%%%%%%
%%%%%%%%%%%%%%%%%%%%%%%%%%%%%%%%%%%%%%%%%%%%%%%%%%%%%%%%%%%%%%%%%%%%%%%%%%%%%%%%%%%%%%%%
\subsubsection{The main term} The extraction of the main term is done is a similar way as in § \ref{SectionMainTerm}. We just conclude with 
$$\Ccal_1^+(f,\ell;q)=\delta_{\ell=1}+O\left(\ell q^{-1/4+\varepsilon}\right) \ , \ \ \Ccal_1^-(f,\ell;q) = O\left(\ell q^{-1/4+\varepsilon}\right).$$
Note that the error terms are $O(q^{-\frac{1}{52}+\varepsilon})$ (c.f. Theorem \ref{Theorem1}) if we assume that
\begin{equation}\label{Condition2ell}
\ell\leqslant q^{\frac{1}{4}-\frac{1}{52}}=q^{\frac{3}{13}}.
\end{equation}

%%%%%%%%%%%%%%%%%%%%%%%%%%%%%%%%%%%%%%%%%%%%%%%%%%%%%%%%%%%%%%%%%%%%%%%%%%%%%%%%%%%%%%%%
%%%%%%%%%%%%%%%%%%%%%%%%%%%%%%%%%%%%%%%%%%%%%%%%%%%%%%%%%%%%%%%%%%%%%%%%%%%%%%%%%%%%%%%%
%%%%%%%%%%							THE ERROR TERM							%%%%%%%%%%
%%%%%%%%%%%%%%%%%%%%%%%%%%%%%%%%%%%%%%%%%%%%%%%%%%%%%%%%%%%%%%%%%%%%%%%%%%%%%%%%%%%%%%%%
%%%%%%%%%%%%%%%%%%%%%%%%%%%%%%%%%%%%%%%%%%%%%%%%%%%%%%%%%%%%%%%%%%%%%%%%%%%%%%%%%%%%%%%%

\subsubsection{The error term} Applying a partition of unity to \eqref{Cal2}, removing the test function $\mathbf{V}$ using its integral representation (see § \ref{SectionErrorTerm}) and we are reduced to analyze $O(\log^2q)$ sums of the shape
\begin{equation}\label{Shape1Cal}
\Ccal_2^\pm(f,N,M;q) := \frac{1}{(qNM)^{1/2}}\sumstar_{n,m\in\Zbf}\overline{\lambda_f(n)}\Kl_3(\pm nm\overline{\ell};q)W_1\left(\frac{m}{M}\right)W_2\left(\frac{n}{N}\right),
\end{equation}
where $W_i$ are smooth and compactly supported functions on $(1/2,2)$ such that $x^jW_i^{(j)}(x)\ll_{\varepsilon,j}q^{\varepsilon j}$ for all $j\geqslant 0$ and $M,N$ are reals numbers with the standard restriction due to the fast decay of $\mathbf{V}$ at infinity
$$1\leqslant M,N \ \ \mathrm{and} \ \ NM\leqslant q^{3/2+\varepsilon}.$$
Note that the trivial bound is
\begin{equation}\label{TrivialBoundCuspidal}
\Ccal_2^\pm (f,N,M;q) \ll \left(\frac{NM}{q}\right)^{1/2}.
\end{equation}
Moreover, if $M\geqslant q/2$, then an application of Poly\'a-Vinograov method in the $m$-variable (see Proposition \ref{PropositionPolya} and § \ref{SectionPolya}) leads to
$$\Ccal_2^\pm(f,N,M;q)\ll_\varepsilon q^\varepsilon\left(\frac{NM}{q^2}\right)^{1/2} \ll q^{-1/4+\varepsilon}.$$
Hence we can suppose from now on that $M<q/2$ in such a way that the condition $(m,q)=1$ under the summation in \eqref{Cal2} is automatically satisfied.

%%%%%%%%%%%%%%%%%%%%%%%%%%%%%%%%%%%%%%%%%%%%%%%%%%%%%%%%%%%%%%%%%%%%%%%%%%%%%%%%%%%%%%%%
%%%%%%%%%%%%%%%%%%%%%%%%%%%%%%%%%%%%%%%%%%%%%%%%%%%%%%%%%%%%%%%%%%%%%%%%%%%%%%%%%%%%%%%%
%%%%%%%%%%					APPLICATION VORONOI-POISSON						%%%%%%%%%%
%%%%%%%%%%%%%%%%%%%%%%%%%%%%%%%%%%%%%%%%%%%%%%%%%%%%%%%%%%%%%%%%%%%%%%%%%%%%%%%%%%%%%%%%		%%%%%%%%%%%%%%%%%%%%%%%%%%%%%%%%%%%%%%%%%%%%%%%%%%%%%%%%%%%%%%%%%%%%%%%%%%%%%%%%%%%%%%%%
\subsubsection{Application of Poisson/Voronoi summation formula}\label{SectionVoronoi} The first step is to apply Voronoi summation formula in the $n$-variable. To get in a good position, we write the Kloosterman sum $\Kl_3$ for $(a,q)=1$ in the form
\begin{equation}\label{Klforme2}
\Kl_3(a;q) = \frac{1}{q^{1/2}}\sum_{x\in\Fbf_q^\times}\Kl_2(\overline{x};q)e\left(\frac{ax}{q}\right).
\end{equation}
Note that this definition can be extended to $a=0$ with the value
$$\Kl_3(0;q)=\frac{1}{q}\left(\sum_{x\in\Fbf_q^\times}e\left(\frac{x}{q}\right)\right)^2=\frac{1}{q}.$$
It follows that after writing $\Kl_3(\pm nm\overline{\ell};q)$ in the form \eqref{Klforme2} and adding the contribution of $q|n$ for negligible error term (of size at most $q^{-3/4+\varepsilon}$), we get
\begin{equation}\label{Voronoi1}
\begin{split}
\Ccal_2^\pm (f,N,M;q)=\frac{1}{(qNM)^{1/2}}\frac{1}{q^{1/2}}\sum_{x\in\Fbf_q^\times}&\Kl_2(\overline{x};q)\sum_{m\in\Zbf}W_1\left(\frac{m}{M}\right) \sum_{n\geqslant 1}\overline{\lambda_f(n)}e\left(\frac{\pm nm\overline{\ell}x}{q}\right)W_2\left(\frac{n}{N}\right).
\end{split}
\end{equation}
Assuming we are dealing with the plus case and applying Voronoi formula (c.f. Proposition \ref{VoronoiFormula}) to the inner sum in \eqref{Voronoi1}, we obtain
\begin{alignat*}{1}
\Ccal_2^+(f,N,M;q)=\left(\frac{N}{q^3M}\right)^{1/2}\frac{1}{q^{1/2}}\sum_{\pm }\sum_{x\in\Fbf_q^\times}\Kl_2(\overline{x};q)&\sum_{n\geqslant 1}\overline{\lambda_f(n)}W_2^{\pm}\left(\frac{nN}{q^2}\right) \\ \times & \ \sum_{m\in\Zbf}e\left(\frac{\mp n\overline{mx}\ell}{q}\right)W_1\left(\frac{m}{M}\right). 
\end{alignat*}
Changing the order of summation, making the change of variable $\overline{x}\leftrightarrow xm$ (recall that $(m,q)=1$) allows us to write
\begin{alignat}{1}
\Ccal_2^+(f,N,M;q) = \left(\frac{N}{q^3M}\right)^{1/2}\frac{1}{q^{1/2}}\sum_\pm \sum_{x\in\Fbf_q^\times}&\sum_{n\geqslant 1}\overline{\lambda_f(n)}e\left(\frac{\mp nx\ell}{q}\right)W_2^\pm \left(\frac{nN}{q^2}\right) \nonumber\\ \times & \ \sum_{m\in\Zbf}\Kl_2(xm;q)W_1\left(\frac{m}{M}\right) \label{Voronoi2}.
\end{alignat}
By Poisson formula and since $\Kl_2$ is the Fourier transform of the function defined by \eqref{DefinitionKbf}, we see that the $m$-sum in \eqref{Voronoi2} is equal to
$$\frac{M}{q^{1/2}}\sum_{(m,q)=1}e\left(-\frac{x\overline{m}}{q}\right)\widehat{W_1}\left(\frac{mM}{q}\right).$$
Replacing this identity in \eqref{Voronoi2} yields
\begin{equation}\label{DiracSymbol}
\begin{split}
\Ccal_2^+(f,N,M;q)=\left(\frac{NM}{q^3}\right)^{1/2}\sum_\pm\sum_{n\geqslant 1}\sum_{(m,q)=1}&\overline{\lambda_f(n)}\widehat{W_1}\left(\frac{mM}{q}\right)W_2^\pm\left(\frac{nN}{q^2}\right) \\ \times & \frac{1}{q}\sum_{x\in\Fbf_q^\times}e\left(x\frac{\mp n\ell-\overline{m}}{q}\right),
\end{split}
\end{equation}
with the same expression for the minus case $\Ccal_2^-$, but with $\mp$ replaced by $\pm$ in the exponential. Because of the fast decay of $\widehat{W_1}$ and $W_2^{\pm}$ at infinity (c.f. Lemma \ref{LemmaVoronoi}), the $n,m$-sum \eqref{DiracSymbol} is essentially supported on $|m|\leqslant q^{1+\varepsilon}/M$ and $|n|\leqslant q^{2+\varepsilon}/N$. In this range, we use the estimate $|\lambda_f(n)|\leqslant \tau(n)\ll_\varepsilon n^\varepsilon$ and we apply Lemma $\ref{LemmaVoronoi}$ with $\vartheta=0$ (recall that $f$ satisfies R-P-C) to bound $W_2^\pm$ by $q^\varepsilon$. Adding the contribution of $x=0$, estimating this extra factor trivially and executing the complete $x$-summation gives 
\begin{alignat*}{1}
\Ccal_2^+ (f,N,M;q) = \left(\frac{NM}{q^3}\right)^{1/2}\sum_\pm\mathop{\sum\sum}_{nm\ell\equiv \mp 1 \ (\modm q)}&\overline{\lambda_f(n)}\widehat{W_1}\left(\frac{mM}{q}\right)W_2^\pm\left(\frac{nN}{q^2}\right) \\ & + O\left(q^\varepsilon\left(\frac{q}{NM}\right)^{1/2}\right).
\end{alignat*}
Therefore, as in § \ref{Section3Poisson}, we obtain
\begin{equation}\label{BoundVoronoi}
\Ccal_2^{\pm}(f,N,M;q) \ll_\varepsilon q^\varepsilon \left(\frac{q}{NM}\right)^{1/2}.
\end{equation}

%%%%%%%%%%%%%%%%%%%%%%%%%%%%%%%%%%%%%%%%%%%%%%%%%%%%%%%%%%%%%%%%%%%%%%%%%%%%%%%%%%%%%%%%
%%%%%%%%%%%%%%%%%%%%%%%%%%%%%%%%%%%%%%%%%%%%%%%%%%%%%%%%%%%%%%%%%%%%%%%%%%%%%%%%%%%%%%%%
%%%%%%%%%%	  ESTIMATION OF CCAL2 USING BILINEAR FORMS AND l-ADIC TWISTS		%%%%%%%%%%
%%%%%%%%%%%%%%%%%%%%%%%%%%%%%%%%%%%%%%%%%%%%%%%%%%%%%%%%%%%%%%%%%%%%%%%%%%%%%%%%%%%%%%%%
%%%%%%%%%%%%%%%%%%%%%%%%%%%%%%%%%%%%%%%%%%%%%%%%%%%%%%%%%%%%%%%%%%%%%%%%%%%%%%%%%%%%%%%%

\subsubsection{Estimation of $\Ccal_2$ using bounds for bilinear forms and Theorem \ref{Theorem2}}
We finally state the analogous of Proposition \ref{FinalProposition} which is an immediate application of Theorem \ref{TheoremPrime} (1)-(2), Theorem \ref{TheoremSawin} and Corollary \ref{CorollarySchwartz}.

\begin{proposition}\label{FinalPropositionCuspidal} For any $\varepsilon>0$, the quantity defined in \eqref{Shape1Cal} satisfies
$$
\Ccal_2^\pm (f,N,M;q) \ll q^\varepsilon \left(\frac{NM}{q}\right)^{1/2}\left\{ \begin{array}{l}
\frac{1}{q^{1/4}}+\frac{1}{M^{1/2}}+\frac{q^{1/4}}{N^{1/2}}  \\  \\
\frac{1}{q^{1/2}}+\frac{q^{1/2}}{M} \\  \\ 
\left(\frac{N^2M^5}{q^3}\right)^{-1/12} \\ \\
\left(1+\frac{q}{N}\right)^{1/2}q^{-1/16},
\end{array} \right.
$$
where the implied constant depends on $\varepsilon$ and polynomially on $t_f$ in the last bound and the third bound is valid in the case where $1\leqslant N\leqslant M^2$, $M<q$ and $NM<q^{3/2}$ $($c.f. \eqref{AssumptionTheoremBilinear}$)$.
\end{proposition}

%%%%%%%%%%%%%%%%%%%%%%%%%%%%%%%%%%%%%%%%%%%%%%%%%%%%%%%%%%%%%%%%%%%%%%%%%%%%%%%%%%%%%%%%
%%%%%%%%%%%%%%%%%%%%%%%%%%%%%%%%%%%%%%%%%%%%%%%%%%%%%%%%%%%%%%%%%%%%%%%%%%%%%%%%%%%%%%%%
%%%%%%%%%%					CONCLUSION OF THE CUSPIDAL CASE					%%%%%%%%%%
%%%%%%%%%%%%%%%%%%%%%%%%%%%%%%%%%%%%%%%%%%%%%%%%%%%%%%%%%%%%%%%%%%%%%%%%%%%%%%%%%%%%%%%%
%%%%%%%%%%%%%%%%%%%%%%%%%%%%%%%%%%%%%%%%%%%%%%%%%%%%%%%%%%%%%%%%%%%%%%%%%%%%%%%%%%%%%%%%

\subsubsection{Conclusion of the cuspidal case} Fix $\eta>0$ a parameter and write $M=q^{\mu}$, $N=q^{\nu}$ with $\mu,\nu\geqslant 0$. By the trivial bound \eqref{TrivialBoundCuspidal} and \eqref{BoundVoronoi}, we can assume that 
\begin{equation}\label{AssumeCusp1}
1-2\eta \leqslant \mu+\nu \leqslant 1+2\eta,
\end{equation}
otherwise we get $\Ccal_2^\pm(f,N,M;q)=O(q^{-\eta+\varepsilon}).$ We now let $\delta_1,\delta_2,\delta_3>0$ be sufficiently small auxiliary parameters and we distinguish four cases :
\begin{enumerate}
\item[$(a)$] Assume that $\mu\leqslant \delta_1.$ In this case we apply the fourth estimate of Proposition \ref{FinalPropositionCuspidal} and we get by \eqref{AssumeCusp1}
$$\Ccal_2^\pm(f,N,M;q)\ll_{\varepsilon,t_f} q^\varepsilon\left(q^{\eta-\frac{1}{16}}+q^{\frac{\delta_1}{2}-\frac{1}{16}}\right)\leqslant q^{-\eta+\varepsilon},$$
provided
\begin{equation}\label{Provided1Cusp}
\eta\leqslant \frac{1}{32} \  \ \mathrm{and} \  \ \delta_1\leqslant \frac{1}{8}-2\eta.
\end{equation}
\item[$(b)$] If $\delta_1<\mu\leqslant \frac{1}{2}-\delta_2$, the first bound of Proposition \ref{FinalPropositionCuspidal} yields
$$\Ccal_2^\pm (f,N,M;q) \ll_\varepsilon q^\varepsilon \left(q^{\eta-\frac{1}{4}}+q^{\frac{1}{2}(\nu-1)}+q^{\frac{1}{2}(\mu-\frac{1}{2})}\right).$$
The first term is less than $q^{\varepsilon-\eta}$ since $\eta\leqslant \frac{1}{32}$. For the second, we have $\nu-1\leqslant 2\eta-\delta_1$ (use \eqref{AssumeCusp1} and $\mu\geqslant \delta_1$). Thus it is less than $q^{\varepsilon-\eta}$ under the assumption that
\begin{equation}\label{Provided2Cusp}
\delta_1\geqslant 4\eta.
\end{equation}
The third term is at most $q^{-\delta_2/2}\leqslant q^{-\eta}$ if 
\begin{equation}\label{Provided3Cusp}
\delta_2\geqslant 2\eta.
\end{equation}
\item[$(c)$] Suppose that $\frac{1}{2}-\delta_2<\mu\leqslant \frac{1}{2}+\delta_3.$ In this configuration, we apply the third bound and we obtain
$$\Ccal_2^\pm (f,N,M;q)\ll_\varepsilon q^{\varepsilon-\frac{1}{4}+\frac{\nu}{3}+\frac{\mu}{12}}=q^{\varepsilon-\frac{1}{4}+\frac{1}{12}(\mu+\nu)+\frac{\nu}{4}}.$$
Using \eqref{AssumeCusp1} and $\nu\leqslant 1+2\eta-\mu\leqslant \frac{1}{2}+2\eta+\delta_2$ allows us to bound the above expression by
$$q^{\varepsilon -\frac{1}{4}+\frac{1}{12}(1+2\eta)+\frac{1}{4}(\frac{1}{2}+2\eta+\delta_2)}=q^{\varepsilon -\frac{1}{12}(\frac{1}{2}-8\eta-3\delta_2)}\leqslant q^{\varepsilon-\eta},$$
provided
\begin{equation}\label{Provided4Cusp}
3\delta_2\leqslant \frac{1}{2}-20\eta.
\end{equation}
\item[$(d)$] Assume that $\mu>\frac{1}{2}+\delta_3$ the second bound gives 
$$\Ccal_2^\pm(f,N,M;q) \ll_\varepsilon q^\varepsilon\left(q^{\eta-\frac{1}{2}}+q^{\eta+\frac{1}{2}-\mu}\right)\leqslant q^{\varepsilon-\eta}+q^{\varepsilon+\eta-\delta_3}\ll q^{\varepsilon-\eta},$$
if we assume that
$$
\delta_3\geqslant 2\eta.
$$
\end{enumerate}
Finally, the combination of conditions \eqref{Provided1Cusp} and \eqref{Provided2Cusp} forces $\eta\leqslant \frac{1}{48}$ and \eqref{Provided3Cusp}-\eqref{Provided4Cusp} are simultaneously satisfied as long as $\eta\leqslant\frac{1}{52}$, which gives the correct exponent of the error term in Theorem \ref{Theorem1}.

\begin{remq} The treatment carried out in Section \ref{SectionCuspidalecase} remains almost identical if $f$ is level $1$ Hecke cusp form. The only change we have to make is to replace the exponent $1/16$ by $1/8$ in the fourth bound of Proposition \ref{FinalPropositionCuspidal}, which is due to the original Theorem \cite[Theorem 1.2]{twists} for small level compared with $q$. However, it does not improve the final exponent $\frac{1}{52}$ since \eqref{Provided3Cusp}-\eqref{Provided4Cusp} is more restrictive and independent of \eqref{Provided1Cusp}-\eqref{Provided2Cusp}.
\end{remq}

\bibliography{ShiftedCubicMoment}
%\addcontentsline{toc}{section}{References}
\bibliographystyle{alpha}

\end{document}